\documentclass[aop, preprint]{imsart}

\RequirePackage{amsthm,amsmath,amsfonts,amssymb}
\RequirePackage[numbers,sort&compress]{natbib}
\RequirePackage[colorlinks,citecolor=blue,urlcolor=blue]{hyperref}
\RequirePackage{graphicx}
\usepackage{adjustbox}
\usepackage{mathrsfs}
\usepackage{mathtools}
\usepackage{epsfig}  		
\usepackage{epic,eepic}       
\usepackage{tikz}
\DeclareUnicodeCharacter{2212}{-}

\startlocaldefs
\theoremstyle{plain}

\newtheorem{theorem}{Theorem}[section]
\newtheorem{lemma}[theorem]{lemma}
\newtheorem{proposition}[theorem]{Proposition}
\newtheorem{conjecture}[theorem]{Conjecture}
\theoremstyle{remark}

\newtheorem{remark}[theorem]{Remark}

\newtheorem{cor}[theorem]{Corollary}

\newcommand{\E}[1]{\mathbb{E}\left\{#1\right\}}
  
\renewcommand{\P}[1]{\mathbb{P}\left\{#1\right\}}

\newcommand{\ind}[1]{\mathbf{1}_{[#1]}}
\endlocaldefs

\begin{document}

\begin{frontmatter}
\title{The Fyodorov--Hiary--Keating Conjecture\\ on Mesoscopic Intervals}
\runtitle{The Fyodorov--Hiary--Keating Conjecture on Mesoscopic Intervals}

\begin{aug}
\author[A,B]{\fnms{Louis-Pierre}~\snm{Arguin}\ead[label=e1]{arguin@maths.ox.ac.uk}},
\author[A]{\fnms{Jad}~\snm{Hamdan}\ead[label=e2]{hamdan@maths.ox.ac.uk}}
\address[A]{Mathematical Institute,
University of Oxford\printead[presep={,\ }]{e1,e2}}
\address[B]{Baruch College and Graduate Center,
City University of New York}
\end{aug}

\begin{abstract}
We derive precise upper bounds for the maximum of the Riemann zeta function on a typical short interval of the critical line. We show that for fixed $\theta\in(-1,0]$, large $T$, and $y\geq 2$ satisfying $y=O(\log\log T/\log\log\log T)$, the proportion of points $t\in [T,2T]$ for which
\begin{align*}
    \max_{|h|\leq \log^\theta T}\big|\zeta(&\tfrac{1}{2}+it+ih)\big|>e^{y} \cdot e^{S\sqrt{(\log\log T)|\theta|/2}}\frac{(\log T)^{(1+\theta)}}{(\log\log T)^{3/4}}
\end{align*}
 is bounded above by a constant times $y\exp({-2y-y^2/((1+\theta)\log\log T)})$, where $S=S(t)$ is a quantity whose value distribution is approximately that of a standard Gaussian. Up to a multiplicative constant, this settles the upper bound of a conjecture of Fyodorov--Hiary--Keating which was only known in the leading order for $\theta\in(-1,0)$.
 
 Using similar techniques, we also derive upper bounds for the second moment of the zeta function on such intervals. We show that for large $T$, the proportion of $t\in [T,2T]$ for which
 \begin{align*}
    \frac{1}{\log^\theta T}\int_{-\log^\theta T}^{\log^\theta T} \big|\zeta(&\tfrac{1}{2}+it+ih)\big|^2\mathrm{d}h > A e^{S\sqrt{2|\theta|\log\log T}} \frac{(\log T)^{(1+\theta)}}{\sqrt{\log\log T}}
 \end{align*}
 tends to zero as $A\to\infty$, for the same $S$ as above.
 This proves a weak form of another conjecture of Fyodorov--Keating and generalizes a result of Harper, which is recovered at $\theta = 0$ (in which case $S$ is defined to be zero). Our proofs use an adaptation of the recursive scheme introduced by one of the authors, Bourgade and Radziwiłł in \cite{FHK1}.
\end{abstract}



\end{frontmatter}
\setcounter{tocdepth}{2}
\tableofcontents

\section{Introduction}\label{intro}

\subsection{Maxima of the Riemann zeta function in short intervals}

The value distribution of the Riemann zeta function $\zeta(s)$ on the line $\Re(s)=1/2$ is of fundamental importance in number theory. It is tightly linked to the distribution of its zeros, which encode the distribution of prime numbers via the explicit formula (see Theorem 5.1 in \cite{KoukoulopoulosBook}). A classical result in this vein is Selberg's central limit theorem \cite{selberg, selberg2}, which states that at a typical point $s=1/2+it$, $\log |\zeta(s)|$ behaves like a Gaussian random variable for large $t$. More precisely, if $\tau$ is uniformly distributed on $[T,2T]$, then
\begin{equation}
\label{eqn: selberg}
    \P{\log|\zeta(1/2+i\tau)|>\sqrt{\tfrac{1}{2}\log\log T}\cdot y} \sim \int_{y}^\infty \frac{e^{-z^2/2}}{\sqrt{2\pi}}\mathrm{d}z,\quad y\in \mathbb{R}
\end{equation}
as $T\to\infty$, suggesting a typical value of the order of $\sqrt{\log\log T}$ on such dyadic  intervals. This Gaussian tail was also recently shown to persist in the large deviations regime \cite{SoundMoments, arguinbailey1, arguinbailey2}.


 As was first noticed by Bourgade \cite{bourgademeso}, logarithmic correlations between $\log|\zeta(1/2+i\tau+ih)|$ and $\log|\zeta(1/2+i\tau+ih')|$ begin to appear when $h$ and $h'$ are sufficiently close. They can play a significant role when studying the distribution of $\log|\zeta|$ on typical \textit{short} intervals of the critical line, meaning either those of constant length, or \textit{mesoscopic} intervals of length proportional to $(\log T)^\theta$ for $\theta\in (-1,0)$. For instance, Fyodorov-Hiary-Keating \cite{fyodorovhiarykeating} and Fyodorov-Keating \cite{fyodorovkeating} made the following striking conjecture about the distribution of local maxima of $\log|\zeta|$ on such intervals.
 \begin{conjecture}[\cite{fyodorovkeating}, Section 2(c)(ii)]\label{conj}
 Let $\theta\in (-1,0]$. Then
\begin{align*}
    &\max_{|h|\leq \log^\theta T} \log\big|\zeta(\tfrac{1}{2}+i\tau+ih)\big| \\
    &\sim (1+\theta)\log\log T-\frac{3}{4}\log\big((1+\theta)\log\log T\big)+\sqrt{\tfrac{|\theta|}{2}\log\log T}\cdot \mathcal{N}(0,1)+\mathcal{M},
\end{align*}
where $\mathcal{N}(0,1)$ denotes a standard Gaussian random variable and $\mathcal{M}$ a random variable that is $O(1)$ in probability, independent of $\mathcal{N}(0,1)$, and satisfies the asymptotic $\mathbb{P}(\mathcal{M}>y)\sim Cye^{-2y}$ for some constant $C>0$.
 \end{conjecture}
\noindent This prediction is based on an analogous conjecture for the maximum of (the modulus of) the characteristic polynomial of a random unitary matrix, following the Keating-Snaith paradigm \cite{jonnina}. In that case, Fyodorov--Keating give the exact distribution for $\mathcal{M}$, which was recognized in \cite{subagzeitouni} to be that of a sum of two independent Gumbel random variables. 

    Even at $\theta=0$, the logarithmic correlations appear through the coefficient $3/4$ in the second-order term. Indeed, one can argue (e.g. using the functional equation for the zeta function) that the analysis of $\max_{|h|\leq 1}\log |\zeta(1/2+i\tau+ih)|$ can be reduced to a discrete set containing roughly $(\log T)$ points, each of which is approximately Gaussian with variance $\tfrac{1}{2}\log\log T$ by Selberg's theorem. If one were to model the maximum by that of $(\log T)$ independent Gaussian random variables of the same variance, a straightforward calculation yields a sub-leading order of $-(1/4)\log\log\log T$ (see, e.g., \cite{harperFHK}); a larger maximum, as one would expect by Slepian's lemma. 

\bigskip

Conjecture \ref{conj} has seen significant progress since its formulation, particularly in the $\theta=0$ case. Following initial progress by Najnudel \cite{naj}, Arguin, Bourgade, Belius, Radziwiłł and Soundararajan \cite{ABBRS2016}, and Harper \cite{harper2019}, sharp upper and lower bounds were established by Arguin, Bourgade and Radziwiłł \cite{FHK1, FHK2} using ideas from the study of branching random walks. On the random-matrix-theoretic side, the leading order was established by Arguin, Belius and Bourgade \cite{arguinbourgadebelius}, the subleading order by Paquette and Zeitouni \cite{paquettezeitouni1} and the tightness of the recentered maximum by Chhaibi, Madaule and Najnudel \cite{CMN}. Finally, a recent work of Paquette and Zeitouni \cite{paquettezeitouni2} proved the distributional convergence of the maximum to a sum of a Gumbel random variable and a ``derivative martingale''; it remains open whether the latter is itself Gumbel-distributed.

For $\theta<0$, the maximum is only known in the leading order, following work by one of the authors, Ouimet, and Radziwiłł \cite{AOR}. Our first result improves their upper bound to $O(1)$ precision.
\begin{theorem}[Upper bound for the local maximum]\label{UB}
    Let $C>0$ be an arbitrary constant and $\theta\in(-1,0]$. Then uniformly in $2\leq  y< {C\log\log T/\log\log\log T}$ and large enough $T$,
    \begin{align}
    \begin{split}
           \mathbb{P}\bigg\{\max_{|h|\leq \log^\theta T} \big|\zeta\big(\tfrac{1}{2}+i(\tau+h)\big) P_\tau(\theta)\big|&>e^{y}\frac{(\log T)^{(1+\theta)}}{(\log\log T)^{3/4}}\bigg\} \leq C' ye^{-2y-\tfrac{y^2}{(1+\theta)\log\log T}}
    \end{split}
    \end{align}
    for some constant $C'>0$, where $P_\tau(\theta)=\prod_{p\leq \exp((\log T)^{|\theta|})}(1-p^{-1/2-i\tau})$.
\end{theorem}
\noindent It follows by a standard moment computation (see, e.g., the proof of Selberg's theorem in \cite{radziwillsound}) that $\log|P_\tau(\theta)|/\sqrt{({|\theta|}/{2})\log\log T}$ converges in distribution to a standard Gaussian random variable as $T\to\infty$. Theorem \ref{UB} thus settles a strong form of the upper bound in Conjecture \ref{conj}, and explicitly identifies the source of the Gaussian term therein.

While the random-matrix-theoretic analogue of this conjecture has yet to be studied explicitly for $\theta<0$, we expect an analogue of Theorem \ref{UB} to be provable using the tools in \cite{paquettezeitouni2}. In light of a classical result of Diaconis and Evans \cite{diaconisevans}, this would likely take the form
\[ 
    \P{\max_{|h|\leq N^\theta} \big|\mathrm{P}_N(e^{ih})X_\theta\big| > e^{y}\frac{N^{(1+\theta)}}{(\log N)^{3/4}}} \leq Cye^{-2y-y^2/((1+\theta)\log N)}.
\]
where $\mathrm{P}_N(z)$ is the characteristic polynomial of an $N\times N$ Haar-unitary matrix $U_N$ and
\[
    X_\theta:=\exp\Big(\sum_{j=1}^{N^{|\theta|}}\frac{\text{Tr}(U_N^j)}{j}\Big).
\]
(See the introduction of \cite{arguinbourgadebelius} for more detail). 

\subsection{The partition function of the zeta function}

The \textit{moments} of $\zeta$ in short intervals 
\begin{equation}\label{shortmoment}
       \mathcal{Z}_\beta{(\theta)}:=\frac{1}{2\log ^\theta T} \int_{|h|\leq \log^\theta T} |\zeta(\tfrac{1}{2}+i(\tau+h))|^\beta\mathrm{d}h
\end{equation}
are closely related to the maximum. It is useful to think of $\mathcal{Z}_\beta{(\theta)}$ as a (random) {\it partition function} at inverse temperature $\beta$, associated to a system whose states are indexed by the interval of radius $\log^\theta T$ around $\tau$, and with energy $\log |\zeta(\tfrac{1}{2}+i(\tau+h))|$ at state $h$.

Upper and lower bounds for the leading order of $\mathcal{Z}_\beta{(\theta)}$ were established in \cite{AOR} for all $\theta>-1$, with the former holding unconditionally, and the latter assuming the Riemann hypothesis when $\theta> 3$. In all cases, one finds that $\mathcal{Z}_\beta (\theta)\asymp (\log T)^{f_\theta(\beta)}$ for some $f_\theta(\beta)$ with high probability; the same work also shows that a \textit{freezing transition} (conjectured in \cite{fyodorovkeating}) occurs at the critical parameter $\beta_c=2\sqrt{1+\min{(\theta,0)}}$, where $f_\theta(\beta)$ transitions from being quadratic in $\beta$ to linear. These bounds were later improved in \cite{arguinbailey1} for $\theta \in[0,3)$ and $\beta>2\sqrt{1+\theta}$, by studying the measure of high points of $\log |\zeta|$.

In the case $\theta=0, \beta=\beta_c=2$, Harper \cite{harper2019} proved the precise estimate
\begin{equation}\label{harper}
    \P{\mathcal{Z}_2(0) \geq A\frac{\log T}{\sqrt{\log \log T}}} \leq C\cdot\frac{\min\{\log A, \sqrt{\log\log T}\}}{A},
\end{equation}
which implies that $\mathcal{Z}_2(0)= O(\log T/\sqrt{\log\log T})$ with high probability. This beats the trivial bound of $O(\log T)$, which follows  from Markov's inequality and the classical second-moment estimate in Lemma \ref{secondmomentzeta}.
The additional $\sqrt{\log\log T}$ in the denominator was shown to stem from a connection to critical Gaussian multiplicative chaos \cite{PowellSurvey}, and is analogous to the so-called Seneta-Heyde normalization in that setting.

For $\theta<0$, it was anticipated in \cite[Section 4(b)]{fyodorovkeating} that $\mathcal{Z}_2{(\theta)}$ should asymptotically split into the product of two uncorrelated random variables: the first being roughly log-normal, and the second behaving similarly to $\mathcal{Z}_2{(0)}$ upon substituting $T$ by $\exp((\log T)^{1+\theta})$. Towards this conjecture, our main result establishes an upper bound for this partition function in mesoscopic intervals. It features an analogue of the log-normal factor in \cite[Section 4(b)]{fyodorovkeating} (recalling that $\log|P_\tau(\theta)|/\sqrt{({|\theta|}/{2})\log\log T}$ is asymptotically Gaussian), as well as a square root correction term as in Harper's result (which is recovered at $\theta=0$).

\begin{theorem}[Upper bound for the partition function at criticality]\label{moments}Fix $\theta\in(-1,0]$ and for any $T$, let $P_\tau(\theta)$ be as in Theorem \ref{UB}. Define
    \[
        \mathcal{Z}_2{(\theta)}:=\frac{1}{2\log^\theta T}\int_{|h|\leq \log^\theta T} \big|\zeta(\tfrac{1}{2}+i(\tau+h)) \big|^2\mathrm{d}h.
    \]
    Then if $\theta\in (-1,0)$, for any $\varepsilon>0$, there exists a constant $C=C(\varepsilon)>0$ such that 
    \[
        \P{\mathcal{Z}_2(\theta)|P_\tau(\theta)|^2 \geq  A \frac{(\log T)^{(1+\theta)}}{\sqrt{(1+\theta)\log\log T}}} \leq CA^{-1+\varepsilon}
    \]
    uniformly in large enough $T$ and $A\leq\sqrt{\log\log T}$. For $\theta=0$, there exists a $C>0$ such that
    \[
                \P{\mathcal{Z}_2(0) \geq A\frac{\log T}{\sqrt{\log\log T}} } \leq C\frac{\min\{\log A,\sqrt{\log\log T}\}}{A}
    \]
    uniformly in large $T$ and $A\geq 2$.
\end{theorem}
\noindent Our proof differs from the approach in \cite{harper2019} and yields a new proof of Harper's bound at $\theta=0$. We express $\mathcal{Z}_2(0)$ in terms of the typical measure of level sets of $\log |\zeta|$, similarly to the argument in Section 4 of \cite{arguinbailey1} except that theirs could not be made rigorous. In order to do so, we employ a novel ``partial barrier'' strategy which is explained in Section \ref{outlines}.

\subsection{Related results} In \cite{arguinbailey1}, the authors also studied the maximum over longer intervals, of length $(\log T)^\theta$ for $\theta>0$. Comparing the subleading order and tail of the maximum in that setting to Conjecture \ref{conj}, one finds a discontinuity as $\theta \downarrow 0$, representing a transition from log-correlated to i.i.d.\ statistics. By analyzing a random model for the zeta function, Dubach, Hartung, and one of the authors \cite{ADH} showed that the subleading order discontinuity can be resolved by letting $\theta$ approach zero at an appropriate rate, revealing an intermediate regime. Chang \cite{chang} then resolved the tail discontinuity for the same model in that regime. Similar phenomena have also been studied in the settings of branching Brownian motion, branching random walks, and random energy models \cite{KistlerSchmidt, BovierHartung}. All of these works are also related to a broader literature on log-correlated fields and Gaussian multiplicative chaos in connection with the Fyodorov-Hiary-Keating conjectures \cite{JLW, NSW, SW, remy}, recently surveyed by Bailey and Keating  \cite{BaileyKeating}.

\bigskip

\textsc{Notation.} Throughout our arguments, $\tau$ will always denote a random variable that is uniformly distributed in $[T,2T]$. For simplicity, we will use the notation 
$$\zeta_\tau(h):=\zeta(\tfrac{1}{2}+i(\tau+h)).$$ 
As was done above, we will let \[
t=\log\log T \text{ and } t_\theta=t|\theta|,
\]
{noting that in this notation
$$t(1+\theta)=t-t_\theta.$$
For any $\ell\geq 1$,  $\log_\ell t$ will denote the logarithm iterated $\ell$ times.}
We will use standard asymptotic notation, writing $f(T)=o(g(T))$ if $|f(T)/g(T)|$ tends to $0$ as ${T\to\infty}$, and $f(T)=O(g(T))$ or $f(T)\ll g(T)$ if $\limsup|f(T)/g(T)|$ is bounded. A subscripted parameter next to $o, O$ or $\ll$ means that the implicit constant is allowed to depend on said parameter. 
Lastly, we will use standard number theoretic notation for the following arithmetical functions: $\omega(n)$ is the number of distinct primes dividing $n$, $\Omega(n)$ the number of such primes counted with multiplicity, and $\mu$ is the Möbius function, which equals $-1$ (resp. $1$) on squarefree integers with an odd (resp. even) number of prime factors, and zero otherwise.

\subsection{Branching random walk model and proof outlines}\label{outlines} Our arguments are all centered around a branching random walk model for the zeta function originating in \cite{arguinbeliusharper}, which we now briefly recall (for more detail, see \cite{AOR, ABBRS2016}). For $k\leq t$, define the partial sums
\begin{equation}\label{Sdefinition}
        S_{k}(h)=\sum_{C<p\leq \exp(e^{k})} \Re\left(\frac{1}{p^{1/2+i(\tau+h)}}+\frac{1}{2}\frac{1}{p^{1+2i(\tau+h)}}\right).
\end{equation}
Each summand corresponds to a second-order Taylor approximation to $-\log |1-p^{-\tfrac{1}{2}-i\tau-ih}|$, and we view $S_k(h)$ as a proxy for $\log|\zeta(1/2+i\tau+ih)|$ when $k\approx t$. (The sum is started at a sufficiently large constant, in this case $C=\exp(e^{1000})$, for technical reasons.) We are then led to study the process $(S_{k}(h), |h|\leq \log^\theta T)$,
which by a standard estimate (see, e.g., Lemma 3.4 in \cite{AOR}), has covariances
\begin{align}\label{eq:correlations}
    \E{S_k(h)S_k(h')}= \frac{1}{2}\sum_{ p\leq \exp(e^k)} \frac{\cos(|h-h'|\log p)}{p}+O\left(\frac{\exp(2e^{k})}{T}\right).
\end{align}
For each $h$, we view $S_t(h)$ as an approximate random walk with $t$ steps: ignoring the fact that the error term in \eqref{eq:correlations} is large for $k\approx t$, the double-exponential indexing of the sum therein is chosen so as to make the increments $(S_k-S_{k-1})(h)$ have variance $$\frac{1}{2}\sum_{\exp(e^{k-1})<p\leq \exp(e^k)}\frac{1}{p}\approx \frac{1}{2}$$ by the prime number theorem. A similar calculation shows that for any $h,h'$, the increments of $S_t(h)$ and $S_t(h')$ are highly correlated up to $k\approx \log |h-h'|^{-1}$, after which they decorrelate rapidly. At intermediate times $k\leq t$, this implies that $S_k(h)\approx S_k(h')$ when $|h-h'|\leq e^{-k}$, meaning that there are roughly $e^k$ distinct ``random walks'' over $|h|\leq 1$.

This leads one to compare $(S_t(h),|h|\leq 1)$ to a branching random walk on a tree of depth $t$ in which each node has $e=2.718\dots$ offspring on average: to each edge, one attaches a Gaussian random variable with variance $1/2$, and to each leaf, the sum of the edge weights from that leaf to the root. The leaves then correspond to the values of $S_t(h)$ on a mesh of width $1/\log T$.

\bigskip

\begin{figure}
    \centering

\tikzset{every picture/.style={line width=0.75pt}} 

\begin{tikzpicture}[x=0.75pt,y=0.75pt,yscale=-1,xscale=1]

\draw    (460.39,1838.94) -- (466.7,1888.47) ;
\draw [color={rgb, 255:red, 74; green, 144; blue, 226 }  ,draw opacity=0.5 ][fill={rgb, 255:red, 74; green, 144; blue, 226 }  ,fill opacity=1 ][line width=3]    (428.83,1838.94) -- (435.15,1888.47) ;
\draw [color={rgb, 255:red, 0; green, 0; blue, 0 }  ,draw opacity=1 ]   (428.83,1838.94) -- (420.94,1888.47) ;
\draw [color={rgb, 255:red, 74; green, 144; blue, 226 }  ,draw opacity=0.5 ][fill={rgb, 255:red, 0; green, 0; blue, 0 }  ,fill opacity=1 ][line width=3]    (444.61,1789.4) -- (428.83,1838.94) ;
\draw [color={rgb, 255:red, 74; green, 144; blue, 226 }  ,draw opacity=0.5 ][line width=3]    (444.61,1789.4) -- (460.39,1838.94) ;
\draw [color={rgb, 255:red, 208; green, 2; blue, 27 }  ,draw opacity=0.5 ][line width=3]    (413.09,1740.19) -- (444.61,1789.4) ;
\draw [color={rgb, 255:red, 208; green, 2; blue, 27 }  ,draw opacity=0.5 ][line width=3]    (476.2,1690.65) -- (413.09,1740.19) ;
\draw [color={rgb, 255:red, 155; green, 155; blue, 155 }  ,draw opacity=1 ][fill={rgb, 255:red, 155; green, 155; blue, 155 }  ,fill opacity=1 ]   (476.2,1690.65) -- (539.32,1740.19) ;
\draw [color={rgb, 255:red, 155; green, 155; blue, 155 }  ,draw opacity=1 ][fill={rgb, 255:red, 155; green, 155; blue, 155 }  ,fill opacity=1 ]   (413.09,1740.19) -- (381.5,1789.4) ;
\draw  [color={rgb, 255:red, 0; green, 0; blue, 0 }  ,draw opacity=1 ][fill={rgb, 255:red, 0; green, 0; blue, 0 }  ,fill opacity=1 ] (474.23,1690.65) .. controls (474.23,1689.51) and (475.11,1688.59) .. (476.2,1688.59) .. controls (477.29,1688.59) and (478.18,1689.51) .. (478.18,1690.65) .. controls (478.18,1691.79) and (477.29,1692.72) .. (476.2,1692.72) .. controls (475.11,1692.72) and (474.23,1691.79) .. (474.23,1690.65) -- cycle ;
\draw    (342.39,1888.47) -- (603.89,1888.47) ;
\draw [shift={(605.89,1888.47)}, rotate = 180] [color={rgb, 255:red, 0; green, 0; blue, 0 }  ][line width=0.75]    (4.37,-1.32) .. controls (2.78,-0.56) and (1.32,-0.12) .. (0,0) .. controls (1.32,0.12) and (2.78,0.56) .. (4.37,1.32)   ;
\draw  [color={rgb, 255:red, 0; green, 0; blue, 0 }  ,draw opacity=1 ][fill={rgb, 255:red, 0; green, 0; blue, 0 }  ,fill opacity=1 ] (411.12,1740.19) .. controls (411.12,1739.05) and (412,1738.12) .. (413.09,1738.12) .. controls (414.18,1738.12) and (415.06,1739.05) .. (415.06,1740.19) .. controls (415.06,1741.33) and (414.18,1742.25) .. (413.09,1742.25) .. controls (412,1742.25) and (411.12,1741.33) .. (411.12,1740.19) -- cycle ;
\draw  [color={rgb, 255:red, 155; green, 155; blue, 155 }  ,draw opacity=1 ][fill={rgb, 255:red, 155; green, 155; blue, 155 }  ,fill opacity=1 ] (537.35,1740.19) .. controls (537.35,1739.05) and (538.23,1738.12) .. (539.32,1738.12) .. controls (540.41,1738.12) and (541.29,1739.05) .. (541.29,1740.19) .. controls (541.29,1741.33) and (540.41,1742.25) .. (539.32,1742.25) .. controls (538.23,1742.25) and (537.35,1741.33) .. (537.35,1740.19) -- cycle ;
\draw  [color={rgb, 255:red, 155; green, 155; blue, 155 }  ,draw opacity=1 ][fill={rgb, 255:red, 155; green, 155; blue, 155 }  ,fill opacity=1 ] (379.53,1789.4) .. controls (379.53,1788.26) and (380.41,1787.34) .. (381.5,1787.34) .. controls (382.59,1787.34) and (383.47,1788.26) .. (383.47,1789.4) .. controls (383.47,1790.54) and (382.59,1791.47) .. (381.5,1791.47) .. controls (380.41,1791.47) and (379.53,1790.54) .. (379.53,1789.4) -- cycle ;
\draw  [color={rgb, 255:red, 74; green, 144; blue, 226 }  ,draw opacity=1 ][fill={rgb, 255:red, 74; green, 144; blue, 226 }  ,fill opacity=1 ] (442.64,1789.4) .. controls (442.64,1788.26) and (443.52,1787.34) .. (444.61,1787.34) .. controls (445.7,1787.34) and (446.59,1788.26) .. (446.59,1789.4) .. controls (446.59,1790.54) and (445.7,1791.47) .. (444.61,1791.47) .. controls (443.52,1791.47) and (442.64,1790.54) .. (442.64,1789.4) -- cycle ;
\draw  [color={rgb, 255:red, 155; green, 155; blue, 155 }  ,draw opacity=1 ][fill={rgb, 255:red, 155; green, 155; blue, 155 }  ,fill opacity=1 ] (505.79,1789.72) .. controls (505.79,1788.58) and (506.67,1787.66) .. (507.76,1787.66) .. controls (508.85,1787.66) and (509.73,1788.58) .. (509.73,1789.72) .. controls (509.73,1790.86) and (508.85,1791.79) .. (507.76,1791.79) .. controls (506.67,1791.79) and (505.79,1790.86) .. (505.79,1789.72) -- cycle ;
\draw  [color={rgb, 255:red, 155; green, 155; blue, 155 }  ,draw opacity=1 ][fill={rgb, 255:red, 155; green, 155; blue, 155 }  ,fill opacity=1 ] (568.9,1789.72) .. controls (568.9,1788.58) and (569.79,1787.66) .. (570.88,1787.66) .. controls (571.97,1787.66) and (572.85,1788.58) .. (572.85,1789.72) .. controls (572.85,1790.86) and (571.97,1791.79) .. (570.88,1791.79) .. controls (569.79,1791.79) and (568.9,1790.86) .. (568.9,1789.72) -- cycle ;
\draw  [color={rgb, 255:red, 155; green, 155; blue, 155 }  ,draw opacity=1 ][fill={rgb, 255:red, 155; green, 155; blue, 155 }  ,fill opacity=1 ] (363.75,1838.94) .. controls (363.75,1837.8) and (364.63,1836.87) .. (365.72,1836.87) .. controls (366.81,1836.87) and (367.69,1837.8) .. (367.69,1838.94) .. controls (367.69,1840.08) and (366.81,1841) .. (365.72,1841) .. controls (364.63,1841) and (363.75,1840.08) .. (363.75,1838.94) -- cycle ;
\draw  [color={rgb, 255:red, 155; green, 155; blue, 155 }  ,draw opacity=1 ][fill={rgb, 255:red, 155; green, 155; blue, 155 }  ,fill opacity=1 ] (395.3,1838.94) .. controls (395.3,1837.8) and (396.19,1836.87) .. (397.28,1836.87) .. controls (398.37,1836.87) and (399.25,1837.8) .. (399.25,1838.94) .. controls (399.25,1840.08) and (398.37,1841) .. (397.28,1841) .. controls (396.19,1841) and (395.3,1840.08) .. (395.3,1838.94) -- cycle ;
\draw  [color={rgb, 255:red, 0; green, 0; blue, 0 }  ,draw opacity=1 ][fill={rgb, 255:red, 0; green, 0; blue, 0 }  ,fill opacity=1 ] (426.86,1838.94) .. controls (426.86,1837.8) and (427.74,1836.87) .. (428.83,1836.87) .. controls (429.92,1836.87) and (430.81,1837.8) .. (430.81,1838.94) .. controls (430.81,1840.08) and (429.92,1841) .. (428.83,1841) .. controls (427.74,1841) and (426.86,1840.08) .. (426.86,1838.94) -- cycle ;
\draw  [color={rgb, 255:red, 74; green, 144; blue, 226 }  ,draw opacity=1 ][fill={rgb, 255:red, 74; green, 144; blue, 226 }  ,fill opacity=1 ] (458.42,1838.94) .. controls (458.42,1837.8) and (459.3,1836.87) .. (460.39,1836.87) .. controls (461.48,1836.87) and (462.36,1837.8) .. (462.36,1838.94) .. controls (462.36,1840.08) and (461.48,1841) .. (460.39,1841) .. controls (459.3,1841) and (458.42,1840.08) .. (458.42,1838.94) -- cycle ;
\draw  [color={rgb, 255:red, 155; green, 155; blue, 155 }  ,draw opacity=1 ][fill={rgb, 255:red, 155; green, 155; blue, 155 }  ,fill opacity=1 ] (355.86,1888.47) .. controls (355.86,1887.33) and (356.74,1886.41) .. (357.83,1886.41) .. controls (358.92,1886.41) and (359.8,1887.33) .. (359.8,1888.47) .. controls (359.8,1889.61) and (358.92,1890.54) .. (357.83,1890.54) .. controls (356.74,1890.54) and (355.86,1889.61) .. (355.86,1888.47) -- cycle ;
\draw  [color={rgb, 255:red, 155; green, 155; blue, 155 }  ,draw opacity=1 ][fill={rgb, 255:red, 155; green, 155; blue, 155 }  ,fill opacity=1 ] (370.06,1888.47) .. controls (370.06,1887.33) and (370.94,1886.41) .. (372.03,1886.41) .. controls (373.12,1886.41) and (374,1887.33) .. (374,1888.47) .. controls (374,1889.61) and (373.12,1890.54) .. (372.03,1890.54) .. controls (370.94,1890.54) and (370.06,1889.61) .. (370.06,1888.47) -- cycle ;
\draw  [color={rgb, 255:red, 155; green, 155; blue, 155 }  ,draw opacity=1 ][fill={rgb, 255:red, 155; green, 155; blue, 155 }  ,fill opacity=1 ] (387.41,1888.47) .. controls (387.41,1887.33) and (388.3,1886.41) .. (389.39,1886.41) .. controls (390.48,1886.41) and (391.36,1887.33) .. (391.36,1888.47) .. controls (391.36,1889.61) and (390.48,1890.54) .. (389.39,1890.54) .. controls (388.3,1890.54) and (387.41,1889.61) .. (387.41,1888.47) -- cycle ;
\draw  [color={rgb, 255:red, 155; green, 155; blue, 155 }  ,draw opacity=1 ][fill={rgb, 255:red, 155; green, 155; blue, 155 }  ,fill opacity=1 ] (401.62,1888.47) .. controls (401.62,1887.33) and (402.5,1886.41) .. (403.59,1886.41) .. controls (404.68,1886.41) and (405.56,1887.33) .. (405.56,1888.47) .. controls (405.56,1889.61) and (404.68,1890.54) .. (403.59,1890.54) .. controls (402.5,1890.54) and (401.62,1889.61) .. (401.62,1888.47) -- cycle ;
\draw  [color={rgb, 255:red, 0; green, 0; blue, 0 }  ,draw opacity=1 ][fill={rgb, 255:red, 0; green, 0; blue, 0 }  ,fill opacity=1 ] (418.97,1888.47) .. controls (418.97,1887.33) and (419.86,1886.41) .. (420.94,1886.41) .. controls (422.03,1886.41) and (422.92,1887.33) .. (422.92,1888.47) .. controls (422.92,1889.61) and (422.03,1890.54) .. (420.94,1890.54) .. controls (419.86,1890.54) and (418.97,1889.61) .. (418.97,1888.47) -- cycle ;
\draw  [color={rgb, 255:red, 74; green, 144; blue, 226 }  ,draw opacity=1 ][fill={rgb, 255:red, 74; green, 144; blue, 226 }  ,fill opacity=1 ] (433.17,1888.47) .. controls (433.17,1887.33) and (434.06,1886.41) .. (435.15,1886.41) .. controls (436.23,1886.41) and (437.12,1887.33) .. (437.12,1888.47) .. controls (437.12,1889.61) and (436.23,1890.54) .. (435.15,1890.54) .. controls (434.06,1890.54) and (433.17,1889.61) .. (433.17,1888.47) -- cycle ;
\draw  [color={rgb, 255:red, 74; green, 144; blue, 226 }  ,draw opacity=1 ][fill={rgb, 255:red, 74; green, 144; blue, 226 }  ,fill opacity=1 ] (450.53,1888.47) .. controls (450.53,1887.33) and (451.41,1886.41) .. (452.5,1886.41) .. controls (453.59,1886.41) and (454.47,1887.33) .. (454.47,1888.47) .. controls (454.47,1889.61) and (453.59,1890.54) .. (452.5,1890.54) .. controls (451.41,1890.54) and (450.53,1889.61) .. (450.53,1888.47) -- cycle ;
\draw  [color={rgb, 255:red, 0; green, 0; blue, 0 }  ,draw opacity=1 ][fill={rgb, 255:red, 0; green, 0; blue, 0 }  ,fill opacity=1 ] (464.73,1888.47) .. controls (464.73,1887.33) and (465.61,1886.41) .. (466.7,1886.41) .. controls (467.79,1886.41) and (468.68,1887.33) .. (468.68,1888.47) .. controls (468.68,1889.61) and (467.79,1890.54) .. (466.7,1890.54) .. controls (465.61,1890.54) and (464.73,1889.61) .. (464.73,1888.47) -- cycle ;
\draw [color={rgb, 255:red, 155; green, 155; blue, 155 }  ,draw opacity=1 ][fill={rgb, 255:red, 155; green, 155; blue, 155 }  ,fill opacity=1 ]   (381.5,1789.4) -- (365.72,1838.94) ;
\draw [color={rgb, 255:red, 155; green, 155; blue, 155 }  ,draw opacity=1 ][fill={rgb, 255:red, 155; green, 155; blue, 155 }  ,fill opacity=1 ]   (365.72,1838.94) -- (357.83,1888.47) ;
\draw [color={rgb, 255:red, 155; green, 155; blue, 155 }  ,draw opacity=1 ][fill={rgb, 255:red, 155; green, 155; blue, 155 }  ,fill opacity=1 ]   (365.72,1838.94) -- (372.03,1888.47) ;
\draw [color={rgb, 255:red, 155; green, 155; blue, 155 }  ,draw opacity=1 ][fill={rgb, 255:red, 155; green, 155; blue, 155 }  ,fill opacity=1 ]   (397.28,1838.94) -- (403.59,1888.47) ;
\draw [color={rgb, 255:red, 155; green, 155; blue, 155 }  ,draw opacity=1 ][fill={rgb, 255:red, 155; green, 155; blue, 155 }  ,fill opacity=1 ]   (397.28,1838.94) -- (389.39,1888.47) ;
\draw  [color={rgb, 255:red, 155; green, 155; blue, 155 }  ,draw opacity=1 ][fill={rgb, 255:red, 155; green, 155; blue, 155 }  ,fill opacity=1 ] (489.98,1838.94) .. controls (489.98,1837.8) and (490.86,1836.87) .. (491.95,1836.87) .. controls (493.04,1836.87) and (493.92,1837.8) .. (493.92,1838.94) .. controls (493.92,1840.08) and (493.04,1841) .. (491.95,1841) .. controls (490.86,1841) and (489.98,1840.08) .. (489.98,1838.94) -- cycle ;
\draw  [color={rgb, 255:red, 155; green, 155; blue, 155 }  ,draw opacity=1 ][fill={rgb, 255:red, 155; green, 155; blue, 155 }  ,fill opacity=1 ] (521.53,1838.94) .. controls (521.53,1837.8) and (522.42,1836.87) .. (523.51,1836.87) .. controls (524.6,1836.87) and (525.48,1837.8) .. (525.48,1838.94) .. controls (525.48,1840.08) and (524.6,1841) .. (523.51,1841) .. controls (522.42,1841) and (521.53,1840.08) .. (521.53,1838.94) -- cycle ;
\draw  [color={rgb, 255:red, 155; green, 155; blue, 155 }  ,draw opacity=1 ][fill={rgb, 255:red, 155; green, 155; blue, 155 }  ,fill opacity=1 ] (553.09,1838.94) .. controls (553.09,1837.8) and (553.97,1836.87) .. (555.06,1836.87) .. controls (556.15,1836.87) and (557.04,1837.8) .. (557.04,1838.94) .. controls (557.04,1840.08) and (556.15,1841) .. (555.06,1841) .. controls (553.97,1841) and (553.09,1840.08) .. (553.09,1838.94) -- cycle ;
\draw  [color={rgb, 255:red, 155; green, 155; blue, 155 }  ,draw opacity=1 ][fill={rgb, 255:red, 155; green, 155; blue, 155 }  ,fill opacity=1 ] (584.65,1838.94) .. controls (584.65,1837.8) and (585.53,1836.87) .. (586.62,1836.87) .. controls (587.71,1836.87) and (588.59,1837.8) .. (588.59,1838.94) .. controls (588.59,1840.08) and (587.71,1841) .. (586.62,1841) .. controls (585.53,1841) and (584.65,1840.08) .. (584.65,1838.94) -- cycle ;
\draw [color={rgb, 255:red, 155; green, 155; blue, 155 }  ,draw opacity=1 ][fill={rgb, 255:red, 155; green, 155; blue, 155 }  ,fill opacity=1 ]   (507.76,1789.72) -- (491.95,1838.94) ;
\draw [color={rgb, 255:red, 155; green, 155; blue, 155 }  ,draw opacity=1 ][fill={rgb, 255:red, 155; green, 155; blue, 155 }  ,fill opacity=1 ]   (570.88,1789.72) -- (555.06,1838.94) ;
\draw [color={rgb, 255:red, 155; green, 155; blue, 155 }  ,draw opacity=1 ][fill={rgb, 255:red, 155; green, 155; blue, 155 }  ,fill opacity=1 ]   (381.5,1789.4) -- (397.28,1838.94) ;
\draw [color={rgb, 255:red, 155; green, 155; blue, 155 }  ,draw opacity=1 ][fill={rgb, 255:red, 155; green, 155; blue, 155 }  ,fill opacity=1 ]   (507.76,1789.72) -- (523.51,1838.94) ;
\draw [color={rgb, 255:red, 155; green, 155; blue, 155 }  ,draw opacity=1 ][fill={rgb, 255:red, 155; green, 155; blue, 155 }  ,fill opacity=1 ]   (570.88,1789.72) -- (586.62,1838.94) ;
\draw  [color={rgb, 255:red, 155; green, 155; blue, 155 }  ,draw opacity=1 ][fill={rgb, 255:red, 155; green, 155; blue, 155 }  ,fill opacity=1 ] (482.09,1888.47) .. controls (482.09,1887.33) and (482.97,1886.41) .. (484.06,1886.41) .. controls (485.15,1886.41) and (486.03,1887.33) .. (486.03,1888.47) .. controls (486.03,1889.61) and (485.15,1890.54) .. (484.06,1890.54) .. controls (482.97,1890.54) and (482.09,1889.61) .. (482.09,1888.47) -- cycle ;
\draw  [color={rgb, 255:red, 155; green, 155; blue, 155 }  ,draw opacity=1 ][fill={rgb, 255:red, 155; green, 155; blue, 155 }  ,fill opacity=1 ] (496.29,1888.47) .. controls (496.29,1887.33) and (497.17,1886.41) .. (498.26,1886.41) .. controls (499.35,1886.41) and (500.23,1887.33) .. (500.23,1888.47) .. controls (500.23,1889.61) and (499.35,1890.54) .. (498.26,1890.54) .. controls (497.17,1890.54) and (496.29,1889.61) .. (496.29,1888.47) -- cycle ;
\draw  [color={rgb, 255:red, 155; green, 155; blue, 155 }  ,draw opacity=1 ][fill={rgb, 255:red, 155; green, 155; blue, 155 }  ,fill opacity=1 ] (513.64,1888.47) .. controls (513.64,1887.33) and (514.53,1886.41) .. (515.62,1886.41) .. controls (516.71,1886.41) and (517.59,1887.33) .. (517.59,1888.47) .. controls (517.59,1889.61) and (516.71,1890.54) .. (515.62,1890.54) .. controls (514.53,1890.54) and (513.64,1889.61) .. (513.64,1888.47) -- cycle ;
\draw  [color={rgb, 255:red, 155; green, 155; blue, 155 }  ,draw opacity=1 ][fill={rgb, 255:red, 155; green, 155; blue, 155 }  ,fill opacity=1 ] (527.85,1888.47) .. controls (527.85,1887.33) and (528.73,1886.41) .. (529.82,1886.41) .. controls (530.91,1886.41) and (531.79,1887.33) .. (531.79,1888.47) .. controls (531.79,1889.61) and (530.91,1890.54) .. (529.82,1890.54) .. controls (528.73,1890.54) and (527.85,1889.61) .. (527.85,1888.47) -- cycle ;
\draw  [color={rgb, 255:red, 155; green, 155; blue, 155 }  ,draw opacity=1 ][fill={rgb, 255:red, 155; green, 155; blue, 155 }  ,fill opacity=1 ] (545.2,1888.47) .. controls (545.2,1887.33) and (546.09,1886.41) .. (547.17,1886.41) .. controls (548.26,1886.41) and (549.15,1887.33) .. (549.15,1888.47) .. controls (549.15,1889.61) and (548.26,1890.54) .. (547.17,1890.54) .. controls (546.09,1890.54) and (545.2,1889.61) .. (545.2,1888.47) -- cycle ;
\draw  [color={rgb, 255:red, 155; green, 155; blue, 155 }  ,draw opacity=1 ][fill={rgb, 255:red, 155; green, 155; blue, 155 }  ,fill opacity=1 ] (559.4,1888.47) .. controls (559.4,1887.33) and (560.29,1886.41) .. (561.38,1886.41) .. controls (562.47,1886.41) and (563.35,1887.33) .. (563.35,1888.47) .. controls (563.35,1889.61) and (562.47,1890.54) .. (561.38,1890.54) .. controls (560.29,1890.54) and (559.4,1889.61) .. (559.4,1888.47) -- cycle ;
\draw  [color={rgb, 255:red, 155; green, 155; blue, 155 }  ,draw opacity=1 ][fill={rgb, 255:red, 155; green, 155; blue, 155 }  ,fill opacity=1 ] (576.76,1888.47) .. controls (576.76,1887.33) and (577.64,1886.41) .. (578.73,1886.41) .. controls (579.82,1886.41) and (580.7,1887.33) .. (580.7,1888.47) .. controls (580.7,1889.61) and (579.82,1890.54) .. (578.73,1890.54) .. controls (577.64,1890.54) and (576.76,1889.61) .. (576.76,1888.47) -- cycle ;
\draw  [color={rgb, 255:red, 155; green, 155; blue, 155 }  ,draw opacity=1 ][fill={rgb, 255:red, 155; green, 155; blue, 155 }  ,fill opacity=1 ] (590.96,1888.47) .. controls (590.96,1887.33) and (591.84,1886.41) .. (592.93,1886.41) .. controls (594.02,1886.41) and (594.91,1887.33) .. (594.91,1888.47) .. controls (594.91,1889.61) and (594.02,1890.54) .. (592.93,1890.54) .. controls (591.84,1890.54) and (590.96,1889.61) .. (590.96,1888.47) -- cycle ;
\draw [color={rgb, 255:red, 155; green, 155; blue, 155 }  ,draw opacity=1 ][fill={rgb, 255:red, 155; green, 155; blue, 155 }  ,fill opacity=1 ]   (491.95,1838.94) -- (484.06,1888.47) ;
\draw [color={rgb, 255:red, 155; green, 155; blue, 155 }  ,draw opacity=1 ][fill={rgb, 255:red, 155; green, 155; blue, 155 }  ,fill opacity=1 ]   (491.95,1838.94) -- (498.26,1888.47) ;
\draw [color={rgb, 255:red, 155; green, 155; blue, 155 }  ,draw opacity=1 ][fill={rgb, 255:red, 155; green, 155; blue, 155 }  ,fill opacity=1 ]   (523.51,1838.94) -- (529.82,1888.47) ;
\draw [color={rgb, 255:red, 155; green, 155; blue, 155 }  ,draw opacity=1 ][fill={rgb, 255:red, 155; green, 155; blue, 155 }  ,fill opacity=1 ]   (555.06,1838.94) -- (561.38,1890.54) ;
\draw [color={rgb, 255:red, 155; green, 155; blue, 155 }  ,draw opacity=1 ][fill={rgb, 255:red, 155; green, 155; blue, 155 }  ,fill opacity=1 ]   (586.62,1838.94) -- (592.93,1888.47) ;
\draw [color={rgb, 255:red, 155; green, 155; blue, 155 }  ,draw opacity=1 ][fill={rgb, 255:red, 155; green, 155; blue, 155 }  ,fill opacity=1 ]   (523.51,1838.94) -- (515.62,1886.41) ;
\draw [color={rgb, 255:red, 155; green, 155; blue, 155 }  ,draw opacity=1 ][fill={rgb, 255:red, 155; green, 155; blue, 155 }  ,fill opacity=1 ]   (555.06,1838.94) -- (547.17,1888.47) ;
\draw [color={rgb, 255:red, 155; green, 155; blue, 155 }  ,draw opacity=1 ][fill={rgb, 255:red, 155; green, 155; blue, 155 }  ,fill opacity=1 ]   (586.62,1838.94) -- (578.73,1886.41) ;
\draw [color={rgb, 255:red, 155; green, 155; blue, 155 }  ,draw opacity=1 ][fill={rgb, 255:red, 155; green, 155; blue, 155 }  ,fill opacity=1 ]   (539.32,1740.19) -- (570.88,1789.72) ;
\draw [color={rgb, 255:red, 155; green, 155; blue, 155 }  ,draw opacity=1 ][fill={rgb, 255:red, 155; green, 155; blue, 155 }  ,fill opacity=1 ]   (507.76,1789.72) -- (539.32,1740.19) ;
\draw    (412.54,1883.81) -- (412.54,1891.97) ;
\draw    (474.87,1884.72) -- (474.87,1892.87) ;
\draw    (305.64,1689.9) -- (305.64,1896.08) ;
\draw [shift={(305.64,1898.08)}, rotate = 270] [color={rgb, 255:red, 0; green, 0; blue, 0 }  ][line width=0.75]    (4.37,-1.32) .. controls (2.78,-0.56) and (1.32,-0.12) .. (0,0) .. controls (1.32,0.12) and (2.78,0.56) .. (4.37,1.32)   ;
\draw    (305.64,1789.82) -- (310.83,1789.82) ;
\draw    (310.83,1839.95) -- (305.77,1839.95) ;
\draw    (305.64,1693.67) -- (310.83,1693.67) ;
\draw    (310.7,1888.72) -- (305.64,1888.72) ;
\draw    (305.64,1740.82) -- (310.83,1740.82) ;
\draw    (444.61,1789.4) -- (428.83,1838.94) ;
\draw  [color={rgb, 255:red, 0; green, 0; blue, 0 }  ,draw opacity=1 ][fill={rgb, 255:red, 0; green, 0; blue, 0 }  ,fill opacity=1 ] (442.64,1789.4) .. controls (442.64,1788.26) and (443.52,1787.34) .. (444.61,1787.34) .. controls (445.7,1787.34) and (446.59,1788.26) .. (446.59,1789.4) .. controls (446.59,1790.54) and (445.7,1791.47) .. (444.61,1791.47) .. controls (443.52,1791.47) and (442.64,1790.54) .. (442.64,1789.4) -- cycle ;
\draw    (444.61,1789.4) -- (460.39,1838.94) ;
\draw  [color={rgb, 255:red, 0; green, 0; blue, 0 }  ,draw opacity=1 ][fill={rgb, 255:red, 0; green, 0; blue, 0 }  ,fill opacity=1 ] (458.42,1838.94) .. controls (458.42,1837.8) and (459.3,1836.87) .. (460.39,1836.87) .. controls (461.48,1836.87) and (462.36,1837.8) .. (462.36,1838.94) .. controls (462.36,1840.08) and (461.48,1841) .. (460.39,1841) .. controls (459.3,1841) and (458.42,1840.08) .. (458.42,1838.94) -- cycle ;
\draw  [color={rgb, 255:red, 0; green, 0; blue, 0 }  ,draw opacity=1 ][fill={rgb, 255:red, 0; green, 0; blue, 0 }  ,fill opacity=1 ] (433.17,1888.47) .. controls (433.17,1887.33) and (434.06,1886.41) .. (435.15,1886.41) .. controls (436.23,1886.41) and (437.12,1887.33) .. (437.12,1888.47) .. controls (437.12,1889.61) and (436.23,1890.54) .. (435.15,1890.54) .. controls (434.06,1890.54) and (433.17,1889.61) .. (433.17,1888.47) -- cycle ;
\draw  [color={rgb, 255:red, 0; green, 0; blue, 0 }  ,draw opacity=1 ][fill={rgb, 255:red, 0; green, 0; blue, 0 }  ,fill opacity=1 ] (450.53,1888.47) .. controls (450.53,1887.33) and (451.41,1886.41) .. (452.5,1886.41) .. controls (453.59,1886.41) and (454.47,1887.33) .. (454.47,1888.47) .. controls (454.47,1889.61) and (453.59,1890.54) .. (452.5,1890.54) .. controls (451.41,1890.54) and (450.53,1889.61) .. (450.53,1888.47) -- cycle ;
\draw [color={rgb, 255:red, 0; green, 0; blue, 0 }  ,draw opacity=1 ]   (428.83,1838.94) -- (435.15,1888.47) ;
\draw    (413.09,1740.19) -- (444.61,1789.4) ;
\draw    (476.2,1690.65) -- (413.09,1740.19) ;
\draw [color={rgb, 255:red, 0; green, 0; blue, 0 }  ,draw opacity=1 ]   (460.39,1838.94) -- (452.5,1888.47) ;

\draw (391.1,1893.18) node [anchor=north west][inner sep=0.75pt]  [font=\tiny]  {$-\log^{\theta } T$};
\draw (460.85,1893.18) node [anchor=north west][inner sep=0.75pt]  [font=\tiny]  {$\log^{\theta } T$};
\draw (598.2,1891.34) node [anchor=north west][inner sep=0.75pt]  [font=\scriptsize]  {$h$};
\draw (427,1734.37) node [anchor=north west][inner sep=0.75pt]  [font=\scriptsize,color={rgb, 255:red, 208; green, 2; blue, 27 }  ,opacity=1 ]  {$S_{t_{\theta }}( h_{1}) \approx S_{t_{\theta }}( h_{2})$};
\draw (428.3,1889.87) node [anchor=north west][inner sep=0.75pt]  [font=\scriptsize]  {$h_{1}$};
\draw (392.39,1812.17) node [anchor=north west][inner sep=0.75pt]  [font=\scriptsize,color={rgb, 255:red, 74; green, 144; blue, 226 }  ,opacity=1 ]  {$\mathscr{S}_{2}( h_{1})$};
\draw (311.93,1784.8) node [anchor=north west][inner sep=0.75pt]  [font=\scriptsize]  {$t_{\theta } =2$};
\draw (312.06,1834.39) node [anchor=north west][inner sep=0.75pt]  [font=\scriptsize]  {$3$};
\draw (312.07,1689.81) node [anchor=north west][inner sep=0.75pt]  [font=\scriptsize]  {$0$};
\draw (289.4,1684.83) node [anchor=north west][inner sep=0.75pt]  [font=\footnotesize]  {$k$};
\draw (312.06,1881.95) node [anchor=north west][inner sep=0.75pt]  [font=\scriptsize]  {$t=4$};
\draw (445.24,1889.87) node [anchor=north west][inner sep=0.75pt]  [font=\scriptsize]  {$h_{2}$};
\draw (455.49,1795.74) node [anchor=north west][inner sep=0.75pt]  [font=\scriptsize,color={rgb, 255:red, 74; green, 144; blue, 226 }  ,opacity=1 ]  {$\mathscr{S}_{1}( h_{2})$};
\draw (311.93,1734.8) node [anchor=north west][inner sep=0.75pt]  [font=\scriptsize]  {$1$};
\draw (471.4,1675.83) node [anchor=north west][inner sep=0.75pt]  [font=\footnotesize]  {$r$};
\draw (446.4,1775.83) node [anchor=north west][inner sep=0.75pt]  [font=\footnotesize]  {$r_{\theta }$};

\end{tikzpicture}
    \caption{Illustration of the branching random walk model for $\log|\zeta|$ on a short interval. In a window of width $2\log^\theta T$, {all leaves share the edges between their least common ancestor $r_\theta$ and the root $r$. These edges roughly correspond to the Dirichlet sum $S_{t_\theta}(0)$ (cf.~Proposition \ref{smallprimes}). The shifted sums $\mathscr{S}_{k}(h):=S_{t_\theta+k}(h)-S_{t_\theta}(h)$, for $|h|\leq \log^\theta T$ and $k\leq t-t_\theta$, are analogous to a branching random walk on the subtree of height $t-t_\theta$ rooted at $r_\theta$.}}
    \label{fig:brw}
\end{figure}

\noindent \textbf{Upper bound for the maximum. }On the interval $|h|\leq \log^\theta T$, all of the $(S_k(h))_{k\leq t}$ approximately share their increments up to $k\approx t_\theta=|\theta|\log\log T$ (as made precise by Proposition \ref{smallprimes}). The maximum of $\log|\zeta(\tfrac{1}{2}+i(\tau+h))|=\log|\zeta_\tau(h)|$ on this interval should thus roughly equal the sum of two random variables: the first is $S_{t_\theta}(0)$ (which we note equals $-\log|P_\tau(\theta)|+O(1)$), and the second is the maximum of $(\mathscr{S}_k(h))_{|h|\leq \log^\theta T}$, where 
\begin{equation}
\label{eqn: Sk script}
\mathscr{S}_k:={S}_{t_\theta+k}-S_{t_\theta},
\end{equation}
which is itself modeled by the maximum of a branching random walk, this time over a tree of depth $t(1+\theta)$ (this is depicted in Figure \ref{fig:brw}, with the tree in question being the subtree rooted at $r_\theta$, the least common ancestor of the leaves corresponding to $(S_t(h),{|h|\leq \log^\theta T})$).

To prove Theorem \ref{UB}, our strategy therefore consists of subtracting $S_{t_\theta}(0)$ from $\log|\zeta_\tau(h)|$ and using the recursive scheme in \cite{FHK1} to bound the maximum of the difference (see Section \ref{upperboundsection}). This scheme allows one to apply, in this number-theoretic setting, the \textit{barrier method} developed by Bramson \cite{Bramson} to bound the maximal displacement in branching Brownian motion. Compared to the $\theta=0$ case in \cite{FHK1}, our proof requires additional care throughout, as one must work with longer Dirichlet polynomials. Estimating $\log|\zeta_\tau(h)|-S_{t_\theta}(0)$ also requires control over the ``error'' $\max_{|h|\leq (\log T)^\theta}|S_{t_\theta}(h)-S_{t_\theta}(0)|$, for which a Gaussian tail (Proposition \ref{smallprimes}) is derived in Section \ref{corrsection}.

\bigskip

\noindent \textbf{Upper bound for the partition function at criticality.} We begin with the case $\theta=0$, and express $\mathcal{Z}_2(0)$ in terms of the measure of level sets of $\log|\zeta_\tau(h)|$:
\[
    \mathcal{Z}_2(0)\asymp \int_{-\infty}^{\infty} e^{2V}\text{meas}\big\{|h|\leq 1:\log|\zeta_\tau(h)|>V\big\}\mathrm{d}V.
\]
\noindent By Theorem \ref{UB}, we may truncate this integral at $M=m(t)+\log A$, since the maximum of $\log |\zeta_\tau(h)|$ only exceeds this threshold with probability $\ll 1/A$. Following the proof of Corollary 1.4 in \cite{arguinbailey1}, the claimed bound would follow if one could show that
\begin{equation}\label{integral}
    \int_{-1}^1\int_{0}^{M} e^{2V}\mathbb{P}\big\{\log |\zeta_\tau(h)|>V\,\big\} \mathrm{d}V\,\mathrm{d}h \ll A \frac{e^t}{\sqrt{t}}.
\end{equation}
 However, the large deviations estimates for $\log |\zeta_\tau(h)|$ in \cite{arguinbailey1, arguincreighton} show that the left-hand side is of the order of $\int_0^{M}e^{2V-V^2/t}/\sqrt{t}\asymp e^t$, the expectation of $\mathcal{Z}_2(0)$ being inflated by atypically large values of $\log|\zeta_\tau(h)|$.

To recover the additional factor of $1/\sqrt{t}$, it is helpful to once again liken the process $(\log |\zeta_\tau(h)|,|h|\leq 1)$ to the approximate branching random walk $(S_t(h), |h|\leq 1)$. Doing so suggests restricting to a good event on which $S_k(h)$ remains below a {barrier} $U_{\log A}(k)$ for all $k\leq t$ and $|h|\leq 1$. Were this a genuine branching random walk, one could use the logarithmic correlation structure of $(S_t(h))_h$ to, roughly speaking, reduce the integrand in \eqref{integral} to
\[
    e^{2V}\mathbb{P}\big\{S_t(h)>V, \, S_k(h) \leq m(k)+U_{\log A}(k)\, \forall k\leq t\big\},
\]
for a $U_{\log A}$ lying \textit{below} the typical fluctuations of a random walk conditioned to end at $V$.
This would be computable via a ballot theorem (Proposition \ref{prop:ballot}), and integrating the resulting estimate over $V\in [0,M]$ would recover the missing $1/\sqrt{t}$. This approach was recently used by Chang \cite{chang} to prove an analogue of \eqref{harper} for a Gaussian model of $\log|\zeta_\tau|$, and Harper's argument, while different, also makes crucial use of barriers \cite{harper2019}.

Adapting this approach to $\log|\zeta_\tau|$ requires one to control the deviations of the increments $S_k-S_{k-1}$ for large $k\approx t$, showing that their behaviour is close to Gaussian. This amounts to computing large moments of long Dirichlet polynomials, and the error in doing so (using Lemma \ref{MVMV}) can quickly become unruly. 
The recursive scheme of \cite{FHK1} handles this by iteratively constructing both \textit{upper} and \textit{lower} barriers $U_{\log A}$ and $L_{\log A}$ which restrict the range of values that these increments can take. At every step of this iteration, one then only needs to consider Dirichlet polynomials short enough for Lemma \ref{MVMV} to be applicable. 

However, this approach relies on the gap $U_{\log A}(k)-L_{\log A}(k)$ between these barriers being sufficiently small as $k$ grows, shrinking to $O(1)$ at time $t$. This raises an obvious problem in our previous argument, which requires the same barrier $U_{\log A}$ to be taken uniformly over all $V$: any corresponding lower barrier would have to lie below $(k/t)V$, potentially resulting in a large gap for $V$ away from $U_{\log A}(t)\approx M$.

To resolve this, we make the simple yet crucial observation that a barrier up to time $t_*=t/2$ suffices to get the sought-after bound. That this should be the case is not a priori clear; it follows from the fact that the resulting integral is localized at levels $V=t-O(\sqrt{t})$, for which a barrier up to $t_*$ yields a bound of the same order as one up to $t$. This is not the case for larger $V$ (e.g., for $V\approx M$, one is off by a factor of $\sqrt{t}$), but the Gaussian weight will suffice to offset the resulting error.

Concretely, we place ourselves on the event that $S_k(h)\in [L_{\log A}(k),U_{\log A}(k)]$ for all $k\leq t_*$ for a suitable choice of barriers. This is straightforward, as $t_*$ is  small enough for Lemma \ref{MVMV} to apply directly, and one must then prove large deviation estimates for $\log|\zeta_\tau(h)|$ {on this event}. This still requires controlling the problematic increments $S_k-S_{k-1}$ for $k\in (t_*,t]$, but since no barrier is required at these times, this can be achieved by a variant of the recursive scheme developed in \cite{arguinbailey1} in which one only restricts $S_k$ at a sequence of sparse checkpoints $(t_\ell)_{\ell}$ (depicted in Figure \ref{fig:partial_barrier}). This gives rise to the ($\theta=0$ case of the) following result.

\begin{figure}
    \centering

\tikzset{every picture/.style={line width=0.75pt}} 
\resizebox{250px}{!}{
\begin{tikzpicture}[x=0.75pt,y=0.75pt,yscale=-1,xscale=1]

\draw (293.01,2450.15) node  {\includegraphics[width=265.5pt,height=156.65pt]{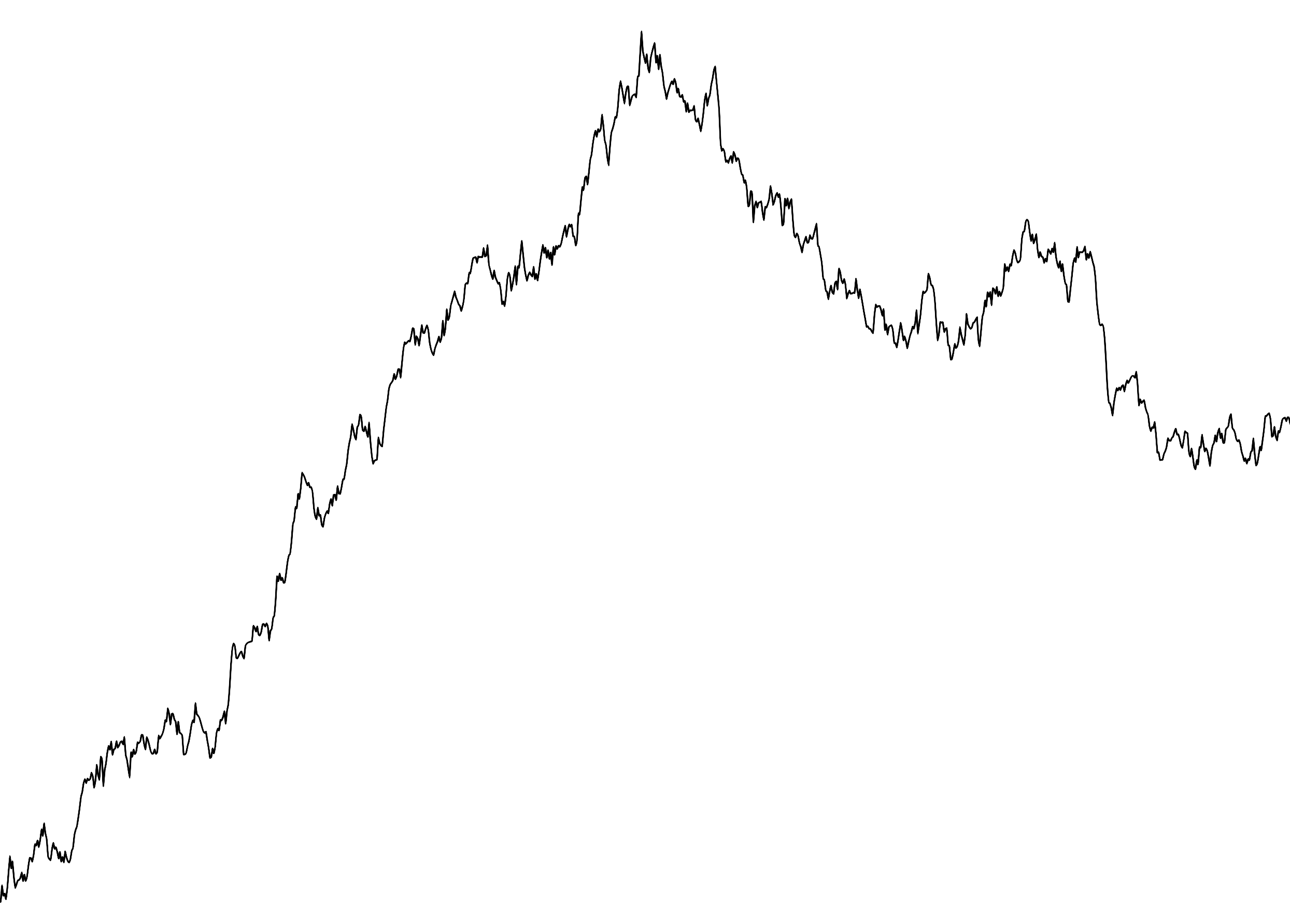}};
\draw    (103.29,2551.9) -- (507.72,2551.9) ;
\draw [shift={(509.72,2551.9)}, rotate = 180] [color={rgb, 255:red, 0; green, 0; blue, 0 }  ][line width=0.75]    (4.37,-1.32) .. controls (2.78,-0.56) and (1.32,-0.12) .. (0,0) .. controls (1.32,0.12) and (2.78,0.56) .. (4.37,1.32)   ;
\draw [line width=0.75]    (116.01,2562.61) -- (116.01,2300.09) ;
\draw [shift={(116.01,2298.09)}, rotate = 90] [color={rgb, 255:red, 0; green, 0; blue, 0 }  ][line width=0.75]    (4.37,-1.32) .. controls (2.78,-0.56) and (1.32,-0.12) .. (0,0) .. controls (1.32,0.12) and (2.78,0.56) .. (4.37,1.32)   ;
\draw [color={rgb, 255:red, 0; green, 0; blue, 0 }  ,draw opacity=1 ]   (116.51,2479.41) .. controls (128.1,2401.25) and (257.39,2317.9) .. (325.74,2308.58) ;
\draw [color={rgb, 255:red, 0; green, 0; blue, 0 }  ,draw opacity=1 ]   (108.44,2479.41) -- (122.07,2479.41) ;
\draw [color={rgb, 255:red, 0; green, 0; blue, 0 }  ,draw opacity=1 ] [dash pattern={on 4.5pt off 4.5pt}]  (116.51,2479.41) -- (469.49,2327.22) ;
\draw  [dash pattern={on 0.84pt off 2.51pt}]  (325.74,2553.13) -- (325.74,2294.2) ;
\draw  [dash pattern={on 0.84pt off 2.51pt}]  (470.01,2294.13) -- (470.01,2553.73) ;
\draw [color={rgb, 255:red, 0; green, 0; blue, 0 }  ,draw opacity=1 ]   (464.78,2327.22) -- (475.44,2327.22) ;
\draw [color={rgb, 255:red, 155; green, 155; blue, 155 }  ,draw opacity=1 ]   (423.85,2417.82) -- (415.23,2417.82) ;
\draw  [dash pattern={on 0.84pt off 2.51pt}]  (440.69,2560.32) -- (453.86,2560.32) ;
\draw    (476.44,2442.14) -- (464.47,2442.14) ;
\draw [color={rgb, 255:red, 155; green, 155; blue, 155 }  ,draw opacity=1 ]   (423.85,2460.47) -- (415.23,2460.47) ;
\draw [color={rgb, 255:red, 155; green, 155; blue, 155 }  ,draw opacity=1 ]   (393.05,2465.1) -- (384.43,2465.1) ;
\draw [color={rgb, 255:red, 155; green, 155; blue, 155 }  ,draw opacity=1 ]   (393.26,2404.18) -- (384.64,2404.18) ;
\draw [color={rgb, 255:red, 155; green, 155; blue, 155 }  ,draw opacity=1 ]   (473.76,2438.02) -- (465.14,2438.02) ;
\draw [color={rgb, 255:red, 155; green, 155; blue, 155 }  ,draw opacity=1 ]   (473.76,2455.85) -- (465.14,2455.85) ;
\draw [color={rgb, 255:red, 155; green, 155; blue, 155 }  ,draw opacity=1 ]   (389.64,2404.18) -- (389.64,2465.1) ;
\draw [color={rgb, 255:red, 155; green, 155; blue, 155 }  ,draw opacity=1 ]   (419.11,2417.82) -- (419.11,2460.47) ;
\draw [color={rgb, 255:red, 155; green, 155; blue, 155 }  ,draw opacity=1 ]   (470.01,2438.1) -- (470.01,2455.85) ;
\draw    (386.91,2549.75) -- (386.91,2554.55) ;
\draw    (325.74,2549.97) -- (325.74,2554.77) ;
\draw    (419.11,2549.75) -- (419.11,2554.55) ;
\draw    (470.01,2549.87) -- (470.01,2554.58) ;
\draw [color={rgb, 255:red, 155; green, 155; blue, 155 }  ,draw opacity=1 ]   (459.48,2456.65) -- (451.14,2456.65) ;
\draw [color={rgb, 255:red, 155; green, 155; blue, 155 }  ,draw opacity=1 ]   (455.48,2435.76) -- (455.48,2456.65) ;
\draw [color={rgb, 255:red, 155; green, 155; blue, 155 }  ,draw opacity=1 ]   (459.48,2436.08) -- (451.14,2436.08) ;
\draw [color={rgb, 255:red, 155; green, 155; blue, 155 }  ,draw opacity=1 ]   (434.48,2430.29) -- (434.48,2459.02) ;
\draw [color={rgb, 255:red, 155; green, 155; blue, 155 }  ,draw opacity=1 ]   (438.48,2430.29) -- (430.14,2430.29) ;
\draw [color={rgb, 255:red, 155; green, 155; blue, 155 }  ,draw opacity=1 ]   (438.48,2458.77) -- (430.14,2458.77) ;
\draw [color={rgb, 255:red, 155; green, 155; blue, 155 }  ,draw opacity=1 ] [dash pattern={on 0.84pt off 2.51pt}]  (439.69,2439.57) -- (452.86,2439.57) ;

\draw (76.37,2471.93) node [anchor=north west][inner sep=0.75pt]  [font=\footnotesize,color={rgb, 255:red, 0; green, 0; blue, 0 }  ,opacity=1 ]  {$\log A$};
\draw (477.04,2319.14) node [anchor=north west][inner sep=0.75pt]  [font=\footnotesize,color={rgb, 255:red, 0; green, 0; blue, 0 }  ,opacity=1 ]  {$m( t) +\log A$};
\draw (320.33,2557.74) node [anchor=north west][inner sep=0.75pt]  [font=\scriptsize]  {$t_{0}$};
\draw (381.37,2557.74) node [anchor=north west][inner sep=0.75pt]  [font=\scriptsize]  {$t_{1}$};
\draw (413.85,2557.74) node [anchor=north west][inner sep=0.75pt]  [font=\scriptsize]  {$t_{2}$};
\draw (464.27,2557.74) node [anchor=north west][inner sep=0.75pt]  [font=\scriptsize]  {$t_{\mathcal{L}}$};
\draw (479.16,2435.85) node [anchor=north west][inner sep=0.75pt]  [font=\scriptsize]  {$V$};
\draw (184.79,2325.32) node [anchor=north west][inner sep=0.75pt]  [font=\footnotesize,color={rgb, 255:red, 0; green, 0; blue, 0 }  ,opacity=1 ]  {$U_{\log A}( k)$};
\draw (94.14,2300.11) node [anchor=north west][inner sep=0.75pt]  [font=\footnotesize]  {$\mathscr{S}_{k}$};
\draw (498.14,2555.76) node [anchor=north west][inner sep=0.75pt]  [font=\footnotesize]  {$k$};

\end{tikzpicture}
}
    \caption{The partial barrier $U_{\log A}(k)$ controls $\mathscr{S}_k=S_{k+t_\theta}-S_{t_\theta}$ uniformly over the value $V$ of $\mathscr{S}_{t-t_\theta}$. }
    \label{fig:partial_barrier}
\end{figure}


{
\begin{theorem}[Large deviations with constraints on a prime interval] \label{conditionLD} Fix $\theta\in (-1,0]$ and $t_*=t_*(\theta)=(t-t_\theta)/2$. Let $C>10$ be a large constant, and for any $A\geq 1$, define the barriers
\begin{align*}
    &{U_A(k)=A+k+C\log k},\quad {L_A(k)=A-Ck},\quad \text{ for } k>A/4,
\end{align*}
{and $U_A(k)=-L_A(k)=\infty$ otherwise.}
Then uniformly in large enough $t$, $|h|\leq 2$, $1\leq A\leq \sqrt{t}$, and $t_*-t_*^{2/3}\leq V/2<U_A(t_*)$,
    \begin{align}\label{condLD}
    \begin{split}
            \mathbb{P}\Big\{\log |\zeta_\tau(h)|-S_{t_\theta}(h) > &\,V, {\mathscr{S}}_{k}(h){\,\in[L_A(k),U_A(k)]}, \,\forall k\leq {t_*}\Big\} \\&\ll \frac{A\left(U_A(t_*)-{V/2}+\sqrt{t_*}\right)}{{t_*}}\frac{e^{-V^2/(t-t_\theta)}}{\sqrt{t-t_\theta}}.
              \end{split}
        \end{align}
\end{theorem}
}
\noindent This theorem can be seen as an instance of a more general result which may be of independent interest. That is, for a reasonable enough event $E$ which only depends on the increments $(S_k-S_{k-1})(h)$ for $k\leq {t_*}$, the same proof gives that
\[
    \mathbb{P}\big\{\{\log |\zeta_\tau(h)|>V\}\cap E\big\} \ll \mathbb{P}\big\{\{\mathcal{W}_t > V\}\cap {\tilde{E}}\big\},
\]
where $(\mathcal{W}_k)_{k\leq t}$ is a Gaussian random walk with $t$ i.i.d. increments of variance $1/2$, and $\tilde{E}$ is the event obtained from $E$ by substituting each $\mathcal{W}_k-\mathcal{W}_{k-1}$ for $S_{k}-S_{k-1}$. 

\bigskip

Lastly, to bound $\mathcal{Z}_2{(\theta)}$ when $\theta <0$, we begin by factoring out the contribution coming from $S_{t_\theta}(h)$ in the integral, as it is (morally speaking) constant for all $h$. This is achieved using the following proposition.
\begin{proposition}[Correlation of primes smaller than $\exp(e^{|\theta|t})$]\label{smallprimes}
    Let $\theta \in (-1,0]$ and ${2<x<e^{t(1+\theta)/2}}$. Then there exist absolute constants $c,C>0$ such that
    \[
        \P{\max_{|h|\leq 2e^{\theta t}}|S_{t_\theta}(h)-S_{t_\theta}(0)|>x}
        \leq Ce^{-x^2/c}.
    \]
\end{proposition}
\noindent The rest of the argument to bound $\mathcal{Z}_2(\theta)$ is the same as the one used for $\mathcal{Z}_2{(0)}$, applied to $\log|\zeta_\tau(h)|-S_{t_\theta}(h)$ instead of $\log|\zeta_\tau(h)|$. The barrier event now restricts the evolution of the shifted sums $(\mathscr{S}_k(h))_k$, and one uses the $\theta\in (-1,0)$ case of Theorem \ref{conditionLD}.

\subsection{Organization} {In Section \ref{sec:moments}, we prove Theorem \ref{moments} from Proposition \ref{smallprimes} and Theorems \ref{UB} and \ref{conditionLD}. These theorems are then proved in Sections \ref{corrsection}, \ref{upperboundsection} and \ref{sec:LD}, respectively. Section \ref{sectionaux} contains proofs of auxiliary results which are used in the proof of Theorem \ref{UB}. The appendix collates known results that are used throughout the paper. }

\section{Upper bound for the partition function}\label{sec:moments}
    We first prove Theorem \ref{moments} for $\theta=0$, in which case we may assume that $\log A \leq \sqrt{t}$ (the desired bound otherwise following directly from Markov's inequality and Lemma \ref{secondmomentzeta}). 
    
    Consider the following set of {\it good} $h$'s, defined by
    \[
        G_A=\{|h|\leq 1 : S_k(h)\in {[L_{\log A}(k), U_{\log A}(k)]} \text{ for all } k\leq {t_*}\},
    \]
    where {$t_*$}, $L_{\log A}(k)$ and $U_{\log A}(k)$ are as in the statement of Theorem \ref{conditionLD} (for $\theta=0$). We are interested in the normalised measure of high points of the zeta function belonging to this good set, meaning
    \[
        \mathcal{S}(V;G_A):=\frac{1}{2}\text{meas}\{|h|\leq 1:\log|\zeta_\tau(h)|>V, h\in G_A\},
    \]
    which we note is a random variable taking values in $[0,1]$. We will also need $$\mathcal{S}(V):=\frac{1}{2}\mathrm{meas}\{|h|\leq 1\, : \log |\zeta_\tau(h)|>V\}.$$

    The proof follows the strategy outlined in \cite{arguinbailey1} (proof of Corollary 1.4), with the main difference being that we use large deviation estimates for zeta \textit{in the set} $G_A$ (given by Theorem \ref{conditionLD}) to get an additional factor of ${(t-t_\theta)}^{-1/2}$. We begin by integrating by parts to get
    \[
        \mathcal{Z}_2(0) = -e^{2V}\mathcal{S}(V)\bigg|^{+\infty}_{-\infty}+2\int_{\mathbb{R}}e^{2V}\mathcal{S}(V)\mathrm{d}V.
    \]
    The boundary term at $-\infty$ is equal to zero, since $\mathcal{S}(V)\leq 1$ by definition. To deal with the limit at $+\infty$, we place ourselves on the event
    $$E_A:=\{\forall h \in [-1,1],\,\log |\zeta_\tau(h)|\leq m(t)+\log A \text{ and } h\in G_A\},$$ where $m(t)=t-\tfrac{3}{4}\log t$, {on which said limit is zero (since $\mathcal{S}(V)=0$ for $V>m(t)+\log A$).} This can be done since $\P{E_A^c}$ is bounded by
    \begin{align*}
         \P{\max_{|h|\leq 1}\log |\zeta_\tau(h)|>m(t)+\log A}+\P{{\substack{\exists k\leq t/2,\\\exists h\in [-1,1]}, S_k(h)\notin [L_{\log A}(k),U_{\log A}(k)]}},
    \end{align*}
     both of these terms being $\ll \log A\cdot e^{-2\log A}\ll1/A$, by Theorem \ref{UB} and 
    the proof of Proposition \ref{base}, respectively. {Noting that $\mathcal{S}(V)=\mathcal{S}(V;G_A)$ on $E_A$,
    we are left with the task of bounding
    \[
        \int_{\mathbb{R}}e^{2V}\mathcal{S}(V;G_A)\mathrm{d}V=\int_{-\infty}^{m(t)+\log A} e^{2V}\mathcal{S}(V;G_A)\mathrm{d}V.
    \]
    Noting that $\mathcal{S}(V;G_A)\leq \mathcal{S}(V)\leq 1$, we can write 
    \begin{align*}
         \int_{-\infty}^{m(t)+\log A} e^{2V}\mathcal{S}(V;G_A)\mathrm{d}V &\leq \int_{-\infty}^{t/4} e^{2V} \mathrm{d}V  +\int_{t/4}^{m(t)+\log A} e^{2V}\mathcal{S}(V;G_A)\mathrm{d}V\\
         &\leq e^{t/2} + \int_{t/4}^{m(t)+\log A} e^{2V}\mathcal{S}(V;G_A)\mathrm{d}V,
    \end{align*}
    and since $e^{t/2}\leq e^t/\sqrt{t}$ for $t\geq 1$, one must then show that the rightmost integral is of the desired order. This requires controlling $\mathcal{S}(V;G_A)$ on this range of $V$.} To that end, we partition said range into sub-intervals of length of the order of $\sqrt{t}$, letting  $(V_j, 1\leq j \leq J)$ be an enumeration of the elements of  $\left[\frac{t}{4},m(t)+\log A\right]\cap (\sqrt{t}\mathbb{Z})$ and defining $V_0=V_1-\sqrt{t}$ and $V_{J+1}=V_J+\sqrt{t}$. The integral in question can then be partitioned into
    \[
        I_j = \int_{V_j}^{V_{j+1}}e^{2V}\mathcal{S}(V;G_A)\mathrm{d}V,\quad 0\leq j\leq J,
    \]
    and we let $E_j$ be the event that the integral $I_j$ does not exceed a constant $a_j$ times its expectation, namely
    \[ 
        E_j := \big\{I_j \leq a_j \E{I_j}\big\}
    \] 
    where 
    \[
        a_j := \frac{A}{\log A}\cdot\left\{
    \begin{array}{ll}
		\big(1+{\big(V_j-t\big)^2}/{t}\big)  & \mbox{if } V_j>t \\
		\big(1+{\big(V_{j+1}-t\big)^2}/{t}\big) & \mbox{if } V_j\leq t \mbox{ and } V_{j+1}\leq t\\ 
        1 & \mbox{if }  V_j\leq t \mbox{ and }V_{j+1}>t  \\
	\end{array}\right..
    \]
    This choice of $a_j$ ensures that on $[V_j,V_{j+1}]$, $$a_j\leq \frac{A}{\log A}\bigg(1+\frac{(V-t)^2}{t}\bigg)$$ while also making the $a_j^{-1}$ summable. By Markov's inequality and a union bound, it follows that the good event $$G=E_A\cap \bigg(\bigcap_j E_j\bigg)$$ is such that
    \[
        \P{G^c} \ll \sum_j a_j^{-1}+1/A \ll {\log (A)}/{A}.
    \]
    We can therefore place ourselves on the event $G$, where we have
    \begin{align}
        &\int_{t/4}^{m(t)+\log A} e^{2V}\mathcal{S}(V;G_A)\mathrm{d}V \leq \sum_{j} a_j \int_{V_{j}}^{V_{j+1}}e^{2V}\E{\mathcal{S}(V;G_A)}\mathrm{d}V \nonumber \\ 
            &\leq \int_{t/4}^{m(t)+\log A} \frac{A}{\log A}\bigg(1+\frac{(V-t)^2}{t}\bigg) e^{2V}\E{\mathcal{S}(V;G_A)}\mathrm{d}V.\label{eq:pre-fubini}
    \end{align}
    {
    By Fubini's theorem and Theorem \ref{conditionLD}, 
    \begin{equation}\label{gooddensity}
               \E{\mathcal{S}(V;G_A)} \ll \frac{\log (A)\left(U_{\log A}({t/2})-V{/2}+\sqrt{t}\right)}{t}\frac{e^{-V^2/t}}{\sqrt{t}},
    \end{equation}
    for $V>2(t_*-t_*^{2/3})$ in \eqref{eq:pre-fubini}. For the remaining $V\in (t/4,2(t_*-t_*^{2/3})]$ the main theorem in \cite{arguinbailey1} gives the bound $\ll e^{-V^2/t}/\sqrt{t}$, and the integral in \eqref{eq:pre-fubini} is thus
    \begin{align*}
        &\ll Ae^{t}\int_{t/4}^{2(t_*-t_*^{2/3})}\bigg(1+\frac{(V-t)^2}{t}\bigg)\frac{e^{-(V-t)^2/t}}{\sqrt{t}} \mathrm{d}V\\&+Ae^{t}\int_{2(t_*-t_*^{2/3})}^{m(t)+\log A}\bigg(1+\frac{(V-t)^2}{t}\bigg)\frac{\left(U_{\log A}({t/2})-V{/2}+\sqrt{t}\right)}{t}\frac{e^{-(V-t)^2/t}}{\sqrt{t}}\mathrm{d}V.
    \end{align*}
    Using the change of variables $u=(t-V)/\sqrt{t}$, the first of these terms is easily seen to satisfy the desired bound. For $V\in (2(t_*-t_*^{2/3}),m(t)+\log A]$ (and $\log A\leq \sqrt{t}$), we can use the estimate
    \begin{align*}
        \frac{\left(U_{\log A}({t/2})-V{/2}+\sqrt{t}\right)}{t}\ll\frac{1}{\sqrt{t}}(\sqrt{t}-V/\sqrt{t}+1)\ll \frac{1}{\sqrt{t}}(1+u),
    \end{align*}
     to conclude that
    \begin{align*}
        \mathcal{Z}_2(0)
        &\ll A \frac{e^{t}}{\sqrt{t}}+ A \frac{e^{t}}{\sqrt{t}} \int_{((3/4)\log t-\log A)/\sqrt{t}}^{2^{1/3}t^{1/6}} (1+u^2)(1+u)e^{-u^2}\mathrm{d}u \\
        &\ll A \frac{e^{t}}{\sqrt{t}}+A \frac{e^{t}}{\sqrt{t}} \int_{u>0} (1+u^2)(1+u)e^{-u^2}\mathrm{d}u\\
        &\ll A \frac{e^{t}}{\sqrt{t}}.
    \end{align*}}
    
    To prove the claim for $\theta\neq 0$, we proceed exactly as above while replacing every occurrence of $\zeta_\tau$ with $\zeta_\tau e^{-S_{t_\theta}}$, $S_k$ with $\mathscr{S}_k$, $t$ with $t-t_\theta$, and using the $\theta<0$ case of Theorem \ref{conditionLD}. We also scale the factors $a_j$ by $e^{-2\sqrt{c\log A}}$, where $c$ is the constant from Proposition \ref{smallprimes}. Doing so shows that the probability that
    \begin{align}\label{eq:meso_moment}
        \frac{1}{e^{t\theta}}\int_{|h|\leq e^{\theta t}}|\zeta_{\tau}e^{-S_{t_\theta}}(h)|^2dh \ll \frac{A}{e^{2\sqrt{c\log A}}} \frac{e^{t-t_\theta}}{\sqrt{t-t_\theta}}
    \end{align}
    holds is $1-O(e^{2\sqrt{c\log A}}\log (A)/A)$.

    Now using Proposition \ref{smallprimes}, we also have that
    \[
        \max_{|h|\leq \log^\theta T} |S_{t_\theta}(h)-S_{t_\theta}(0)| \leq \sqrt{c\log A}
    \]
    occurs with probability $1-O(1/A)$, {and we can therefore place ourselves on the event where this inequality holds. This yields
    \begin{align*}
        e^{-2S_{t_\theta}(0)}\mathcal{Z}_2{(\theta)} \ll A\frac{e^{2\cdot{\max_{h}|S_{t_\theta}(h)-S_{t_\theta}(0)|}}}{e^{2\sqrt{c\log A}}} \frac{e^{t-t_\theta}}{\sqrt{t-t_\theta}} \leq A \frac{e^{t-t_\theta}}{\sqrt{t-t_\theta}},
    \end{align*}
    and the claim follows since $e^{-2S_{t_\theta}(0)}\asymp |P_\tau(\theta)|^2.$}

\section{Correlation of primes up to $\exp(e^{t|\theta|})$}\label{corrsection}This section proves Proposition \ref{smallprimes}, which roughly states that the contribution to $|\zeta_\tau(h)|$ coming from primes $p\leq \exp(e^{t|\theta|})$ is typically constant over $|h|\leq 2e^{t\theta}$.
\bigskip 

 The claim is trivial for $\theta =0$ as $S_{t_\theta}\equiv 0$. Otherwise, we can use Lemma \ref{disclemma} with $N=\exp(2e^{t_\theta})$ and $A=100$ to write
 \begin{align}\label{discmeso}
     \nonumber &\mathbb{P}\Big\{\max_{|h|\leq 2 e^{\theta t}}|S_{t_\theta}(0)-S_{t_\theta}(h)|>x\Big\} \\&\leq \mathbb{P}\Big\{\sum_{|j|\leq 32} |S_{t_\theta}(0)-S_{t_\theta}(h_j)|^2>x^2/C\Big\}+\mathbb{P}\Big\{\sum_{|j|>32} \frac{|S_{t_\theta}(0)-S_{t_\theta}(h_j)|^2}{(|j|^{100}+1)}>x^2/C\Big\}
 \end{align}
 where $h_j=h_j(t)=\pi j/(8e^{t_\theta})$, and $C>0$ is a constant that is assumed to be large without loss of generality. We will bound these probabilities using the following moment estimate, whose proof is deferred to later in the section.

\begin{lemma}\label{MV} For any $h\in \mathbb{R}$ and any positive integer $q$ satisfying $2q\leq e^{t(1+\theta)}$, 
\[
    \mathbb{E}\Big\{|S_{t_\theta}(0)-S_{t_\theta}(h)|^{2q}\Big\}\ll\frac{(2q)!}{2^qq!}\big( e^{2t_\theta}|h|^2\big)^q.
\]
\end{lemma}
To use this, we first note the bound
\[
    \mathbb{P}\Big\{\sum_{|j|\leq32} |S_{t_\theta}(0)-S_{t_\theta}(h_j)|^2>x^2/C\Big\} \ll \sum_{|j|\leq 32}\mathbb{P}\Big\{|S_{t_\theta}(0)-S_{t_\theta}(h_j)|^2>x^2/(64C)\Big\}.
\]
By Markov's inequality and Lemma \ref{MV}, this is 
\[
    \ll \sum_{|j|\leq 32}\frac{(2q)!}{2^qq!}\Big(\frac{64 C\cdot e^{2t_\theta}|h_j|^2}{x^2}\Big)^q \ll \frac{(2q)!}{2^qq!}\Big(\frac{C'}{x^2}\Big)^q
\]
for any positive integer $q$ satisfying $2q\leq e^{t(1+\theta)}$, where $C'=(32\pi)^2C$. Picking $q=\lfloor x^2/2C'\rfloor$ (which is a valid choice by our assumption that $x^2\leq e^{t(1+\theta)}$, and the fact that $C>1$), we can use Stirling's approximation to bound the right-hand side by
\[
    \ll \bigg(\frac{2q}{e}\bigg)^{2q}\bigg(\frac{e}{2q}\bigg)^q \bigg(\frac{C'}{x^2}\bigg)^q \ll e^{-q}\bigg(\frac{2C'}{x^2}\cdot q\bigg)^q \ll e^{-x^2/c},
\]
for some constant $c>0$, proving the claim.

To show that the remaining term in \eqref{discmeso} is negligible, we first note that
\[
    \bigg\{\sum_{|j|>32} \frac{|S_{t_\theta}(0)-S_{t_\theta}(h_j)|^2}{(|j|^{100}+1)}>\frac{x^2}{C}\bigg\}\subseteq \bigcup_{|j|>32} \Big\{|S_{t_\theta}(0)-S_{t_\theta}(h_j)|^2>\frac{x^2}{C}|j|^{50}\Big\}
\]
since $\sum_{|j|>32}|j|^{50}(|j|^{100}+1)^{-1}<1$. It follows that
\[
    \mathbb{P}\Big\{\sum_{|j|>32} \frac{|S_{t_\theta}(0)-S_{t_\theta}(h_j)|^2}{(|j|^{100}+1)}>\frac{x^2}{C}\Big\}  \leq \sum_{|j|>32}\mathbb{P}\Big\{ |S_{t_\theta}(0)-S_{t_\theta}(h_j)|^2>\frac{x^2}{C}|j|^{50}\Big\},
\]
and we can bound each term in the resulting sum using Markov's inequality and Lemma \ref{MV}. To be precise, picking $q=\lfloor x^2/(C\pi^2/32)\rfloor$ (which is $\leq e^{t(1+\theta)}$ for $C>64/\pi^2$) gives
\[
    \sum_{|j|>32}\mathbb{P}\Big\{ |S_{t_\theta}(0)-S_{t_\theta}(h_j)|^{2q}>\Big(\frac{x^2}{C}\Big)^q|j|^{50q}\Big\} \ll \sum_{|j|>32}\frac{(2q)!}{2^qq!}\bigg(\frac{C\pi^2}{64x^2}\bigg)^q |j|^{-40q},
\]
which is $\ll e^{-q}\sum_{|j|>1}|j|^{-2}\ll e^{-x^2/c}$ by Stirling's approximation, for a constant $c>0$.

It remains to prove Lemma \ref{MV}.

\subsection{Proof of Lemma \ref{MV}}

We begin by introducing some notation. Let $(Z_p, \text{$p$ prime})$ be a sequence of independent and identically distributed random variables that are uniformly distributed on the unit circle. Such variables are sometimes called {\it Steinhaus random variables}. For any integer $n$ with prime factorization $\prod_i p_i^{\alpha_i}$ (where the $p_i$ are distinct for different $i$), let 
\[
    Z_n := \prod_{i=1}Z_{p_i}^{\alpha_i}.
\]
Note that we have the orthogonality relation $\mathbb{E}{Z_n\overline{Z_m}}=\ind{n=m}$. 

Consider the random variables 
\begin{equation}
\label{eqn: X_p}
    X_p(h)=\text{Re}\left(Z_p p^{-1/2-ih}+\tfrac{1}{2}Z_p^2p^{-1-2ih}\right), \text{ $p$ prime,}
\end{equation}
which replace $p^{-i\tau}$ by $Z_p$.
The sums $\sum_{e^{1000}<\log p\leq e^{t_{\theta}}} X_p(0)-X_p(h)$ will serve as a random model for $S_{t_\theta}(0)-S_{t_\theta}(h)$. We will show that the former's moments approximate the latter's, and that they are roughly Gaussian. 

We argue as in the proof of Lemma 16 in \cite{FHK1}. Let $\Phi$ be a smooth function on $\mathbb{R}$ satisfying $\Phi(x)\geq 0$ for all $x\in \mathbb{R}$, $\Phi(x)\gg 1 $ for $x\in [-1,1]$, and $\hat{\Phi}$ is compactly supported in $[-1,1]$ (see \cite{FHK1}, p. 37 for an example). Then for any collection $p_1,..., p_k$ and $q_1,...,q_\ell$ of (not necessarily distinct) primes satisfying $\prod_i p_i, \prod_jq_j\leq T$,
\[
    \int_{\mathbb{R}}\bigg(\frac{p_1\cdots p_k}{q_1\cdots q_\ell}\bigg)^{it}{\Phi}\bigg(\frac{t}{2T}\bigg)\mathrm{d}t=2T\hat{\Phi}\bigg(2T\log\Big(\frac{p_1\cdots p_k}{q_1\cdots q_\ell}\Big)\bigg)=2T\hat{\Phi}(0)\mathbf{1}_{p_1\cdots p_k=q_1\cdots q_\ell},
\]
and the right-hand side equals $2T\hat{\Phi}(0)\mathbb{E}\big\{Z_{p_1\cdots p_k}\overline{Z}_{q_1\cdots q_\ell}\big\}$ by the aforementioned orthogonality relation.

It follows that for any $h$ and any integer $q>0$ satisfying $2q\leq e^{t(1+\theta)}$, we can develop the product $(S_{t_\theta}(0)-S_{t_\theta}(h))^{2q}$ to get
\begin{align*}
    \E{|S_{t_\theta}(0)-S_{t_\theta}(h)|^{2q}} \ll \E{\bigg(\sum_{e^{1000}<\log p\leq e^{t_\theta}} X_p(0)-X_p(h)\bigg)^{2q}}.
\end{align*}
The right-hand side can then be compared to moments of a Gaussian analogue. Let
\[
    \mathcal{S}_1=\sum_{e^{1000}<\log p\leq e^{t_\theta}}\text{Re}\bigg(Z_p\cdot\frac{\Delta(p)}{p^{1/2}}\bigg),\quad \mathcal{S}_2=\frac{1}{2}\sum_{e^{1000}<\log p\leq e^{t_\theta}}\text{Re}\bigg(Z_p^2\cdot\frac{\Delta(p)}{p}\bigg),
\]
where $\Delta(p):=|1-p^{-ih}|$, and similarly 
\[
    \mathcal{G}_1=\sum_{e^{1000}<\log p\leq e^{t_\theta}}\bigg(Y_p\cdot\frac{\Delta(p)}{p^{1/2}}\bigg),\quad \mathcal{G}_2=\frac{1}{2}\sum_{e^{1000}<\log p\leq e^{t_\theta}}\bigg(Y_p\cdot\frac{\Delta(p)}{p}\bigg),
\]
where the $(Y_p)_{p\text{ prime}}$  are centered, independent Gaussian random variables of variance $1/2$. Noting that $\E{(\text{Re}Z_p)^{2m}}\leq \E{Y_p^{2m}}$ for any prime $p$ and positive integer $m$, and that the odd moments of $\text{Re}Z_p$ and $Y_p$ are zero, the binomial theorem gives 
\[
    \E{\mathcal{S}_i^{2q}}\leq \E{\mathcal{G}_i^{2q}}\text{ for $i\in \{1,2\}$}.
\]
We can therefore write
$$
\begin{aligned}
    \E{\bigg(\sum_{e^{1000}<\log p\leq e^{t_\theta}} X_p(0)-X_p(h)\bigg)^{2q}} &= \E{\big(\mathcal{S}_1+\mathcal{S}_2\big)^{2q}} \\
    &\leq \bigg( \E{\mathcal{S}_1^{2q}}^{1/2q}+\E{\mathcal{S}_2^{2q}}^{1/2q}\bigg)^{2q}\\
    &\leq \bigg( \E{\mathcal{G}_1^{2q}}^{1/2q}+\E{\mathcal{G}_2^{2q}}^{1/2q}\bigg)^{2q}.
\end{aligned}
$$
{Since $\mathcal{G}_1$ and $\mathcal{G}_2$ are Gaussian random variables, we can then factor out the $2q$-th moment of a standard Gaussian from the right-hand side to get
\begin{align*}
   \frac{(2q)!}{2^qq!}\Bigg(\bigg({\sum_{e^{1000}<\log p \leq e^{t_\theta}}\frac{\Delta(p)^2}{2p}}\bigg)^{1/2}+\bigg({\sum_{e^{1000}<\log p \leq e^{t_\theta}}\frac{\Delta(p)^2}{8p^2}}\bigg)^{1/2}\Bigg)^{2q}.
\end{align*}
By the bound $\Delta(p)^2=2-2\cos(|h|\log p)\leq |h|^2(\log p)^2$, this is no greater than
\[
   \ll \frac{(2q)!}{2^qq!}\left(|h|^2\sum_{e^{1000}<\log p\leq e^{t_\theta}} \frac{(\log p)^2}{p}\right)^{q}.
\]
The sum can then be estimated using a quantitative form of the prime number theorem (see, e.g., Theorem 6.9 in \cite{MontgomeryVaughanBook}), yielding
\begin{equation}\label{eq:error_moment}
    \ll \frac{(2q)!}{2^qq!} \left(e^{2t_\theta}|h|^2\right)^{q}.
\end{equation}
}

\section{Upper bound for the maximum}\label{upperboundsection}

In this section, we show how the recursive scheme in \cite{FHK1} can be adapted to prove Theorem \ref{UB}, emphasizing the ways in which the proof below differs from that of Theorem 2 in \cite{FHK1}. 
In the process, we will need three auxiliary lemmas (Lemma \ref{basedecoupling}, Lemma \ref{lem3} and Lemma \ref{lem4}), whose proofs are postponed to Section \ref{sectionaux}.

\subsection{Initial reductions}
Throughout this section, assume that $\theta\in (-1,0)$ (the $\theta=0$ case being the main result in \cite{FHK1}). To begin with, note that we may replace $|P_\tau(\theta)|$ in the statement of Theorem \ref{UB} by $e^{-S_{t_\theta}(0)}$, since $\log|P_\tau(\theta)|=-S_{t_\theta}(0)+O(1)$. By Lemma \ref{disczeta}, it also suffices to study the maximum on the discrete set
\begin{equation}\label{lattice}
    \mathcal{T}_t^\theta = e^{-t-100}\mathbb{Z}\cap[-2e^{-t_\theta},2 e^{-t_\theta}].
\end{equation}
Theorem \ref{UB} thus follows from the following. 

\begin{theorem}\label{UBdisc} Let $C$ be arbitrary and $\tau$ be uniformly distributed in $[T,2T]$. {Then for large enough $T$ and $2\leq y\leq C(t-t_\theta)/\log (t-t_\theta)$,}
    \[  
        \P{\max_{h\in \mathcal{T}_t^{\theta}} \log|\zeta_\tau(h)| > y+{(t-t_\theta)}-\frac{3}{4}\log(t-t_\theta) +S_{t_\theta}(0) } \ll ye^{-2y-y^2/(t-t_\theta )}.
    \]
\end{theorem}

\noindent

Consider the following subset of high points $h$ of $\zeta_\tau e^{-S_{t_\theta}}$
\[
    H^\theta (y) = \left\{h\in \mathcal{T}^\theta_t\,:\, |\zeta_\tau e^{-S_{t_\theta}}(h)|>e^{y}\frac{e^{t-t_\theta}}{(t-t_\theta)^{3/4}}\right\},
\]
as well as the set 
\begin{equation}\label{def:E}
        E^\theta(u)= \big\{h\in \mathcal{T}_t^\theta \,:\,S_{t_\theta}(0)-S_{t_\theta}(h)\in [u, u+1)\big\},\quad u\in \mathbb R,
\end{equation}
which controls how close $S_{t_\theta}(h)$ is to $S_{t_\theta}(0)$ for $h\in \mathcal{T}_t^\theta$. Note that by Proposition \ref{smallprimes}, the probability that $h\in E^\theta(u)$ should obey a Gaussian tail provided $u$ is not too large.

The following result bounds the probability that the set $H^\theta (y)\cap E^\theta(u)$ is non-empty. 
Heuristically, it gives an upper bound on the maximum of $\log |\zeta_\tau(h)|-S_{t_\theta}(h)$, and shows that this quantity behaves somewhat independently from the error $S_{t_\theta}(0)-S_{t_\theta}(h)$.
This result implies Theorem \ref{UBdisc} by essentially integrating over $u$, as we now show.

\begin{theorem}\label{auxUB}
    Let $C>0$ be an arbitrary constant. Then for large enough $t$, $2\leq y<C(t-t_\theta)/\log(t-t_\theta)$ and $|u|\leq C\sqrt{{(t-t_\theta})}$, we have
    \begin{align}
        \P{\exists h \in H^{\theta}(y)\cap E^\theta(u) }&\ll ye^{-2y-y^2/(t-t_\theta)}e^{-u^2/C_0},
    \end{align}
    for an absolute constant $C_0>0$.
    Furthermore, $\mathbb{P}\{\exists h\in H^{\theta}(y)\}\ll ye^{-2y-y^2/(t-t_\theta)}$.
\end{theorem}

\begin{proof}[Proof of theorem \ref{UBdisc} from \ref{auxUB}] 
Partitioning according to the value $u\in\mathbb{Z}$ of $S_{t_\theta}(0)-S_{t_\theta}(h)${, we have
\begin{align}\label{UBdecomposition}
    \mathbb{P}\bigg\{\max_{h\in\mathcal{T}_t^{\theta}}\log |\zeta_\tau(h)|>y&+m(t-t_\theta)+S_{t_\theta}(0)\bigg\}  \nonumber \\
    \ll \mathbb{P}\Big\{\exists h : |S_{t_\theta}(0)-&S_{t_\theta}(h)| > \sqrt{c\big(2y+{y^2}/{(t-t_\theta})\big)}\Big\} \\
    +\,\mathbb{P}\big\{\exists h\in H^{\theta}(y)\big\}&+\sum_{\substack{-\sqrt{c(2y+{y^2}/{(t-t_\theta}))}}<u<0} \mathbb{P}\bigg\{\exists h\in H^\theta\big(y+u\big) \cap E^\theta (u)\bigg\},\nonumber
\end{align}
where $c$ is the constant from Proposition \ref{smallprimes}.
The cutoff $\sqrt{c(2y+{y^2}/{(t-t_\theta}))}$ was chosen to make the first term in \eqref{UBdecomposition} $\ll e^{-2y-y^2/(t-t_\theta)}$ by Proposition \ref{smallprimes}, and we note that $1\leq y+u\leq y$ in this range provided $y>3(c+1)$ and $t$ is large enough (the former can be assumed without loss of generality). To bound what remains, we can therefore apply Theorem \ref{auxUB} to each term to get
\begin{align}
\begin{split}\label{convolutioncalculation}
    &\ll ye^{-2y-y^2/(t-t_\theta)} +\sum_{-\sqrt{c(2y+{y^2}/{(t-t_\theta))}}<u<0} (y+u)e^{-2(y+u)-(y+u)^2/(t-t_\theta)}e^{-u^2/C_0} \\
    &\ll ye^{-2y-y^2/(t-t_\theta)}+ye^{-2y-y^2/(t-t_\theta)} \sum_{u} e^{-f(u)}.
\end{split}
\end{align}
where $f(u)=2u+{(u^2+2yu)}/{(t-t_\theta)}+{u^2}/{C_0}$.
The claim then follows since
\[
    \sum_{u\in \mathbb{Z}} e^{-f(u)}\ll\int_{u\in \mathbb{R}}e^{-f(u)}\mathrm{d}u = \sqrt{\frac{\pi}{(t-t_\theta)^{-1}+C_0^{-1}}}\exp\Big(\frac{C_0(t-t_\theta+y)^2}{(t-t_\theta)(C_0+t-t_\theta)}\Big) \ll 1.
\]}
\end{proof}

\subsection{{The recursive scheme}}
We now prove Theorem \ref{auxUB}. For each $h\in \mathcal{T}_t^\theta$, recall from Section \ref{outlines} that $S_t(h)$ can be thought of as a random walk with $t$ steps, and that the random variables $(S_{k}(h))_{|h|\leq e^{-t_\theta}}$ are highly correlated for $k\leq t_\theta$. We will treat this range of primes separately (via the event $E^\theta(u)$), and adapt the recursive scheme from \cite{FHK1} to prove large deviation estimates for the shifted sums
\begin{equation}
{\mathscr{S}}_k(h):=S_{k+t_\theta}(h)-S_{t_\theta}(h), \quad 0\leq k\leq t-t_\theta.
\end{equation}
These serve as a proxy for $\log |\zeta_\tau e^{-S_{t_\theta}}(h)|$ when $k\approx t-t_\theta$.

To that end, we introduce the time scales{
\begin{equation}\label{eq:timescales}
        t_\ell := (t-t_\theta)-\mathfrak{s}\log_{\ell}(t-t_\theta)
\end{equation}
where $\log_{\ell} $ is the $\ell-$th iterated logarithm (assuming $\ell\geq 1$) and $\mathfrak{s}$ is a large constant (for this section, $\mathfrak{s}=10^6$ will do).} {The final time $\ell=\mathcal{L}$ is defined as the largest integer $\ell$ for which 
\begin{equation}\label{finalincrement}
        \exp(10^6((t-t_\theta)-t_\ell)^{10^5}e^{t_{\ell+1}+t_\theta})\leq T^{1/100}.
\end{equation}
While it is not immediately clear that such an $\mathcal{L}$ exists, a quick manipulation reduces this condition to $\log_\ell (t-t_\theta)>(10^8 \mathfrak{s}^{10^5})^{1/(\mathfrak{s}-10^5)}$($\approx 4.64$ for our choice of $\mathfrak{s}$), from which this is more apparent. This definition also ensures that $(t-t_\theta)-t_{\mathcal{L}}=O(1)$ and $\log_\mathcal{L} (t-t_\theta)>0$.}

The recursive scheme consists of partitioning the set $H^\theta(y)\cap E^\theta(u)$ of high points $h$ in terms of {\it decreasing good sets} $G(t_\ell)=G_\ell$, $0\leq \ell\leq \mathcal L$, defined in \eqref{eqn: G} below. The probability in Theorem \ref{auxUB} can then be decomposed as
\begin{equation}
\label{recursiveschemedecomposition}
\begin{aligned}
    \P{\exists h\in H^\theta(y)\cap E^\theta(u)} &\leq \P{\exists h \in H^\theta(y)\cap G_{0}^c\cap E^\theta(u)}\\
    &+\sum_{\ell=1}^{\mathcal{L}-1}\mathbb{P}\big\{\exists h\in H^\theta(y)\cap (G_\ell\setminus G_{\ell+1})\cap{E^\theta(u)}\}\\
    &+\P{\exists h \in H^\theta(y)\cap G_\mathcal{L}\cap{E^\theta(u)}},
\end{aligned}
\end{equation}
    The sets $G_\ell$ are built around the observation that the points $h\in H^\theta(y)$ at which $|\zeta_\tau e^{-S_{t_\theta}}|$ achieves a high value are those for which the (approximate) random walk $\mathscr{S}_k$ remains within the  corridor
\[
    (m(k)+L_y(k), m(k)+U_y(k)],
\]
consisting of upper and lower \textit{barriers} 
\begin{equation}\label{lowerbarrierl}
        L_y(k):= y -    \begin{cases*}
      \infty & if $1\leq k\leq \lceil y/4 \rceil$ \\ 
      20k & if $y/4 \leq k\leq t_0$ \\
      20\big((t-t_\theta)-k\big)        & if $t_0< k< t-t_\theta$
    \end{cases*} 
\end{equation}
\begin{align}\label{logbarrier}
    U_y(k):= y+    \begin{cases*}
      \infty & if $1\leq k\leq \lceil y/4 \rceil$ \\ 
      10^3\log k & if $y/4 \leq k\leq t_0$ \\
      10^3\log\big((t-t_\theta)-k\big)       & if $t_0< k<t-t_\theta$
    \end{cases*} 
\end{align}
(for $t_0=(t-t_\theta)/2$ the midpoint of $[t_\theta, t]$), centered at the linear interpolation from $y$ to $y+(t-t_\theta)-\tfrac{3}{4}\log(t-t_\theta)$:
\[  
    m(k) = \alpha k,\quad\quad \alpha = 1-\frac{3}{4}\frac{\log (t-t_\theta)}{t-t_\theta}.
\]
On these events ${G}_\ell$, one can obtain large deviation estimates for $\mathscr{S}_k$ by expressing the probabilities in question as moments of Dirichlet polynomials. The barriers restrict the length of these polynomials, so that their moments can be computed with small enough error via the mean-value theorem in Lemma \ref{MVMV} (see the proof of Lemmas \ref{lem3} and \ref{lem4}). Were it not for these restrictions, estimating the probability that an increment $\mathscr{S}_k-\mathscr{S}_{k-1}$ is atypically large could force one to consider Dirichlet polynomials of length much larger than $T$ in Lemma \ref{MVMV}.

The form of these barriers is motivated as follows.
If conditioned to end around $m (t-t_\theta)$, the approximate random walk ${\mathscr S}_{t-t_\theta}$ (of length $t-t_\theta$) should typically fluctuate like a Brownian bridge, i.e.\! like $(k\wedge (t-t_\theta-k))^{1/2}$. We may therefore take any power of $(k\wedge (t-t_\theta-k))$ greater than $1/2$ in the lower barrier $L_y(k)$, and we choose the power $1$ for simplicity. For the upper barrier, we keep with the branching random walk heuristic which suggests that one should (indeed must) pick a barrier lying {below} the typical fluctuations; we pick a large multiple of $\log (k\wedge (t-t_\theta-k))$ following Bramson \cite{Bramson}, and this is ultimately responsible for the appearance of ballot-theorem-type savings in Lemmas \ref{lem3} and \ref{lem4} (cf. Proposition \ref{prop:ballot}). For technical reasons, we do not impose any restrictions for small $k$ in \eqref{logbarrier} and \eqref{lowerbarrierl}.

\bigskip 

From these barriers, we define the decreasing sets
\begin{align*}
    &B_\ell = B_{\ell-1} \cap \left\{h\in \mathcal{T}_t^\theta \,:\, \mathscr{S}_k(h) \leq m(k)+U_y(k) \text{ for all } k\in(t_{\ell-1},t_\ell]\right\}\\
    &C_\ell = C_{\ell-1} \cap \left\{h\in \mathcal{T}_t^\theta \,:\, \mathscr{S}_k(h)>m(k)+L_y(k) \text{ for all } k\in(t_{\ell-1},t_\ell] \right\}
\end{align*}
for $\ell\geq 1$, and
\begin{align*}
    &B_0 = \left\{h\in \mathcal{T}_t^\theta \,:\, \mathscr{S}_k(h) \leq m(k)+U_y(k) \text{ for all } k\in[1,t_0]\right\}\\
    &C_0 = \left\{h\in \mathcal{T}_t^\theta \,:\, \mathscr{S}_k(h)>m(k)+L_y(k) \text{ for all } k\in[1,t_0] \right\}.
\end{align*}

To relate $\zeta$ to the partial sums $\mathscr{S}_k$, we follow  \cite{FHK1} and define the following mollifiers:
\begin{equation}\label{def:mollifiers}
        \mathcal{M}_\ell (h) := \sum_{{\substack{p|m\Rightarrow \log_2 p\ \in(t_{\ell-1}+t_\theta,t_\ell+t_\theta]\\ \Omega_\ell(m)\leq (t_\ell-t_{\ell-1})^{10^5}}}}\frac{\mu(m)}{m^{1/2+i\tau+ih}}, 
\end{equation}
for $\ell\geq 0$, adopting the convention $\mathcal{M}_{-1}\equiv 1$ and $t_{-1}=1000-t_\theta$ so that $\mathcal{M}_{0}$ is a sum over all primes $\exp(1000)\leq p\leq \exp(e^{t_\theta+t_0})$. 
Here, $\Omega_\ell(m)$ denotes the number of distinct prime factors of $m$ in the interval $(\exp(e^{t_{\ell-1}+t_\theta}), \exp(e^{t_{\ell}+t_\theta})]$.
The goal of the mollifier $\mathcal M_\ell(h)$ is to act as a ``statistical'' inverse for the contribution to $\zeta$ coming from the primes in the range $(\exp(e^{t_{\ell-1}+t_\theta}), \exp(e^{t_\ell+t_\theta})]$;
heuristically, one would want to take
$$
\prod_{\log_2 p\ \in(t_{\ell-1}+t_\theta,t_\ell+t_\theta]}\Big(1-\frac{1}{p^{1/2+i(\tau+h)}}\Big)=\sum_{{\substack{p|m\Rightarrow \log_2 p\ \in(t_{\ell-1}+t_\theta,t_\ell+t_\theta]}}} \frac{\mu(m)}{m^{1/2+i\tau+ih}},
$$
but this Dirichlet polynomial has length far exceeding $T$, rendering Lemma \ref{MVMV} inapplicable. The restriction on $\Omega_\ell$ in the definition of $\mathcal{M}_\ell$ truncates this polynomial, shortening it to an admissible length.
(Note that for a typical integer $n$, $\Omega_\ell(n)$ should roughly have size $t_\ell-t_{\ell-1}$; the bound $(t_\ell-t_{\ell-1})^{10^5}$ therefore ensures that most integers are considered
but none of them are too large.) We will also need the following refinement of $\mathcal M_\ell$ up to $k+t_\theta$ for $k\in(t_{\ell-1},t_\ell]$:
\[
\mathcal{M}^{(k)}_{\ell-1}(h)=\sum_{{\substack{p|m\Rightarrow \log_2 p \in(t_{\ell-1}+t_\theta,k+t_\theta])\\ \Omega_\ell(m)\leq (t_\ell-t_{\ell-1})^{10^5}}}}\frac{\mu(m)}{m^{1/2+i\tau+ih}}, \quad \ell\geq -1.
\]
 In this notation, $\mathcal{M}_{\ell-1}^{(t_\ell)}=\mathcal{M}_\ell$  and the product $\mathcal{M}_{-1}\cdots\mathcal{M}_{\ell-1}\mathcal{M}_{\ell-1}^{(k)}$ will serve as an approximant for $ e^{-S_{t_\theta+k}}$ (see Lemma \ref{mollifying}).

 To ensure that this approximation holds with small enough error, we introduce the complex version of the shifted partial sums $\mathscr{S}_k$
 \begin{equation}\label{curlySdefinition}
     \widetilde{\mathscr{S}}_{k}(h)=\sum_{e^{1000+t_\theta}< \log p\leq e^{k+t_\theta}} \frac{1}{p^{1/2+i(\tau+h)}}+\frac{1}{2}\frac{1}{p^{1+2i(\tau+h)}}, \quad k\leq t-t_\theta,
\end{equation}
and restrict to points $h$ belonging to the following sets:
\begin{align*}
    &A_\ell = A_{\ell-1} \cap \left\{ h\in \mathcal{T}_t^\theta \,:\,|\widetilde{\mathscr{S}}_{k}(h)-\widetilde{\mathscr{S}}_{t_{\ell-1}}(h)|\leq 10^3(t_\ell-t_{\ell-1}) \text{ for all } k\in(t_{\ell-1},t_\ell] \right\} \\
    &D_\ell = D_{\ell-1} \cap \big\{h\in \mathcal{T}_t^\theta \,:\, |\zeta_\tau e^{-S_{k+t_\theta}}(h)|\leq c_\ell|\zeta_\tau \mathcal{M}_{-1}\cdots \mathcal{M}_{\ell-1}^{(k)}(h)|+e^{-10^4((t-t_\theta)-t_{\ell-1})}\\ &\quad\quad\quad\quad\quad\quad\quad\quad\quad\quad\quad\quad\quad\quad\quad\quad\quad\quad\quad\quad\quad\quad\quad\quad\quad\quad\text{ for all } k\in (t_{\ell-1},t_\ell]\big\},
\end{align*}
where 
\begin{equation}\label{def:cl}
c_{\ell}:=\prod_{i=0}^{\ell}(1+e^{-(t_{i-1}+t_\theta)}),
\end{equation}
and $A_{-1}=D_{-1}=[-2\log^\theta T,2\log^\theta T]$.

Putting everything together, we define the decreasing good sets to be 
\begin{equation}
\label{eqn: G}
G_\ell=A_\ell\cap B_{\ell}\cap C_\ell \cap D_\ell, \quad \ell\geq -1.
\end{equation}
While we only consider $\theta\in (-1,0)$, note that at $\theta=0$, $\mathscr{S}_k=S_k$ for every $k$ and these definitions reduce to their counterparts in \cite{FHK1}. 

\bigskip 

We evaluate the probabilities in \eqref{recursiveschemedecomposition} using the following propositions. 
\begin{proposition}\label{base}
    Fix $\theta\in (-1,0)$. Let $C>0$ be arbitrary, $1\leq y<C(t-t_\theta)/\log (t-t_\theta)$ and $|u|\leq C\sqrt{t-t_\theta}$. Then there exist constants $K,c>0$ such that
    \[  
        \P{\exists h \in H^\theta(y)\cap G_{0}^c\cap E^\theta(u)} \leq Kye^{-2y-y^2/(t-t_\theta)}e^{-u^2/c}.
    \]
\end{proposition}
\begin{proposition}\label{induction}
    Fix $\theta\in (-1,0)$. Let $C>0$ be arbitrary, $1\leq y<C(t-t_\theta)/\log (t-t_\theta)$, $|u|\leq C\sqrt{t-t_\theta}$, and $\ell\geq 0$ satisfying (\ref{finalincrement}). Then there exist constants $K,c>0$ such that uniformly in all such $\ell$, 
    \begin{align*}
        \mathbb{P}\Big\{\exists &h\in H^\theta(y)\cap (G_\ell\setminus G_{\ell+1})\cap{E^\theta(u)}\Big\}\leq K\frac{ye^{-2y-y^2/(t-t_\theta)}}{\log_{\ell+1}(t-t_\theta)}\cdot e^{-u^2/c}.
    \end{align*}
    
\end{proposition}
\begin{proposition}\label{last}
        Fix $\theta\in (-1,0)$. Let $C>0$ be arbitrary, $1\leq y<C(t-t_\theta)/\log (t-t_\theta)$, and $|u|\leq C\sqrt{t-t_\theta}$. Then there exist constants $K,c>0$ such that
    \[  
        \P{\exists h \in H^\theta(y)\cap G_\mathcal{L}\cap{E^\theta(u)}
        } \leq K ye^{-2y-y^2/(t-t_\theta)}e^{-u^2/c}.
    \]
\end{proposition}
\begin{proof}[Proof of theorem \ref{auxUB}]\label{proofusingscheme}
Using Propositions \ref{base}, \ref{induction} and \ref{last} in the decomposition of Equation \eqref{recursiveschemedecomposition}, we get
    \begin{align*}
    \mathbb{P}\Big\{&\exists h\in H^{\theta}(y)\cap E^{\theta}(u)\Big\} 
        \ll ye^{-2y-y^2/(t-t_\theta)}e^{-u^2/c} + \sum_{1\leq \ell \leq \mathcal{L}}\frac{ye^{-2y-y^2/(t-t_\theta)}e^{-{u^2/c}}}{\log_{\ell}(t-t_\theta)}, 
\end{align*}
for some $c>0$, and the claim then follows from the fact that $\sum_{1\leq \ell \leq \mathcal{L}}\frac{1}{\log_\ell (t-t_\theta)}=O(1)$. {That this sum is bounded can be seen by the change of variables $n=\mathcal{L}-\ell$, which gives 
\begin{equation}
\label{eqn: summable}
    \sum_{1\leq \ell \leq \mathcal{L}}\frac{1}{\log_\ell (t-t_\theta)}=\frac{1}{\log_\mathcal{L}(t-t_\theta)}+\sum_{1\leq n \leq \mathcal{L}-1} \frac{1}{\exp_n\hspace{-2px}\big(\log_\mathcal{L}(t-t_\theta)\big)}
\end{equation}
where $\exp_n$ denotes the exponential map iterated $n$ times.
Recalling that $\log_\mathcal{L}(t-t_\theta)$ is a constant by definition of $\mathcal{L}$ (cf.~ Equation \eqref{finalincrement}), it is clear that this sum is rapidly convergent in $n$.} 
\end{proof}

These propositions are proved in the following three subsections. For simplicity, we will use the abbreviation
 \begin{equation}\label{def:g}
    g(y,\theta)=ye^{-2y-y^2/(t-t_\theta)}
\end{equation}
for the desired tail bound in what follows.

\subsection{Proof of Proposition \ref{base}}  {The proof is based on the fact that the Dirichlet polynomials $\mathscr{S}_{k}$ and $S_{t_\theta}$ behave to some extent like independent Gaussians, being supported on disjoint ranges of primes. This is for instance visible in the following lemma, required for the proof, whose proof is deferred to Section \ref{sectionaux}.}
\begin{lemma}\label{basedecoupling}
    Let $C>0$ be arbitrary, $y<C(t-t_\theta)/\log(t-t_\theta)$ and $|u|< C\sqrt{t-t_\theta}$. Suppose that $\mathcal{Q}$ is Dirichlet polynomial of length $N\leq \exp(\tfrac{1}{100}e^{t})$, supported on integers whose prime factors are all $>\exp(e^{t_\theta})$. Then there exist constants $C',c>0$ such that
\begin{align*}
    &\E{\max_{|h|\leq 2\log^\theta T} |\mathcal{Q}(\tfrac{1}{2}+i\tau+ih)|^2\cdot \mathbf{1}\big(h\in E^\theta (u)\big)} \\
    &\ll \left(e^{-t_\theta}\log N +C'\right) \E{|\mathcal{Q}(\tfrac{1}{2}+i\tau)|^2} e^{-u^2/c}.
\end{align*}
\end{lemma}
We now prove Proposition \ref{base}, beginning with a union bound to get
 \begin{align}\label{firstunionbound}
     &\mathbb{P}\big\{\exists h\in H^\theta(y)\cap{E^\theta(u)}\cap G_0^c\big\}\nonumber \\ 
     &\leq \P{\exists h\in A_0^c}+ \P{\exists h \in B_0^c
     \cap{E^\theta(u)}} \\ 
     &+\P{\exists h \in D_0^c \cap A_0} +\P{\exists h \in C_0^c\cap A_{0}\cap E^{\theta}(u)}. \nonumber
 \end{align}
{Recall that $t_0=(t-t_\theta)/2$ and that $g(y,\theta)=ye^{-2y-y^2/(t-t_\theta)}$ (cf.~Equation \eqref{def:g}).} By a union bound on $h$ and $k\leq t_0-t_{-1}=(t+t_\theta)/2$, and the Gaussian tail bound in Equation \eqref{Gaussiantail}, we have
\[
    \P{\exists h \in A_0^c} \ll e^{t(1+\theta)}(t+t_\theta)\exp(-10^2(t_0+t_\theta)) \ll g(y,\theta)e^{-u^2/c},
\]
where the rightmost inequality holds in the assumed ranges of $y$ and $u$.

For the second term in Equation \eqref{firstunionbound}, note that if there exists a $k\in[1,t_0]$ and $h$ for which
\[
    \mathscr{S}_k(h)>m(k)+U_y(k),
\]
{then $k$ must be greater than $y/4$ (the right-hand side being infinite otherwise). It follows that if such a $k$ exists, then}
\[
    \sum_{y/4<k\leq t_0} \max_{|h|\leq 2\log^\theta T} \frac{|\mathscr{S}_k(h)|^{2q_k}}{(y+k+10\log k)^{2q_k}} \geq 1, \mbox{ for any sequence $(q_k)_k \subseteq \mathbb{Z}_{\geq 1}$}
\]
and we thus have the bound
 \begin{align}\label{B0}
      \P{\exists h \in B_0^c\cap E^{\theta}(u)} 
      &\leq  \sum_{y/4 < k\leq t_0}  \E{\max_{|h|\leq 2\log^\theta T} \frac{|\mathscr{S}_k(h)|^{2q_k}}{(y+k+10\log  k)^{2q_k}} \mathbf{1}\big(h\in E^\theta (u)\big)}.
 \end{align}
Picking $q_k = \lceil (k+10\log k+y)^2/k\rceil$, we find that $\mathscr{S}_k(h)^{2q_k}$ is a Dirichlet polynomial of length $\exp(2q_ke^{k+t_\theta}) (\leq \exp(\tfrac{1}{100}e^{t})$ for large enough $t$). We can thus apply Lemma \ref{basedecoupling} to each summand in \eqref{B0} { to get
\begin{align*}
    \P{\exists h \in B_0^c\cap E^{\theta}(u)}
     &\ll\sum_{y/4< k \leq t_0}e^{-u^2/c}e^k q_k \E{\frac{|\mathscr{S}_k(0)|^{2q_k}}{(y+k+10\log k)^{2q_k}}}.
\end{align*}
Using the Gaussian moment estimate in Lemma \ref{gaussianmoments}, followed by Stirling's approximation as in the proof of Corollary \ref{cor:tails} (cf.~ Equation \eqref{eq:stirling}), we have that
\begin{equation}\label{eq:basestirling}       \E{\frac{|\mathscr{S}_k(0)|^{2q_k}}{(y+k+10\log k)^{2q_k}}} \ll e^{-q_k}
\end{equation}
for each $y/4<k\leq t_0$.
The sum is thus bounded by
 \begin{align*}
     &\ll e^{-u^2/c}\sum_{y/4<k\leq t_0} e^{k}\frac{(k+y)^2}{k} \exp\big(-(k+y+10\log k)^2/k\big) \\
     &\ll  e^{-u^2/c}e^{-y^2/(t-t_\theta)-2y}\sum_{y/4<k\leq t_0} e^{k}\frac{(k+y)^2}{k} e^{-k-20\log k}
     \\
     &\ll e^{-u^2/c}g(y,\theta)  \sum_{y/4<k\leq t_0}(k+y^2k^{-1})k^{-20}\\
     &\ll e^{-u^2/c}g(y,\theta),
 \end{align*}
 where in the second line, we used the fact that $e^{-y^2/k}\leq e^{-y^2/(t-t_\theta)}$ uniformly over $k\leq t_0$.
  }
  
We use a similar approach to deal with $C_0^c$, with the bound
 \begin{align*}
     \P{\exists h \in C_0^c\cap E^\theta(u)} &\leq \sum_{y/4 < k\leq t_0} \E{\max_{|h|\leq \log^\theta T}\frac{|\mathscr{S}_k(h)|^{2q_k}}{\big(20k-y\big)^{2q_k}} \mathbf{1}(h\in E^\theta(u))},
 \end{align*}
holding for any $q_k\geq 1$. By picking $q_k=\lceil (20k-y)^2/k\rceil$ (making $\mathscr{S}_k(h)^{2q_k}$ a Dirichlet polynomial of length $\exp(2q_ke^{k+t_\theta})\leq\exp(\tfrac{1}{100}e^t)$) and applying Lemma \ref{basedecoupling}, this is
 \[
     \ll e^{-u^2/c}\sum_{y/4 < k\leq t_0} (q_ke^{k}) \E{\frac{|\mathscr{S}_k(0)|^{2q_k}}{\big(20k-y\big)^{2q_k}}}.
 \]
Using the Gaussian moment estimate in Lemma \ref{gaussianmoments} for $\mathscr{S}_k$ and Stirling's approximation as we did for Equation \eqref{eq:basestirling} gives 
 \begin{align*}
     &\ll e^{-u^2/c}\sum_{y/4 < k \leq t_0} q_ke^{k} \exp\bigg(-\frac{(20k-y)^2}{k}\bigg).
 \end{align*}
 By the estimate $q_k\ll k$ in $y/4<k\leq t_0$ and elementary manipulations, this is
 \begin{align*}
     &\ll e^{-u^2/c}\sum_{y/4 < k \leq t_0} k e^{k} \exp\bigg(-400k-\frac{y^2}{k}+{40y}\bigg)\\
     &\ll e^{-u^2/c}\exp\left({-2y-y^2/(t-t_\theta)}\right).
 \end{align*}

 We conclude by proving that $\P{\exists h \in D_0^c\cap A_0}$ in Equation \eqref{firstunionbound} is $\ll e^{-100t}$, from which the proposition follows {since $e^{-100t}$ is much smaller than  $g(y,\theta)e^{-u^2/c}$ in our range of $y$ and $u$}. This requires following lemma, which justifies the use of mollifiers as a proxy for $e^{-(\mathscr{S}_k-\mathscr{S}_{t_{\ell-1}})}$.
 \begin{lemma}[Lemma 23 in \cite{FHK1}] \label{mollifying}
     Let $\ell\geq 0$, $k\in (t_{\ell-1},t_\ell]$ and suppose that $h\in A_\ell$. Then
     \[
        e^{-(\mathscr{S}_{k}(h)-\mathscr{S}_{t_{\ell-1}}(h))} \leq (1+e^{-(t_{\ell-1}+t_\theta)})|\mathcal{M}_{\ell-1}^{(k)}(h)|+e^{-10^5(t_{\ell}-t_{\ell-1})}.
     \]
 \end{lemma}
 \begin{proof}
     The proof is identical to that of Lemma 23 in \cite{FHK1}, noting that their argument holds for our choice of times $t_{\ell}$ and mollifiers $\mathcal{M}_\ell^{(k)}$.
 \end{proof}
 {
 Recall that $t_{-1}=1000-t_\theta$ by convention and $c_0=1+e^{-1000}$ by definition (cf.~\eqref{def:cl}).
 This implies that that $\mathscr{S}_k-\mathscr{S}_{t_{-1}}=S_{t_\theta+k}$. 
 With this notation, if $h\in \mathcal{T}_t^\theta\cap A_0$ is such that $|\zeta_\tau(h)|\leq e^{100t}$, then the previous lemma implies that
 \begin{align*}
     |(\zeta_\tau e^{-S_{t_\theta+k}})(h)|&\leq c_0|\zeta_\tau\mathcal{M}_{-1}^{(k)}(h)|+e^{100t-10^5(t_0-t_{-1})}\\ 
     &\leq c_0|\zeta_\tau\mathcal{M}_{-1}^{(k)}(h)|+e^{-10^4(t-t_\theta-t_{-1})}.
 \end{align*}
 It therefore suffices to show that $\P{\exists h:|\zeta_\tau(h)|> e^{100t}}\ll g(y,\theta)e^{-u^2/c}$. This follows immediately from a union bound over $h\in \mathcal{T}_t^\theta$, followed by Chebyshev's inequality, using a classical second moment bound for the zeta function (Lemma \ref{secondmomentzeta}):
 \[
    \P{\exists h\in \mathcal{T}_t^\theta:|\zeta_\tau(h)|> e^{100t}} \ll e^{t(1+\theta)}\cdot \frac{\E{|\zeta_\tau(h)|^{2}}}{e^{200t}}\ll e^{-199t+t}.
 \] 

 }
 
\subsection{Proof of Proposition \ref{induction}}

We follow the strategy behind the proof of Proposition 2 in \cite{FHK1}, making the necessary changes to consider $\mathscr{S}_k$ instead of $S_k$.
{
We will need the following refinements of $A_\ell, B_\ell, C_\ell$ and $D_{\ell}$ to study increments of $\mathscr{S}_k$ for $k\in (t_{\ell},t_{\ell+1}]$:
\begin{align*}
    A_{\ell}^{(k)}&=A_\ell \cap \{h\in \mathcal{T}_t^\theta\,:\, |\widetilde{\mathscr{S}}_k(h)-\widetilde{\mathscr{S}}_{t_\ell-1}(h)|\leq 10^3(t_\ell-t_{\ell-1}) \text{ for all } t_\ell<j\leq k\}\\
    B_{\ell}^{(k)}&=B_\ell \cap \{h\in \mathcal{T}_t^\theta\,:\, \mathscr{S}_j(h)\leq m(j)+U_y(j) \text{ for all } t_\ell<j\leq k\}\\
    C_{\ell}^{(k)}&=C_\ell \cap \{h\in \mathcal{T}_t^\theta\,:\,\mathscr{S}_j(h)>m(j)+L_y(j)\text{ for all } t_\ell<j\leq k \}\\
    D_{\ell}^{(k)}&=D_\ell \cap \{h\in \mathcal{T}_t^\theta\,:\, |\zeta_\tau e^{-S_{k+t_\theta}}(h)|\leq c_{\ell+1}|\zeta_\tau \mathcal{M}_{-1}\cdots \mathcal{M}_{\ell}\mathcal{M}_{\ell}^{(k)}(h)|\\
    &\quad\quad\quad\quad\quad\quad+e^{-10^4((t-t_\theta)-t_{\ell-1})}\text{ for all } t_{\ell}<j \leq k\}
\end{align*}
where $c_{\ell+1}=\prod_{i=0}^{\ell+1}(1+e^{-t_{i-1}+t_\theta})$, and $A_{\ell}^{(t_{\ell+1})}=A_{\ell+1}$ (similarly for $B_{\ell+1}, C_{\ell+1}$ and $D_{\ell+1}$).

\bigskip

The proof relies on two lemmas, whose proofs are again deferred to Section \ref{sectionaux}.
The intuition behind them is that expectation of a product of random variables depending on primes in  $(2,t_\theta]$, $(t_\theta, t_\theta+k]$ and $(t_\theta+k, t]$, respectively, should decouple into the product of the corresponding expectations. 

In the following, the random variables in question are $\mathbf{1}(E^\theta(u))$, an indicator involving the sums $(\mathscr{S}_j)_j$ for $j\in (1, k]$, and the maximum over $|h|\leq 2\log^\theta T$ of a (short enough) Dirichlet polynomial supported on primes  $p>\exp(e^{t_\theta+k})$.
\begin{lemma}\label{lem3}
    Let $\ell\geq 0$ satisfy (\ref{finalincrement}). Suppose that $\mathcal{Q}$ is Dirichlet polynomial of length $N\leq \exp(\tfrac{1}{100}e^{t})$ supported on integers all of whose prime factors are $>\exp(e^{k+t_\theta})$. Let $C>0$ be arbitrary, $1\leq C(t-t_\theta)/\log (t-t_\theta)$ and {$|u|< C\sqrt{t-t_\theta}$}, $L_y(k)<w-m(k)<U_y(k)$, we have
    \begin{align}\label{eq:4.8}
        \E{\max_{|h|\leq 2\log^\theta T}|\mathcal{Q}(\tfrac{1}{2}+i\tau+ih)|^2\cdot \mathbf{1}\big({h\in B_\ell^{(k)}\cap C_\ell^{(k)}\cap E^\theta(u)\text{, } \mathscr{S}_k(h)\in (w,w+1]}\big)} \nonumber \\
        \ll\E{|\mathcal{Q}(\tfrac{1}{2}+i\tau)|^2}\left(e^{-t_\theta}\log N+e^{k}\big((t-t_\theta)-k\big)^{800}\right)  \nonumber\\
         e^{-u^2/c}\times\frac{y(U_y(k)-w+m(k)+2)e^{-w^2/k}}{k^{3/2}}.
    \end{align}
\end{lemma}
\noindent The last term in \eqref{eq:4.8} arises from the application of a ballot theorem (Proposition \ref{prop:ballot}), which is used to estimate the probability that the partial sum $\mathscr S_j$ remains below $U_y(j)$ at every $j\leq k$, and ends at $\mathscr S_k\approx w$. The same term would arise were $\mathscr{S}_k$ a genuine Gaussian random walk with variance-$1/2$ increments.

We also require a variant of Lemma \ref{lem3} that  controls $\log|\zeta_\tau(h)|-S_{k+t_\theta}(h)$, which represents the contribution to $\log|\zeta_\tau(h)|$ from primes $p>\exp(e^{k+t_\theta})$. This can be achieved using a \textit{twisted fourth moment} estimate for the Riemann zeta function (see Lemma \ref{twisted}), which in this cases bounds  the expectation of the product of $|\zeta_\tau \mathcal{M}_{-1}\cdots\mathcal{M}_\ell\mathcal{M}_\ell^{(k)}(h)|^4$ with Dirichlet polynomials supported on $p <\exp(e^{k+t_\theta})$. Here, too, we obtain decoupling between random variables supported on disjoint ranges of primes, and we note that a second moment would not be sufficient for our purposes.
\begin{lemma}\label{lem4}
Let $\ell\geq 0$ satisfy (\ref{finalincrement}). Let $k\in[t_{\ell},t_{\ell+1}]$, and let $\eta(m)$ be a sequence of complex coefficients with $|\eta(m)|\ll\exp(\tfrac{1}{1000}e^t)$ for all $m\geq 1$. Let 
\[
    \mathcal{Q}_\ell^{(k)}(h):=\sum_{\substack{p|m\implies p \in(\exp(e^{t_\ell+t_\theta}), \exp(e^{k+t_\theta})]\\ \Omega_{\ell+1}(m)\leq (t_{\ell+1}-t_\ell)^{10^4} }} \frac{\eta(m)}{m^{1/2+i\tau+ih}}.
\]
Lastly, let $C>0$ be arbitrary, $1\leq y< C(t-t_\theta)/\log(t-t_\theta)$ and {$|u|< C\sqrt{t-t_\theta}$}. Then for any $h\in[-2\log^\theta T, 2\log^\theta T]$, and $L_y(t_\ell)<w-m(k)<U_y(t_\ell)$, 
\begin{align*}
\E{|(\zeta_\tau\mathcal{M}_{-1}\cdots\mathcal{M}_{\ell}\mathcal{M}_{\ell}^{(k)}(h)|^4\cdot |\mathcal{Q}_\ell^{(k)}(h)|^2\cdot\mathbf{1}\left(h\in B_\ell\cap C_\ell\cap E^\theta(u),\, \mathscr{S}_{t_\ell}\in(v,v+1]\right)} \\
\ll e^{4((t-t_\theta)-k)} \E{|\mathcal{Q}_\ell^{(k)}(h)|^2} y(U_y(t_\ell)-v+m(t_\ell)+2)e^{-v^2/t_\ell}(t-t_\theta)^{-3/2}e^{-u^2/c}.
\end{align*}
\end{lemma}
}
We are now ready to prove Proposition \ref{induction}. To begin with, a union bound gives
\begin{align*}
    &\P{\exists h \in H^\theta (y)\cap G_\ell\cap E^\theta(u)} \\
    &\leq \P{\exists h \in H^\theta (y)\cap G_\ell \cap G_{\ell+1}^c\cap E^\theta(u)} + \P{\exists h \in H^\theta (y)\cap G_\ell\cap G_{\ell+1}\cap E^\theta(u)} 
\end{align*}
and the first term on the right-hand side is bounded above by
\begin{align}
    &\P{\exists h\in A_{\ell+1}^c\cap H^\theta(y)\cap G_\ell\cap E^\theta(u)} \label{increments} \\
    &+ \P{\exists h \in D_{\ell+1}^c\cap A_{\ell+1}\cap H^{\theta}(y)\cap G_\ell\cap E^\theta(u)} \label{mollifiers}\\
    &+\P{\exists h \in C_{\ell+1}^c\cap D_{\ell+1}\cap A_{\ell+1}\cap H^\theta (y)\cap G_\ell\cap E^\theta(u)} \label{lowerext}\\
    &+\P{\exists h\in B_{\ell+1}^c\cap C_{\ell+1}\cap A_{\ell+1}\cap H^\theta(y)\cap G_\ell\cap E^\theta(u)}\label{upperext}.
\end{align}
We will show that each of these probabilities is $$\ll e^{-u^2/c}\frac{g(y,\theta)}{\big(\log_{\ell+1}(t-t_\theta)\big)^{100}}.$$
(The dominant term is \eqref{upperext}.)
\subsubsection{Bound on (\ref{increments})}

By definition of $A_{\ell+1}^c$, if there is a $k\in(t_\ell, t_{\ell+1}]$ and an $h$ such that $|\widetilde{\mathscr{S}}_k(h)-\widetilde{\mathscr{S}}_{t_\ell}(h)|>10^3((t-t_\theta)-t_\ell)$, then 
\[
    \sum_{k\in(t_\ell,t_{\ell+1}]} \max_{|h|\leq \log^\theta T} \frac{|\widetilde{\mathscr{S}}_k(h)-\widetilde{\mathscr{S}}_{t_\ell}(h)|^{2q}}{(10^3((t-t_\theta)-t_\ell))^{2q}} \geq 1, \quad \text{for all $q\geq 1$.}
\]
This implies the bound
\begin{align}\label{Ainduction}
    &\P{\exists h \in A_{\ell+1}^c \cap G_\ell \cap E^\theta(u)} \nonumber \\ &\leq \sum_{k\in(t_\ell,t_{\ell+1}]} \E{\max_{|h|\leq \log^\theta T} \frac{|\widetilde{\mathscr{S}}_k(h)-\widetilde{\mathscr{S}}_{t_\ell}(h)|^{2q}}{(10^3((t-t_\theta)-t_\ell))^{2q}}\mathbf{1}{\big(h\in G_\ell\cap E^\theta(u)\big)}}.
\end{align}
Picking $q = \lfloor 10^6((t-t_\theta)-t_\ell)^2/(k-t_\ell)\rfloor$ in Equation (\ref{Ainduction}), we note that $|\widetilde{\mathscr{S}}_k(h)-\widetilde{\mathscr{S}}_{t_\ell}(h)|^{2q}$ is a Dirichlet polynomial of length $\leq \exp(2qe^{k+t_\theta}) \leq \exp(2\cdot10^6((t-t_\theta)-t_\ell)^2e^{t_{\ell+1}+t_\theta})\leq \exp(e^{t}/100)$ (by the assumption that $\ell$ satisfies (\ref{finalincrement})). We can therefore bound each summand using Lemma \ref{lem3} to get
\[
    \ll e^{-u^2/c}g(y,\theta)\sum_{k\in (t_\ell, t_{\ell+1}]} \big(q+((t-t_\theta)-t_\ell)^{800}\big)e^{100((t-t_\theta)-t_\ell)}\E{\frac{|\widetilde{\mathscr{S}}_k(h)-\widetilde{\mathscr{S}}_{t_\ell}(h)|^{2q}}{\big(10^3((t-t_\theta)-t_\ell)\big)^{2q}}}.
\]
{For our choice of $q$, we can apply a Gaussian moment estimate for the rightmost term (Lemma \ref{complexgaussian}) along with Stirling's approximation as in the proof of Corollary \ref{cor:tails} to get
\begin{align*}
    \E{\frac{|\widetilde{\mathscr{S}}_k(h)-\widetilde{\mathscr{S}}_{t_\ell}(h)|^{2q}}{\big(10^3((t-t_\theta)-t_\ell)\big)^{2q}}} &\ll \sqrt{q}e^{-q}\bigg(\frac{q(k-t_\ell)}{10^3\big((t-t_\theta)-t_\ell\big)^2}\bigg)^q \\
    &\ll \sqrt{q}e^{-10^6((t-t_\theta)-t_\ell)^2/(k-t_\ell)}\\
    &\ll \sqrt{q}e^{-10^6((t-t_\theta)-t_\ell)},
\end{align*}
uniformly in $k\in(t_\ell,t_{\ell+1}]$. Plugging this back into the sum, we have
\[
    \ll e^{-u^2/c}g(y,\theta)\sum_{k\in(t_\ell,t_{\ell+1}]} \sqrt{q}(q+((t-t_\theta)-t_\ell))^{800} e^{(100-10^6)((t-t_\theta)-t_\ell)},
\]
and summing over $k$ then gives the desired bound:
\[
    \ll g(y,\theta)e^{-u^2/c} \big(\log_{\ell-1}(t-t_\theta)\big)^{-1}.
\]}
\subsubsection{Bound on (\ref{mollifiers})}\label{boundwithmollifiers}
{
For any $h\in A_{\ell+1}\cap D_\ell$ we have 
\begin{align}\label{eq:ind_moll}
    |\zeta_\tau e^{-S_{k+t_\theta}}(h)|&=
    |\zeta_\tau e^{-{S}_{t_\theta+t_\ell}}(h)|e^{-(\mathscr{S}_k-\mathscr{S}_{t_\ell})(h)} \nonumber\\
    &\leq c_\ell |\zeta_\tau \mathcal{M}_{-1}\cdots \mathcal{M}_\ell(h)|e^{-(\mathscr{S}_k-\mathscr{S}_{t_\ell})(h)}+e^{-10^3((t-t_\theta)-t_{\ell-1})},
\end{align}
using the definition of $D_{\ell}$ and $A_{\ell+1}$ to bound $|\zeta_\tau e^{-\mathscr{S}_{t_\ell}}(h)|$ and $e^{-(\mathscr{S}_k-\mathscr{S}_{t_\ell})(h)} $ respectively. To bound \eqref{mollifiers}, it then suffices to show that the following estimates hold with sufficiently high probability for each $h\in A_{\ell+1}\cap D_\ell$ and $k\in(t_\ell,t_{\ell+1}]$:
\begin{equation}\label{lmoll}
    |(\zeta_\tau \mathcal{M}_{-1}\cdots \mathcal{M}_{\ell})(h)|<e^{10^3((t-t_\theta)-t_\ell)},
\end{equation}
\begin{equation}\label{incmoll}
    |e^{-(\mathscr{S}_k-\mathscr{S}_{t_\ell})}(h)|\leq (1+e^{-t_\ell})|\mathcal{M}_\ell^{(k)}(h)|+e^{-10^5(t_{\ell+1}-t_\ell)},
\end{equation}
since applying them in succession to the right-hand side of \eqref{eq:ind_moll} implies that
\[
    |\zeta_\tau e^{-S_k+t_\theta}(h)|\leq c_{\ell+1}|\zeta_\tau \mathcal{M}_{-1}\cdots \mathcal{M}_\ell\mathcal{M}_\ell^{(k)}(h)|+e^{-10^4((t-t_\theta)-t_\ell)}.
\]
}

By Lemma \ref{mollifying}, Equation \eqref{incmoll} holds pointwise for each $h\in A_{\ell+1}$. As for \eqref{lmoll}, 
{a union bound on $h$ bounds the probability of its complement by}
\begin{align*}
    &\ll\sum_{h\in \mathcal{T}_t^\theta} \P{ |(\zeta_\tau \mathcal{M}_{-1}\cdots \mathcal{M}_\ell)(h)|\geq e^{10^3((t-t_\theta)-t_\ell)}, h\in G_\ell\cap E^{\theta}(u)} \\
    &\ll e^{-4\cdot 10^3((t-t_\theta)-t_\ell)}e^{t-t_\theta}\E{|(\zeta_\tau \mathcal{M}_{-1}\cdots \mathcal{M}_\ell)(0)|^4\cdot \mathbf{1}\big(0\in G_\ell\cap E^\theta(u)\big)}.
\end{align*}
Using Lemma \ref{lem4} with $\mathcal{Q}\equiv 1$ and $k=t_\ell$ to bound this expectation, we find that this is
\begin{equation}\label{2.5end}
  e^{-u^2/c}g(y,\theta) e^{-4\cdot 10^3((t-t_\theta)-t_\ell)} e^{100((t-t_\theta)-t_\ell)}\ll e^{-u^2/c}\frac{g(y,\theta) }{\log_{\ell-1} (t-t_\theta)}.
\end{equation}

\subsubsection{Bound on (\ref{lowerext})}\label{lowerextsection}

By definition, for $h\in C_{\ell+1}^c$, there must be a $k$ such that $\mathscr{S}_k(h)\leq m(k)-20((t-t_\theta)-k)+y$. We split $\mathscr{S}_k(h)$ according to the value of $\mathscr{S}_{t_\ell}\in[v,v+1]$ and $(\mathscr{S}_{k}-\mathscr{S}_{t_\ell})(h)\in[w,w+1]$, where $v,w\in \mathbb{Z}$, $|w|\leq 10^3((t-t_\theta)-t_\ell)$ and $w+v\leq m(k)-20((t-t_\theta)-k)+y$. 
Since $h\in H^\theta(y)$, we also have that
\[  {
    |\zeta_\tau e^{-S_{k+t_\theta}}(h)|>V_{\theta} e^{-w-v},}
\]
where $V_\theta = e^ye^{(t-t_\theta)}(t-t_\theta)^{-3/4}$, and since $h\in D_{\ell+1}$, this implies that either
\[
    |\zeta_\tau \mathcal{M}_{-1}\cdots \mathcal{M}_\ell^{(k)} (h)|\gg V_\theta e^{-w-v}
\]
or $\tfrac{1}{2}V_\theta e^{-w-u}\leq e^{-10^4(t-t_\theta-t_\ell)}$. However, the latter would force $\mathscr{S}_{t_\ell}(h)>m(t_\ell)+U_y(t_\ell)$ for large enough $t$, contradicting the fact that $h\in G_\ell$. A union bound on $h\in \mathcal{T}_t^\theta$ and $k\in (t_\ell,t_{t_\ell+1}]$ therefore yields
\begin{align}\label{3.4bound}
    &\mathbb{P}\big\{\exists h \in H^\theta(y)\cap C_{\ell+1}^c\cap D_{\ell+1}\cap A_{\ell+1}\cap G_\ell\cap E^\theta(u) \big\} \nonumber\\
    \ll &\sum_{\substack{k\in(t_\ell,t_{\ell+1}] \\ h\in \mathcal{T}_t^\theta  }} \sum_{\substack{w+v\leq m(k)+L_y(k) \\ |w|\leq 10^3(t_{\ell+1}-t_\ell) \\ L_y(t_\ell)\leq u-m(t_\ell)\leq U_y(t_\ell)}} \frac{e^{4w+4v}}{V_\theta^4} \mathbb{E}\bigg\{|(\zeta_\tau \mathcal{M}_{-1}\cdots \mathcal{M}_\ell^{(k)}(h)|^4\cdot \frac{|(\mathscr{S}_k-\mathscr{S}_{t_\ell})(h)|^{2q}}{(1+w^{2q})} \nonumber \\
    &\quad\quad\quad\quad\quad\quad\quad\quad\quad\quad\quad\quad\cdot
    \mathbf{1}\Big(\mathscr{S}_{t_\ell}(h) \in [u,u+1] \text{ and } h\in A_\ell\cap B_\ell \cap C_\ell\cap E^{\theta}(u)\Big)\bigg\}.
\end{align}

Picking $q=\lfloor w^2/(k-t_\ell)\rfloor$ and applying Lemma \ref{lem4} with the Dirichlet polynomial $\mathcal{Q}=(\mathscr{S}_k-\mathscr{S}_{t_\ell})^{2q}$, we find that the above is
\begin{align}\label{eq:bigsum1}
    \ll e^{-u^2/c}\hspace{-25px}\sum_{\substack{k\in(t_\ell,t_{\ell+1}]  \\ w+v\leq m(k)-20((t-t_\theta)-k)+y \\ |w|\leq 10^3((t-t_\theta)-t_\ell) \\ L_y(t_\ell)\leq v-m(t_\ell)\leq U_y(t_\ell)}}\hspace{-25px} \frac{e^{4w+4v}}{V_\theta^4} &e^{5(t-t_\theta)-4k} e^{-\frac{w^2}{(k-t_\ell)}} \frac{y(U_y(t_\ell)-v+m(t_\ell)+2)e^{-v^2/t_\ell}}{t_\ell^{3/2}}.
\end{align}
It is left to show that this summation is $\ll g(y,\theta)(\log_\ell (t-t_\theta))^{-1}$. 
{
    The restriction $L_y(t_\ell)\leq v-m(t_\ell)\leq U_y(t_\ell)$ implies that 
    \[
       0\leq U_{y}(t_\ell)-v+m(t_\ell) \leq U_y(t_\ell)-L_y(t_\ell) \ll \log_\ell (t-t_\theta)
    \]
    for all $y$. Using this in \eqref{eq:bigsum1} and removing the restriction on $v$ in that sum shows that it is
\begin{align*}
    \ll ye^{-4y-u^2/c}\hspace{-25px}\sum_{\substack{k\in(t_\ell,t_{\ell+1}] \\ w+v\leq m(k)-20((t-t_\theta)-k)+y \\ |w|\leq 10^3((t-t_\theta)-t_\ell)}}\hspace{-25px} (t-t_\theta)^3\log_\ell (t-t_\theta)\frac{e^{-v^2/t_\ell}}{t_\ell^{3/2}}{e^{4w+4v}} &e^{(t-t_\theta)-4k}e^{-{w^2}/{(k-t_\ell)}},
\end{align*}
recalling that $V_\theta=e^ye^{t-t_\theta}(t-t_\theta)^{-3/4}$. Making the change of variables $\bar{v}=v-m(t_\ell)-y, \bar{w}=w-m(k-t_\ell)$ and dropping the condition on $|w|$, this becomes
\begin{align}\label{eq:bigsum2}
        \ll ye^{-u^2/c}\hspace{-25px}\sum_{\substack{k\in(t_\ell,t_{\ell+1}] \\ \bar{w}+\bar{v}\leq -20((t-t_\theta)-k)\\\bar{w}\in \mathbb{Z}}}\hspace{-25px}{(t-t_\theta)^3}\log_\ell (t-t_\theta)e^{4m(k)+(t-t_\theta)-4k}{e^{4(\bar{w}+\bar{v})}}\frac{e^{-\frac{(\bar{v}+m(t_\ell)+y)^2}{t_\ell}}}{t_\ell^{3/2}}e^{-\frac{(\bar{w}+m(k-t_\ell))^2}{(k-t_\ell)}}.
\end{align}
To estimate this, we begin by expanding the Gaussian term in $\bar{v}$ in the above to get
\[
    t_{\ell}^{-3/2}e^{-(\bar{v}+m(t_\ell)+y)^2/t_\ell}\ll t_\ell^{3/2} e^{-y^2/t_\ell-2ym(t_\ell)/t_\ell-2\bar{v}(y/t_\ell)-(m(t_\ell)+\bar{v})^2/t_\ell}.
\]
Since $y\ll (t-t_\theta)/\log(t-t_\theta)$ by assumption, $e^{-2ym(t_\ell)/t_\ell} \ll e^{-2y}$, and $y/t_\ell=o(1)$. For our purposes, we will view $2y/t_\ell$ as a small constant which we henceforth denote by $\delta$.  Using the bound
\[
    \frac{e^{-(m(t_\ell)+\bar{v})^2/t_\ell}}{t_\ell^{3/2}}\leq e^{-t_\ell}e^{-2\bar{v}},
\]
we conclude that  
\[
    t_{\ell}^{-3/2}e^{-(\bar{v}+m(t_\ell)+y)^2/t_\ell} \ll e^{-t_\ell}e^{-y^2/t_\ell-2y-(2+\delta)\bar{v}}.
\]
Plugging this back into \eqref{eq:bigsum2} yields
\begin{align*}
        &\ll g(y,\theta)e^{-u^2/c}\\
        &\hspace{10px}\times\hspace{-15px}\sum_{\substack{k\in(t_\ell,t_{\ell+1}] \\ \bar{w}+\bar{v}\leq -20((t-t_\theta)-k)\\\bar{w}\in \mathbb{Z}}}\hspace{-25px}{(t-t_\theta)^{3}}\log_\ell (t-t_\theta)e^{4m(k)+(t-t_\theta)-4k-t_\ell}{e^{4\bar{w}+(2-\delta)\bar{v}}}e^{-\frac{(\bar{w}+m(k-t_\ell))^2}{(k-t_\ell)}}.\nonumber
\end{align*}
We perform the sum over $\bar{v}$ to get
\begin{align*}
        &\ll g(y,\theta)e^{-u^2/c}\log_\ell (t-t_\theta)\\
        &\times \sum_{\substack{k\in(t_\ell,t_{\ell+1}] \\\bar{w}\in \mathbb{Z}}}{(t-t_\theta)^{3}}e^{4m(k)+(t-t_\theta)-4k-t_\ell-20(2-\delta)((t-t_\theta)-k)}{e^{(2+\delta)\bar{w}}}e^{-\frac{(\bar{w}+m(k-t_\ell))^2}{(k-t_\ell)}},\nonumber
\end{align*}
followed by the sum over $\bar{w}$, yielding
\begin{align*}
            &\ll g(y,\theta)e^{-u^2/c}\sum_{\substack{k\in(t_\ell,t_{\ell+1}]}}\log_\ell (t-t_\theta)e^{(t-t_\theta)-t_\ell}(k-t_\ell)^{1/2}\\
            &\hspace{50px}\times{(t-t_\theta)^{3}}e^{4m(k)-(2+\delta)m(k-t_\ell)-4k-20(2-\delta)((t-t_\theta)-k)}e^{\tfrac{(2+\delta)^2}{4}(k-t_\ell)}.\nonumber
\end{align*}
Using the fact that $(k-t_\ell)\ll\log_\ell(t-t_\theta)$ along with the definition of $m(k)$ and  $m(k-t_\ell)$, this is
\begin{align*}
                &\ll g(y,\theta)e^{-u^2/c}\sum_{\substack{k\in(t_\ell,t_{\ell+1}]}}(\log_\ell (t-t_\theta))^{3/2}\cdot(t-t_\theta)^{3-3\Big(\tfrac{k}{t-t_\theta}-\tfrac{2+\delta}{4}\Big(1-\tfrac{k}{t-t_\theta}\Big)\Big)}\\
            &\hspace{150px}\times e^{(t-t_\theta)-t_\ell}e^{(k-t_\ell)((2+\delta)^2/4-(2+\delta))}e^{-20(2-\delta)((t-t_\theta)-k)}\\
            &\ll g(y,\theta)e^{-u^2/c}\sum_{\substack{k\in(t_\ell,t_{\ell+1}]}}\log_\ell (t-t_\theta)^{3/2}\cdot(t-t_\theta)^{3-3\Big(\tfrac{k}{t-t_\theta}-\tfrac{2+\delta}{4}\Big(1-\tfrac{k}{t-t_\theta}\Big)\Big)}\\
            &\hspace{300px}\times e^{-19((t-t_\theta)-k)}.
\end{align*}
If $\ell\geq 1$, then $k/(t-t_\theta)\sim 1$ and the powers of $(t-t_\theta)$ in the above cancel out to give a factor of the form $(t-t_\theta)^{c \log_{\ell} (t-t_\theta)/(t-t_\theta)}$ for a constant $c>0$, which is bounded in $t$. Otherwise, we can use a power of $e^{-((t-t_\theta)-k)}$ to offset the resulting positive power of $(t-t_\theta)$, recalling that $\log_0 (t-t_\theta):=(t-t_\theta)$. It follows that the above is 
\[
    \ll g(y,\theta)e^{-u^2/c}\sum_{\substack{k\in(t_\ell,t_{\ell+1}]}}(\log_\ell(t-t_\theta))^{3/2} e^{-10((t-t_\theta)-k)} \ll g(y,\theta)e^{-u^2/c}(\log_\ell (t-t_\theta))^{-1}.
\]
}
\subsubsection{Bound on (\ref{upperext})}\label{proofupperext} We show that
\[
    \P{\exists h\in (B_\ell\setminus B_{\ell+1})\cap C_{\ell+1}\cap E^\theta(u)} \ll \frac{g(y,\theta)}{\big(\log_{\ell+1}(t-t_\theta)\big)^{100}}{e^{-u^2/c}}.
\]
Using a union bound over $k\in [t_\ell,t_{\ell+1})$ and summing over all possible values $\mathscr{S}_k(h)\in[w,w+1], w\in \mathbb{Z}$, we bound the above by 
\begin{align*}
    \sum_{\substack{k\in[t_\ell,t_{\ell+1})\\ w\in[L_y(k),U_y(k)]}} \mathbb{P}\Big\{\exists h : \overline{\mathscr{S}}_j(h)<U_y(j)\, 
    \forall j\leq k,\, \overline{\mathscr{S}}_{k+1}(h)>U_y(k+1),\\ \overline{\mathscr{S}}_k(h)\in (w,w+1],\, h\in E^\theta(u)\Big\}
\end{align*}
where $\overline{\mathscr{S}}_k:=\mathscr{S}_k-m(k)$. Letting $V_{w,k}=U_y(k+1)-w$, using the fact that
\[
    \mathscr{S}_{k+1}-\mathscr{S}_k > U_y(k+1)+m(k+1)-m(k)-\overline{\mathscr{S}}_k > U_y(k+1)-w+o\Big(\frac{\log (t-t_\theta)}{t-t_\theta}\Big)
\]
and applying Markov's inequality shows that this sum is
\begin{align}\label{upperextmain}
    \ll \sum_{\substack{k\in[t_\ell, t_{\ell+1})\\ w\in [L_y(k),U_y(k)]}} &\mathbb{E}\Big\{\max_{|h|\leq \log^\theta T} \frac{|(\mathscr{S}_{k+1}-\mathscr{S}_k+1)(h)|^{2q}}{(V_{w,k}+1)^{2q}} \\ &*\mathbf{1}\left(\overline{\mathscr{S}}_j(h)<U_y(j), \forall j\leq k, \overline{\mathscr{S}}_k(h)\in(w,w+1],\, h\in E^\theta(u)\right)\Big\}.
\end{align}
Picking $q=\lfloor(V_{w,k}+1)^2/10\rfloor$, which is bounded by $ 400((t-t_\theta)-k)^2$ for $w\geq L_y(k)$, we note that the Dirichlet polynomial $(\mathscr{S}_{k+1}-\mathscr{S}_k+1)^{2q}$ has length $\leq \exp(1000((t-t_\theta)-k)^2e^{k+1+t_\theta})$. We can therefore apply Lemma \ref{lem3} to each summand, yielding
\begin{align}\label{eq:sumupper}
    \ll e^{-u^2/c}&\sum_{\substack{k\in[t_\ell,t_{\ell+1})\\ w\in [L_y(k),U_y(k)]}}e^{k}\big((t-t_\theta)-k\big)^{800} \\ &\frac{\E{|(\mathscr{S}_{k+1}-\mathscr{S}_k+1)(0)|^{2q}}}{(V_{w,k}+1)^{2q}} \frac{y(U_y(k)-w+1)e^{-(w+m(k))^2/k}}{k^{3/2}}.
\end{align}
{By Equation \eqref{eq:gaussianmomentshifted}, the expectation is $\ll (2q)!/q!+4^q$, which by Stirling's approximation is $\ll 100^q(q/e)^q$. For our choice of $q$, it then follows that
\[
    \frac{\E{|(\mathscr{S}_{k+1}-\mathscr{S}_k+1)(0)|^{2q}}}{(V_{w,k}+1)^{2q}}\ll e^{-(V_{w,k}+1)^2/10}, 
\]
and plugging this back into \eqref{eq:sumupper} yields
\begin{align*}
    \ll ye^{-u^2/c} \sum_{\substack{k\in[t_\ell,t_{\ell+1})\\ w\in [L_y(k),U_y(k)]}}e^{k}\big((t-t_\theta)-k\big)^{800} e^{-(V_{w,k}+1)^2/10} \frac{(U_y(k)-w+1)e^{-(w+m(k))^2/k}}{k^{3/2}}.
\end{align*}
To estimate this sum, we first make the change of variables $\bar{w}=w-y$ to get
\begin{align}\label{eq:sumupper2}
\ll ye^{-u^2/c} \hspace{-10px}\sum_{\substack{k\in[t_\ell,t_{\ell+1})\\ \bar{w}\in [L(k), U(k)]}}\hspace{-10px}e^{k}\big((t-t_\theta)-k\big)^{800} e^{-(U(k+1)-\bar{w}+1)^2/10} \frac{(U(k)-\bar{w}+1)e^{-(\bar{w}+y+m(k))^2/k}}{k^{3/2}}
\end{align}
where $U(k):=U_y(k)-y$ and similarly for $L(k)$. Expanding the rightmost Gaussian term yields 
\begin{align*}
    \frac{e^{-(\bar{w}+y+m(k))^2/k}}{{k^{3/2}}}&=e^{-y^2/k-2y}e^{-2\bar{w}(y/k)}\frac{e^{-(\bar{w}+m(k))^2/k}}{k^{3/2}}\\
    &\ll e^{-2y-y^2/(t-t_\theta)}e^{-\delta\bar{w}}e^{-k-2\bar{w}},
\end{align*}
where $\delta:=(y/k)=o(1)$ can be seen as a small constant. Using this estimate in \eqref{eq:sumupper2}, and multiplying the resulting sum by $e^{(2+\delta)(U(k+1)-1)-(2+\delta)U(k)}$ (which is $O(1)$) yields
\begin{align*}
   \ll g(y,\theta)e^{-u^2/c} \sum_{k\in[t_\ell,t_{\ell+1})}&\big((t-t_\theta)-k\big)^{800}e^{-(2+\delta)U(k)}\\ 
   &\times\sum_{\bar{w}} {(U(k)-\bar{w}+1)}e^{-(U(k+1)-\bar{w}+1)^2/10-(2+\delta)(U(k+1)-\bar{w}+1)}.\nonumber
\end{align*}
Notice that the sum over $\bar{w}$ is now finite, while the sum over $k$ is 
\[
    \ll \sum_{k\in [t_\ell,t_\ell+1)} \big((t-t_\theta)-k\big)^{800}e^{-10^3\log((t-t_\theta)-k)} \ll \big(\log_{\ell+1}(t-t_\theta)\big)^{-100},
\]
which yields the desired bound.
}

\subsection{Proof of Proposition \ref{last}} 
If $h\in H^\theta(y)\cap G_\mathcal{L}$ and $\mathscr{S}_{t_\mathcal{L}}\in[v,v+1]$, then $|v-y-m(t_\mathcal{L})|\leq 20((t-t_\theta)-t_\mathcal{L})$ and $|(\zeta_\tau e^{-{S_{t_\mathcal{L}+t_\theta}}})(h)|>V_\theta e^{-v}$ where $V_\theta=e^{y}e^{(t-t_\theta)}(t-t_\theta)^{-3/4}$. {A union bound on $h\in \mathcal{T}_t^\theta$ thus gives
\begin{align}\label{4.5sum}
        &\P{\exists h\in H^\theta(y)\cap G_\mathcal{L}\cap E^\theta(u)} \\ \nonumber
        &\ll e^{(t-t_\theta)}\sum_{\substack{|v-y-m(t_\mathcal{L})| \\
    \,\,\leq 20((t-t_\theta)-t_\mathcal{L})}} \P{|\zeta_\tau e^{-S_{t_\mathcal{L}+t_\theta}}(h)|>V_\theta e^{-v},\, \mathscr{S}_{t_\mathcal{L}}\in[v,v+1],\, h\in G_\mathcal{L} \cap E^\theta(u)}.
\end{align}
Noting that $V_\theta e^{-v}>e^{-10^4((t-t_\theta)-t_\mathcal{L})}$ and that
\[
|\zeta_\tau e^{-{S}_{k+t_\theta}}(h)|\leq c_\mathcal{L}|\zeta_\tau \mathcal{M}_{-1}\cdots \mathcal{M}_{\mathcal{L}-1}^{(k)}(h)|+e^{-10^4((t-t_\theta)-t_{\mathcal{L}-1})}
\]
since $h\in D_\mathcal{L}$, it follows that
\[
    V_\theta e^{-v} \ll |(\zeta_\tau \mathcal{M}_{-1}\cdots \mathcal{M}_\mathcal{L})(h)|.
\]
We can therefore bound each summand in (\ref{4.5sum}) using Markov's inequality and Lemma \ref{lem4}
\begin{align*}
    \mathbb{P}&\big\{\exists h\in H^\theta(y)\cap G_\mathcal{L}\cap E^\theta(u)\big\} \\
    &\ll e^{-4y-4(t-t_\theta)}(t-t_\theta)^3e^{(t-t_\theta)} e^{-u^2/c}\hspace{-10px}\sum_{\substack{|v-y-m(t_\mathcal{L})| \\
    \,\,\leq 20((t-t_\theta)-t_\mathcal{L})}}\hspace{-10px} e^{4v} e^{4((t-t_\theta)-t_\mathcal{L})}y((t-t_\theta)-t_\mathcal{L})\frac{e^{-v^2/t_\mathcal{L}}}{t_\mathcal{L}^{3/2}}.\\
\end{align*}
Recall that $(t-t_\theta)-\mathcal{L}=O(1)$ by definition of $\mathcal{L}$ (cf.~ Equation \eqref{finalincrement}). It follows that the above is 
\[
    \ll ye^{-4y-4(t-t_\theta)}(t-t_\theta)^3e^{(t-t_\theta)} e^{-u^2/c}\hspace{-10px}\sum_{\substack{|v-y-m(t_\mathcal{L})| \\
    \,\,\leq 20((t-t_\theta)-t_\mathcal{L})}}\hspace{-10px} \frac{e^{4v-v^2/t_\mathcal{L}}}{t_\mathcal{L}^{3/2}},
\]
and that this sum contains $O(1)$ many terms, each satisfying
\[
    \ll e^{4y+4m(t_\mathcal{L})-y^2/(t-t_\theta)-2y-t_\mathcal{L}}.
\]
We conclude that
\[
    \mathbb{P}\big\{\exists h\in H^\theta(y)\cap G_\mathcal{L}\cap E^\theta(u)\big\} \ll g(y,\theta)e^{-u^2/c}\big((t-t_\theta)^3e^{4m(t_\mathcal{L})-4t_\mathcal{L}}\big) \ll g(y,\theta)e^{-u^2/c}
\]
since $\big((t-t_\theta)^3e^{4m(t_\mathcal{L})-4t_\mathcal{L}}\big)=O(1)$.
}

\section{Auxiliary results in Section 4}\label{sectionaux}
\subsection{Proof of Lemma \ref{lem3}}
\noindent We begin with the proof of Lemma \ref{lem3} and omit that of \ref{basedecoupling}, which is similar and much simpler. We follow the strategy in \cite[Section 7.2]{FHK1}, where the main idea is to approximate the indicator appearing in \eqref{eq:4.8} using Dirichlet polynomials. We first express it as a product of indicators involving the increments
$$
\mathscr{Y}_r:=\mathscr{S}_r,\quad \mathscr{Y}_j:=\mathscr{S}_j-\mathscr{S}_{j-1}, \quad j>r:=\lceil y/4\rceil
$$
(noting the slight abuse of notation). Approximating each of these by a Dirichlet polynomial then allows us to use Lemma \ref{splittinglemma} to decouple their second moments. This eventually leads to a Gaussian comparison argument and an application of the ballot theorem in Proposition \ref{prop:ballot}.

To express the restrictions $\mathscr S_j\in [m(j)+L_y(j),m(j)+U_y(j)]$, $r\leq j\leq k$, and $\mathscr S_k\in (w,w+1]$ in terms of the increments, we partition the range of values of each $\mathscr{Y}_j$ into intervals of the form $[u_j,u_j+\Delta_j^{-1}]$ where $u_j\in \Delta_j^{-1}\mathbb{Z}$ for a discretization parameter
\begin{equation}
\Delta_j=\min(j,(t-t_\theta)-j)^4.
\end{equation}
Note that $\sum_j \Delta_j^{-1}\leq 1$.
Define the set $\mathcal{I}_{k}(w)\subseteq \mathbb{R}^{k-r+1}$ of tuples $\mathbf u=(u_{r},\cdots,u_k)$ with $u_j\in\Delta^{-1}_j\mathbb{Z}$, $r\leq j\leq k$, such that
\begin{align*}
     &L_y(j)-1\leq \sum_{i=r}^ju_i- m(j)\leq U_y(j)+1,\quad \forall r\leq j\leq k,\\
     &\bigg|\sum_{i=r}^ku_i-w\bigg|\leq 1.
\end{align*}
With this definition, one readily checks the inclusion
\begin{equation}
\label{eqn: inclusion}
\begin{aligned}
&\Big\{\mathscr S_j(h)-m(j)\in [L_y(j),U_y(j)], r\leq j\leq k, \mathscr S_k(h)\in (w,w+1]\Big\}\\
&\subseteq \bigcup_{\mathbf u\in \mathcal{I}_{k}(w)} \Big\{\mathscr Y_j(h)\in [u_j,u_j+\Delta_j^{-1}), \mathscr S_k(h)\in (w,w+1]\Big\}.
\end{aligned}
\end{equation}
We can then work on the right-hand side for an upper bound.
Crucially, note that for $\mathbf{u}\in \mathcal{I}_k(w)$, 
 \begin{equation}\label{eq:urestriction}
     |u_j|\leq U_y(j)-L_y(j)\leq 40\min(j, (t-t_\theta)-j) \leq 100\Delta_j^{1/4},\quad \forall r\leq j\leq k.
 \end{equation}
 
\bigskip 

We now approximate the indicator function of every interval $[u_j,u_j+\Delta_j^{-1}]$ by a polynomial (applying this approximation to
$\mathbf 1(\mathscr Y_j(h)\in [u_j,u_j+\Delta_j^{-1}])$ then yields a Dirichlet polynomial as desired).
Given $\Delta>0$ and  $A>3$, let $G_{\Delta, A}$ be an entire function in $L^2(\mathbb{R})$ that approximates $\mathbf{1}(x\in[0,\Delta^{-1}])$ and satisfies, for some absolute constant $C>0$, the following properties:
\begin{enumerate}\label{list:properties}
    \item The Fourier transform $\hat{G}_{\Delta, A}$ is supported in $[-\Delta^{2A}, \Delta^{2A}]$.
    \item For all $x\in \mathbb{R}$, $0\leq G_{\Delta,A} \leq 1$.
    \item $\mathbf{1}(x\in[0,\Delta^{-1}])\leq G_{\Delta,A}(x)\cdot (1+Ce^{-\Delta^{A-1}})$. 
    \item We have $G_{\Delta,A}(x)\leq \mathbf{1}(x\in[-\Delta^{-A/2}, \Delta^{-1}+\Delta^{-A/2}) +Ce^{-\Delta^{A-1}}$. 
    \item $\int_{\mathbb{R}} |\hat{G}_{\Delta,A}(x)|\mathrm{d}x\leq 2\Delta^{2A}$.
\end{enumerate}
The existence of such a function is given by Lemma 6 in \cite{FHK1}. 
We then define a truncated version of $G_{\Delta,A}$ given by
\begin{equation}
\label{eqn: D}
\mathcal{D}_{\Delta,A}(x)=\sum_{\ell\leq \Delta^{10A}}\frac{(2\pi i x)^\ell}{\ell!} \int_{\mathbb{R}}\xi^\ell \hat{G}_{\Delta, A}(\xi)\mathrm{d}\xi,
\end{equation}
and approximate $\mathbf{1}(\mathscr{Y}_j(h)\in[u_j,u_j+\Delta_j^{-1}])$ by $\mathcal{D}_{\Delta_j,A}(\mathscr{Y}_j-u_j)$. 
Note that $\mathcal{D}_{\Delta_j,A}(\mathscr{Y}_j-u_j)$ is supported on integers $n$ whose prime factors lie in $(\exp e^{j-1+t_\theta},\exp e^{j+t_\theta}]$ and for which $\Omega(n)\leq \Delta_{j}^{10A}$. This polynomial is of length $\leq \exp(2\Delta_j^{10A}e^{j+t_\theta})$, and 
\begin{equation}\label{eq:propertyofcoeffs}
    \int_{\mathbb{R}}|\xi|^\ell |\hat{G}_{\Delta, A}(\xi)|\mathrm{d}\xi \leq \Delta^{2A\ell}  \int_{\mathbb{R}} |\hat{G}_{\Delta,A}|\mathrm{d}\xi\leq 2\Delta^{2A\ell+2A}.
\end{equation}
by the first and fifth properties of $G_{\Delta,A}$ in the list above.
As a result, its coefficients are bounded by $\ll \Delta_{j}^{2A(\ell+1)}$. 

This leads to the following.
\begin{lemma}[Approximation of indicator functions by Dirichlet polynomials]\label{indicatordirichlet}
    Let $\theta\in (-1,0)$, $\ell\geq -1$, $A>10$, and $y>1$. Let $|u_0|>1$, $k\geq 1$, and $w$ be such that $L_y(k)\leq w-m(k)\leq U_y(k)$. Then for any fixed $\tau\in [T,2T]$, we have
    \begin{align*}
        \mathbf{1}&\Big(h\in B_\ell^{(k)}\cap C_{\ell}^{(k)}\cap E^{\theta}(u_0),\mathscr{S}_k(h)\in[w,w+1]\Big)\\
                &\leq C \frac{|S_{t_\theta}(0)-S_{t_\theta}(h)|^{2q}}{(|u_0|-1)^{2q}}\sum_{\mathbf{u}\in\mathcal{I}_{k}(w)}\prod_{j=r}^k |\mathcal{D}_{\Delta_{j},A}(\mathscr{Y}_j(h)-u_j)|^2
    \end{align*}
    for some absolute constant $C>0$ and any $q>0$.
\end{lemma}
\begin{proof}
        It suffices to show that for every $j\in[r, k]$ and $|u_j|\leq100\min(j, (t-t_\theta)-j)$
        \begin{align}\label{incind}
            &\mathbf{1}(\mathscr{Y}_j(h)\in[u_j, u_j+\Delta_{j}^{-1}]) \leq |\mathcal{D}_{\Delta_{j},A}(\mathscr{Y}_j(h)-u_j)|^2(1+Ce^{-\Delta_j^{A-1}}),\quad \text{ for }j>r,
        \end{align}
        and otherwise that
        \begin{align}\label{incindr}
        \mathbf{1}(\mathscr{Y}_r(h)\in[u_r, u_r+\Delta_{r}^{-1}]) \leq C|\mathcal{D}_{\Delta_{r},A}(\mathscr{Y}_r(h)-u_r)|^2
        \end{align}
        for some constant $C>0$.
        Noting that 
        \[
            \mathbf{1}(h\in E^\theta(u_0))\leq \frac{|S_{t_\theta}(0)-S_{t_\theta}(h)|^{2q}}{(|u_0|-1)^{2q}},
        \]
         the inclusion in \eqref{eqn: inclusion} together with \eqref{incind} and \eqref{incindr} would then imply that
    \begin{align*}
        \mathbf{1}&\Big(h\in B_\ell^{(k)}\cap C_{\ell}^{(k)}\cap E^\theta(u_0):\mathscr{S}_k(h)\in[w,w+1]\Big)
                \leq \frac{|S_{t_\theta}(0)-S_{t_\theta}(h)|^{2q}}{(|u_0|-1)^{2q}} *\\&C\sum_{\mathbf{u}\in\mathcal{I}_{k}(w)}|\mathcal{D}_{\Delta_{r},A}(\mathscr{Y}_r(h)-u_r)|^2\prod_{j=r+1}^k \left(|\mathcal{D}_{\Delta_{j},A}(\mathscr{Y}_j(h)-u_j)|^2(1+Ce^{-\Delta_j^{A-1}})\right),
    \end{align*}
    and the lemma follows since $\prod_j(1+Ce^{-\Delta_j^{A-1}})$ is bounded by an absolute constant.

    The proof of \eqref{incind} and \eqref{incindr} follows from the proof of Lemma 7 in \cite{FHK1}. We sketch the main steps. To prove \eqref{incind}, one first bounds the left-hand side from above by
    \[
        |G_{\Delta_j,A}(\mathscr{Y}_j(h)-u_j)|^2(1+Ce^{-\Delta_j^{A-1}}) = \Big|\int_{\mathbb{R}}e^{2\pi i\xi(\mathscr{Y}_j(h)-u_j)}\hat{G}_{\Delta_j,A}(\xi)\mathrm{d}\xi\Big|^2(1+Ce^{-\Delta_j^{A-1}}).
    \]
    By expanding the exponential up to order $\Delta_j^{10A}$, this integral equals the sum of $\mathcal{D}_{\Delta_j, A}(\mathscr{Y}_j(h)-u_j)$ and an error term $\mathcal{E}$, which is shown to be $\leq e^{-\Delta_j^{4A}}$ in absolute value using the aforementioned properties of $G_{\Delta_j,A}$ and the fact that $|\mathscr{Y}_j(h)-u_j|\leq 10^4\Delta_j^{1/4}$. Noting that $|\mathcal{D}_{\Delta_j,A}(\mathscr{Y}_j(h)-u_j)|\geq 1/2$ whenever the left-hand side in \eqref{incind} equals $1$, this additive error $\mathcal{E}$ can then be made multiplicative, and the inequality follows.
    
    The proof of \eqref{incindr} is similar.
\end{proof}
We now compare the increments $\mathscr{Y}_j$ to the same random Steinhaus model used in Section \ref{corrsection}. An application of the Berry-Esseen theorem is then used to compare the latter to Gaussian random variables, and by extension, $\mathscr{S}_k$ to a bona fide random walk.

As in \eqref{eqn: X_p}, let
\begin{equation}\label{eq:Steinhaus}
        X_p(h)=\text{Re}\left(Z_p p^{-1/2-ih}+\tfrac{1}{2}Z_p^2p^{-1-2ih}\right),
\end{equation}
where the $Z_p$'s are i.i.d.\ Steinhaus random variables.
Let
\begin{equation}\label{eq:Steinhaus2}
        \mathcal{S}_k(h)=\sum_{e^{1000}<\log p \leq e^k}X_p(h),\text{ and } \mathcal{Y}_k(h)=\mathcal{S}_{k+t_\theta}(h)-\mathcal{S}_{k+t_\theta-1}(h), \quad 1\leq k\leq t-t_\theta.
\end{equation}
We also define
\begin{equation}\label{eq:Wdef}
        \mathcal{W}_k = \sum_{1000\leq \ell \leq k} \mathcal{N}_\ell,\quad 1000\leq k\leq t,
\end{equation}
where the $\mathcal{N}_\ell's$ are real Gaussian random variables with zero mean and variance $1/2$. 

{\begin{lemma}[Gaussian comparison]\label{gaussiancomparison}
Let $A=20$ and $\ell\geq 0$ satisfy (\ref{finalincrement}). Let $k\in(t_\ell,t_{\ell+1}]$, and assume that $1< |u_0|<C\sqrt{t-t_\theta}$ for some constant $C$. Then one has for any $|h|\leq 2$,
\begin{align*}
    &\E{\frac{|S_{t_\theta}(0)-S_{t_\theta}(h)|^{2q}}{(|u_0|-1)^{2q}}\prod_{j=r}^k |\mathcal{D}_{\Delta_j,A}(\mathscr{Y}_j(h)-u_j)|^2} \\
    &\leq (1+Ce^{-ae^t})\E{\frac{|{S}_{t_\theta}(0)-{S}_{t_\theta}(h)|^{2q}}{(|u_0|-1)^{2q}}} \prod_{j=r}^k \E{|\mathcal{D}_{\Delta_j, A}(\mathcal Y_j(h)-u_j)|^2}
\end{align*}
with $C,a>0$ absolute constants, and $q=\lfloor (|u_0|-1)^2/8\rfloor$. Furthermore, for $w-m(k)\in[L_y(k),U_y(k)]$, we have
\begin{align}\label{eq:GaussianComparison}
    &\sum_{\substack{\mathbf{u}\in \mathcal{I}_k(w)}}\E{\frac{|{S}_{t_\theta}(0)-{S}_{t_\theta}(h)|^{2q}}{(|u_0|-1)^{2q}}} \prod_{j=r}^k \E{|\mathcal{D}_{\Delta_j, A}(\mathcal Y_j(h)-u_j)|^2} \nonumber \\
    &\leq C e^{-u_0^2/16}\sum_{\substack{\mathbf{u}\in \mathcal{I}_k(w)}} {\P{\mathcal{W}_{r}\in[u_r,u_r+\Delta_r^{-1}),\mathcal{N}_j\in[u_j,u_j+\Delta_j^{-1}], \ \forall r<j\leq k},}
\end{align}
for some absolute constant $C>0$.
\end{lemma}}

\begin{proof}
    This follows the proof of Lemma 8 in \cite{FHK1}. We give the broad lines. The first claim follows from using Lemma \ref{splittinglemma} to factor out the $2q-$th moment of $S_{t_\theta}(0)-S_{t_\theta}(h)$, and the mean-value theorem for Dirichlet polynomials (Lemma \ref{MVMV}) to compare $\E{\prod_j|\mathcal{D}_{\Delta_j, A}(\mathscr Y_j(h)-u_j)|^2}$ to $\prod_j\E{|\mathcal{D}_{\Delta_j, A}(\mathcal Y_j(h)-u_j)|^2}$ (noting the independence of the $\mathcal{Y}_j$).
    The error when using these lemmas is small since the length of the polynomial $(S_{t_\theta}(h)-S_{t_\theta}(0))^q\prod_{j=r}^k\mathcal{D}_{\Delta_j, A}(\mathscr{Y}_j(h)-u_j)$ is at most
    \[
    \exp\big(2qe^{t_\theta}+2\sum_{r\leq j\leq k} e^{j+t_\theta}\Delta_j^{200}\big)\leq \exp\big(100e^{t_{\ell+1}+t_\theta}((t-t_\theta)-k)^{800}\big)\leq T^{1/100},
\]
 for large enough $T$, by the condition \eqref{finalincrement}.
This proves the first inequality.

For the second inequality, we begin by using Lemma \ref{MV} and Stirling's approximation on the $2q$-th moment of $|{S}_{t_\theta}(0)-{S}_{t_\theta}(h)|$ to get the Gaussian tail $e^{-u_0^2/16}$. 

To bound each $\mathbb{E}\{|\mathcal{D}_{\Delta_j,A}(\mathcal{Y}_j(h)-u_j)|^2\}$ for $r<j\leq k$, one first argues that
\begin{equation}\label{eq:DtoG}
        \mathbb{E}\big\{|\mathcal{D}_{\Delta_j,A}(\mathcal{Y}_j(h)-u_j)|^2\big\}\leq \mathbb{E}\big\{|G_{\Delta_j,A}(\mathcal{Y}_j(h)-u_j)|^2\big\}+O\big(e^{-\tfrac{1}{8}\Delta_j^{6A}}\big)
\end{equation}
by decomposing the left-hand side according to whether $|\mathcal{Y}_j(h)-u_j|\leq \Delta_{j}^{6A}$ or $>\Delta_{j}^{6A}$. In the former case, the inequality follows directly from the defining equation for $\mathcal{D}_{\Delta_j,A}$ \eqref{eqn: D} and the fact that $G_{\Delta_j,A}(\mathcal{Y}_j(h)-u_j)\in (0,1)$. For the latter, the Cauchy-Schwarz inequality yields
\[
    \mathbb{P}\big\{|\mathcal{Y}_j(h)-u_j|> \Delta_{j}^{6A}\big\}^{1/2}\mathbb{E}\big\{|\mathcal{D}_{\Delta_j,A}(\mathcal{Y}_j(h)-u_j)|^4\big\}^{1/2}.
\]
The probability is then shown to be small by a Chernoff bound, while the fourth moment is straightforwardly bounded by \eqref{eqn: D} and an estimate similar to \eqref{eq:propertyofcoeffs}. 

Using the fourth property of $G_{\Delta_j,A}$ \eqref{list:properties} and the Berry-Esseen theorem \cite[Lemma 20]{FHK1}, it follows that
\begin{align*}
        \mathbb{E}\big\{|G_{\Delta_j,A}(\mathcal{Y}_j(h)-u_j)|^2\big\}&\leq \mathbb{P}\big\{\mathcal{Y}_j(h)\in [u_j-\Delta_j^{-A/2},u_j+\Delta_j^{-1}+\Delta_j^{-A/2}]\big\}+O\big(e^{-\Delta_j^{A-1}}\big)
\end{align*}
and this probability can be shown to equal $ \mathbb{P}\big\{\mathcal{N}_j\in [u_j,u_j+\Delta_j^{-1}]\big\}\cdot(1+O(\Delta_j^{-A/4}))$  using the restrictions $|u_j|\leq 100\min(j,t-t_\theta-j)$ (see \eqref{eq:urestriction}). The same argument also shows that $\mathbb{E}\{|\mathcal{D}_{\Delta_r,A}(\mathcal{Y}_r(h)-u_r)|^2\}\ll \mathbb{P}\{\mathcal{W}_r\in [u_r,u_r+\Delta_r^{-1}]\}$, and the lemma follows since $\prod_j(1+O(\Delta_j^{-A/4}))$ is bounded by an absolute constant.

\end{proof}

To complete the proof of Lemma \ref{lem3}, we need to deal with the maximum and the additional Dirichlet polynomial $\mathcal Q$. We note that Lemma \ref{indicatordirichlet} implies that
\begin{align}\label{fourtyfour}
    &\mathbf{1}\left(h\in B^{(k)}_\ell \cap C_{\ell}^{(k)}\cap E^\theta(u_0)\text{ and } \mathscr{S}_k(h)\in(w,w+1]\right) \nonumber \\
    &\leq C\frac{|S_{t_\theta}(0)-S_{t_\theta}(h)|^{2q}}{(|u_0|-1)^{2q}}\sum_{\mathbf{u}\in \mathcal{I}_k(w)} {\prod_{j=r}^k}|\mathcal{D}_{\Delta_j,A}(\mathscr{Y}_{j}(h)-u_j)|^2,
\end{align}
for some absolute constant $C$, $q=\lfloor (|u_0|-1)^2/8\rfloor$ and $A=20$. The right-hand side can be written as a linear combination $\sum_{i\in\mathcal{I}}D_i$ of Dirichlet polynomials $D_i$, each of length
\[
\leq {\exp\big(2qe^{t_\theta}+2\sum_{r\leq j\leq k} e^{j+t_\theta}\Delta_j^{200}\big)\leq \exp\big(100e^{k+t_\theta}((t-t_\theta)-k)^{800}\big).}
\]
Therefore, multiplying (\ref{fourtyfour}) by an arbitrary Dirichlet polynomial $\mathcal{Q}$ of length $N\leq \exp(\tfrac{e^t}{100})$ and discretizing using Lemma \ref{disclemma} yields
\begin{align}\label{midpointlem3}
\begin{split}
    \E{\max_{|h|\leq 2\log^\theta T}|\mathcal{Q}(\tfrac{1}{2}+i\tau+ih)|^2\mathbf{1}\left(h\in B_{\ell}^{(k)}\cap C_{\ell}^{(k)}\cap E^\theta(u_0),\, \mathscr{S}_k(h)\in(w,w+1]\right)} \\
    \ll (e^{-t_\theta} \log N+e^{k}((t-t_\theta)-k)^{800})\sum_{i\in \mathcal{I}}\E{|\mathcal{Q}(\tfrac{1}{2}+i\tau)|^2|D_i(\tfrac{1}{2}+i\tau)^2|}.
    \end{split}
\end{align}
We used the fact that the expectations are roughly of the same order for all the relevant $h$'s after discretising, except for the $h$'s associated with very large $j$ in Lemma \ref{disclemma} for which the contribution is negligible (see the paragraph after (8) in \cite{FHK1} for more detail). 

\begin{remark}\label{mesomaxsubtlety} In the $\theta=0$ case, the Dirichlet polynomial used to approximate the indicator function pertaining to $S_k$ is of length $\leq \exp(100e^k(t-k)^{800})$, whereas for $\mathscr{S}_{k}$ (when $\theta<0$), our bound increased to $\exp(100e^{k+t_\theta}(t-t_\theta-k)^{800})$. The number of points arising from the discretization however remains of the same order, since doing so on a smaller interval (of width $\asymp e^{t\theta}$) gives the additional factor of $e^{-t_\theta}$ needed to counteract the longer polynomial.
\end{remark}

Noting that the $D_i$ are supported on primes up to $\exp(e^{k+t_\theta})$ while $\mathcal{Q}$ is supported on primes $>\exp(e^{k+t_\theta})$, we can use Lemma \ref{splittinglemma} to split each  expectation in the sum in \eqref{midpointlem3} and write
\[
    \E{|\mathcal{Q}(\tfrac{1}{2}+i\tau)|^2|D_i(\tfrac{1}{2}+i\tau)|^2} \leq 2\E{|\mathcal{Q}(\tfrac{1}{2}+i\tau)|^2}\E{|{D}_i(\tfrac{1}{2}+i\tau)|^2}.
\]
By definition of the $D_i$ and Lemma \ref{gaussiancomparison}, we have
\begin{align*}
    &\sum_{i\in \mathcal{I}}\E{|D_i(\tfrac{1}{2}+i\tau)^2|}
     \\  &\leq C\sum_{\mathbf{u}\in \mathcal{I}_k(w)} e^{-u_0^2/16}\cdot \P{\mathcal{W}_r\in [u_r,u_r+\Delta_r^{-1}],\mathcal{N}_j\in[u_j,u_j+\Delta_j^{-1}],\,\forall r<j\leq k},
\end{align*}
for $C>0$ an absolute constant. By definition of $\mathcal{I}_k(w)$, this is
\begin{align*}
    \leq C\ e^{-u_0^2/16}\cdot
    \P{\mathcal{W}_j\leq m(j)+U_y(j)+2 \text{ for all $1\leq j\leq k$ and }\mathcal{W}_k\in[w-2,w+2]}.
\end{align*}
It follows that for any $|h|\leq 2\log^\theta T$,
\begin{align*}
        &\E{\max_{|h|\leq \log^\theta T}|\mathcal{Q}(\tfrac{1}{2}+i\tau+ih)|^2\mathbf{1}\left(h\in B_{\ell}^{(k)}\cap C_{\ell}^{(k)}\cap E^\theta(u_0),\, \mathscr{S}_k(h)\in(w,w+1]\right)} \\
        &\ll  e^{-u_0^2/16}\cdot \E{|\mathcal{Q}(\tfrac{1}{2}+i\tau)|^2}(e^{-t_\theta}\log N+e^k((t-t_\theta)-k)^{800}) \\
        &\quad\quad\times \P{\mathcal{W}_j\leq m(j)+U_y(j)+2 \text{ for all $1\leq j\leq k$ and }\mathcal{W}_k\in[w-2,w+2]}.
\end{align*}
{
After decomposing the rightmost factor into
\[
    \sum_{i\in\{-1,0,1,2\}} \P{\mathcal{W}_j\leq m(j)+U_{y}(j)+2,\  \forall 1\leq j\leq k \text{ and } \mathcal{W}_k\in (w+i-1, w+i]},
\]
we can bound each term in this sum using the ballot estimate (Proposition \ref{prop:ballot}) (with $y$ replaced by $y+2$) to get
\[
     \ll \frac{y(U_y(k)-w+m(k)+2)e^{-w^2/k}}{k^{3/2}}.
\]
This proves Lemma \ref{lem3}.}

\subsection{Proof of Lemma \ref{lem4}} The argument follows that of the previous section, with the addition of a twisted fourth moment estimate for the zeta function. 

For any $\ell\geq 0$, we call a Dirichlet polynomial $\mathcal{Q}$ \textit{degree-$k$ well-factorable} if it can be expressed as $\mathcal{Q}_\ell^{(k)}(s)\prod_{0\leq \lambda \leq \ell}\mathcal{Q}_\lambda (s)$, where 
\[
    \mathcal{Q}_{\lambda}(s)=\sum_{\substack{p|m\implies \log p \in (e^{t_{\lambda-1}+t_\theta}, e^{t_\lambda+t_\theta}] \\ \Omega_\lambda(m)\leq 10(t_\lambda -t_{\lambda-1})^{10^4}}}\frac{\eta(m)}{m^s}, \text{ and } \mathcal{Q}_{\ell}^{(k)}(s)=\sum_{\substack{p|m\implies \log p \in (e^{t_\ell+t_\theta}, e^{k+t_\theta}] \\ \Omega_\ell(m)\leq 10(t_{\ell+1} -t_{\ell})^{10^4}}}\frac{\eta(m)}{m^s}
\]
for arbitrary coefficients $\eta$ satisfying $|\eta(m)|\leq \exp(e^{t}/500)$ for every $m\geq 1$. 
With this definition, we have the following lemma.
\begin{lemma}[Lemma 9 in \cite{FHK1}]\label{twisted} Let $\ell\geq 0$ satisfy (\ref{finalincrement}), and $k\in [t_\ell,t_{\ell+1}]$. Let $\mathcal{Q}$ be a degree-k well-factorable Dirichlet polynomial. Then for any $|h|\leq 2$ and large enough $t$,
\[
    \E{|(\zeta_\tau\mathcal{M}_{-1}\cdots\mathcal{M}_{\ell}\mathcal{M}_{\ell}^{(k)})(h)|^4\cdot |\mathcal{Q}(\tfrac{1}{2}+i\tau+ih)|^2} \ll e^{4(t-t_\theta-k)}\E{|\mathcal{Q}(\tfrac{1}{2}+i\tau+ih)|^2}. 
\]
\end{lemma}
\begin{remark}
    The factor $e^{4(t-t_\theta-k)}$ differs from the one in \cite{FHK1} since the product of mollifiers here ranges over primes up to $\exp(e^{t_\theta+k})$ as opposed to $\exp(e^{k})$. The lemma as stated in \cite{FHK1} is recovered at $\theta=0$.
\end{remark}
The intuition behind this result is that $\zeta_\tau\mathcal{M}_{-1}\cdots\mathcal{M}_{\ell}\mathcal{M}_{\ell}^{(k)}$ should be supported on primes larger than $k$, and is therefore expected to decouple from $\mathcal{Q}$.
The exponential term $e^{4((t-t_\theta)-k)}$ then corresponds to $\mathbb{E}\big\{|(\zeta_\tau\mathcal{M}_{-1}\cdots\mathcal{M}_{\ell}\mathcal{M}_{\ell}^{(k)})(h)|^4\big\}$. Note that it is also the moment generating function of a Gaussian random variable of mean $0$ and variance $\frac{1}{2}(t-t_\theta-k)$, which agrees with the random walk heuristic for $\log|\zeta_\tau \mathcal{M}_{-1}\cdots\mathcal{M}_\ell^{(k)}|\approx \mathscr{S}_{t-t_\theta}-\mathscr{S}_{k}$.

Arguing as in the previous section, we have 
\begin{align*}
    &\E{|(\zeta_\tau\mathcal{M}_{-1}\cdots\mathcal{M}_{\ell}^{(k)}(h)|^4\cdot |\mathcal{Q}_\ell^{(k)}(h)|^2\cdot\mathbf{1}\left(h\in B_\ell\cap C_\ell\cap E^\theta(u),\, \mathscr{S}_{t_\ell}\in(v,v+1]\right)} \\
    &\ll \sum_{\mathbf{u}\in \mathcal{I}_{t_\ell}(v)} \mathbb{E}\bigg\{|\zeta_\tau \mathcal{M}_{-1}\cdots \mathcal{M}_\ell^{(k)}(h)|^4 \cdot |\mathcal{Q}_{\ell}^{(k)}(h)|^2 \\
    &\hspace{3cm}\cdot \frac{|S_{t_\theta}(0)-S_{t_\theta}(h)|^{2q}}{(|u_0|-1)^{2q}}\prod_j |\mathcal{D}_{\Delta_j, A}(\mathscr{Y}_k(h)-u_j)|^2\bigg\},
\end{align*}
where we picked $A=20$ and $q:=\lfloor(|u_0|-1)^2/8\rfloor$. Then for every $\mathbf{u}$ in the sum, the polynomial
\[
    \mathcal{Q}_{\ell}^{(k)}(h){(S_{t_\theta}(0)-S_{t_\theta}(h))^q}\prod_j \mathcal{D}_{\Delta_j, A}(\mathscr{Y}_k(h)-u_j)
\]
is well-factorable and has length $<\exp(e^t/100)$ by the definitions of $\mathcal{Q}_{\ell}^{(k)}$ and $\mathcal{D}_{\Delta_j, A}$. It thus follows from Lemma \ref{twisted} that 
\begin{align*}
    &\sum_{\mathbf{u}\in \mathcal{I}_{t_\ell}(v)} \mathbb{E}\bigg\{|\zeta_\tau \mathcal{M}_{-1}\cdots \mathcal{M}_\ell^{(k)}(h)|^4 \cdot |\mathcal{Q}_{\ell}^{(k)}(h)|^2 \\
    &\hspace{3cm}\cdot \frac{|S_{t_\theta}(0)-S_{t_\theta}(h)|^{2q}}{(|u_0|-1)^{2q}}\prod_j |\mathcal{D}_{\Delta_j, A}(\mathscr{Y}_k(h)-u_j)|^2\bigg\} \\
    &\ll e^{4((t-t_\theta)-t_{\ell+1})} \sum_{\mathbf{u}\in \mathcal{I}_{t_\ell}(v)}\mathbb{E}\bigg\{|\mathcal{Q}_{\ell}^{(k)}(h)|^2 \cdot \frac{|S_{t_\theta}(0)-S_{t_\theta}(h)|^{2q}}{(|u_0|-1)^{2q}}
    \\
    &\hspace{7cm}\times\prod_j |\mathcal{D}_{\Delta_j, A}(\mathscr{Y}_k(h)-u_j)|^2\bigg\}.
\end{align*}
The lemma then follows by proceeding as in the previous section starting from Equation (\ref{midpointlem3}).

\section{Constrained large deviation estimates}\label{sec:LD}

We now prove Theorem \ref{conditionLD} by combining the arguments of the previous two sections with those in \cite{arguinbailey1}, which adapt the recursive scheme to prove large deviation estimates for $\log|\zeta(1/2+i\tau)|$.

Here, the situation is considerably simpler than in Section \ref{upperboundsection}. First, since Theorem \ref{conditionLD} is an estimate at a single point $h$, no (spatial) discretization argument is needed. Second, unlike in Theorem \ref{auxUB}, the primes $p<\exp(e^{t_\theta})$ and ``error'' $S_{t_\theta}(h)-S_{t_\theta}(0)$ play essentially no role. 
Throughout this section, we take $h=0$, as the proof holds mutatis mutandis for general $|h|\leq 2$.

\bigskip

Since we are working at a single point, the recursive scheme employed here replaces the {good sets} of Section \ref{upperboundsection} by good {\it events}, also denoted by $G_\ell$. These restrict the behaviour of the sum $\mathscr{S}_k(0)$ at a sequence of sparse checkpoints $k=t_\ell$, where
$$t_\ell = (t-t_\theta)-\mathfrak{s}\log_\ell (t-t_\theta)$$ for $\ell>0$, and $\mathfrak{s}=2\cdot 10^6$. (This is the same definition as in \eqref{eq:timescales}, but with a larger constant $\mathfrak{s}$.)

{Our initial step, Proposition \ref{baseLD} below, goes up to time $t_1$. Accordingly, we want $\mathcal{M}_1$ to include primes up to $\exp(e^{t_1+t_\theta})$; we keep the same definition of $\mathcal{M}_\ell$ as in Equation \eqref{def:mollifiers}, but adopt the convention $t_0=1000-t_\theta$ to get
\[
    \mathcal{M}_1 (0) = \sum_{{\substack{p|m\Rightarrow 1000\leq \log \log p \leq t_1+t_\theta\\ \Omega_1(m)\leq (t_1+t_\theta)^{10^5}}}}\frac{\mu(m)}{m^{1/2+i\tau}}.
\]} 
Moreover, $\mathscr{S}_{t_1}-\mathscr{S}_{t_0}=S_{t_1+t_\theta}$. For simplicity, we will write
$S_k, \mathscr{S}_k, \zeta_\tau$ and $\mathcal{M}_\ell$ to denote each of these (random) functions evaluated at $0$.

For $\ell\geq 1$ and $\kappa=V/(t-t_\theta)$, we introduce the events
\begin{align*}
    &A_{\ell}=A_{\ell-1}\cap \{|\widetilde{\mathscr{S}}_{t_\ell}-\widetilde{\mathscr{S}}_{t_{\ell-1}}|<10^3(t_{\ell}-t_{\ell-1})\}\\
    &B_{\ell}=B_{\ell-1}\cap \{\mathscr{S}_{t_\ell}\leq \kappa t_\ell+\mathfrak{s}_0\log_{\ell}(t-t_\theta)\},\\
    &C_{\ell}=C_{\ell-1}\cap \{\mathscr{S}_{t_\ell}> \kappa t_\ell-\mathfrak{s}_0\log_{\ell}(t-t_\theta)\}\\
    &D_{\ell}=D_{\ell-1}\cap \{|\zeta_\tau e^{-S_{t_\ell+t_\theta}}|\leq c_{\ell}|\zeta_\tau \mathcal{M}_{1}\cdots \mathcal{M}_{\ell}|+e^{-10^4(t_{\ell}-t_{\ell-1})}\}
\end{align*}
where {$c_\ell=\prod_{i=1}^\ell (1+e^{-(t_{i-1}+t_\theta)})$}, and $\mathfrak{s}_0=(3/4)\mathfrak{s}$. At $\ell=0$, we take
\begin{align*}
    A_0=B_0=C_0=D_0=[T,2T] \text{ (the sample space of $\tau$)}.
\end{align*}
Our good events are then defined as
$$
G_\ell = A_\ell\cap B_\ell\cap C_\ell \cap D_\ell, \quad \ell \geq 0. 
$$
{The choice of $\mathfrak{s}_0$ in the definition of $B_\ell$ and $C_\ell$ ensures that certain exponents are small enough in the proofs below, namely those in Equation \eqref{eq:BCconstraint1}, and \eqref{eq:BCconstraint}--\eqref{eq:BCconstraintf}. There is some flexibility in this choice; see \cite[Equations (19)--(27)]{arguinbailey1} for more detail.
}

\bigskip
Letting 
\[
    G_{A} = \{{\mathscr{S}_k\in[L_A(k), U_A(k)] \,\forall k\leq t_*}\}, \quad {t_*:=(t-t_\theta)/2}.
\]
and
\[
    H^\theta = \{\log |\zeta_\tau e^{-S_{t_\theta}}|>V\},
\]
 we proceed similarly to how we did in Equation \eqref{recursiveschemedecomposition} and partition $H^\theta \cap G_A$ into
 \[
    H^\theta \cap G_A = (H^\theta \cap G_A\cap G_1^c)\cup \bigcup_{\ell=1}^{\mathcal{L}-1}\, (H^\theta \cap G_A\cap G_\ell\setminus G_{\ell+1}) \,\cup (H^\theta \cap G_A\cap G_\mathcal{L}),
\]
for $\mathcal{L}$ defined in Equation \eqref{finalincrement}. The theorem then follows by a union bound and the following three propositions. 
\begin{proposition}\label{baseLD}
    Let $\theta\in (-1,0]$. Then there exists  $\delta>0$ such that uniformly in large enough $t$ and $t_*-t_*^{2/3}\leq V/2\leq U_A(t_*)$, 
    \[
        \mathbb{P}\big\{H^\theta\cap G_A \cap G_1^c\big\} \ll \frac{A\big(U_A({t_*})-{V/2}+\sqrt{t_*}\big)}{{t_*}}\frac{e^{-V^2/(t-t_\theta)}}{\sqrt{t-t_\theta}}(t-t_\theta)^{-\delta}.
    \]
\end{proposition}
\begin{proposition}\label{indLD} Let $\theta\in (-1,0]$ and $1\leq \ell \leq \mathcal{L}-1$. Then there exists  $\delta>0$ such that uniformly in $\ell$, large enough $t$ and $t_*-t_*^{2/3}\leq V/2\leq U_A(t_*)$, 
\[
    \mathbb{P}\big\{H^\theta\cap G_A\cap G_\ell\setminus G_{\ell+1}\big\} \ll \frac{A\big(U_A({t_*})-{V/2}+\sqrt{t_*}\big)}{{t_*}}\frac{e^{-V^2/(t-t_\theta)}}{\sqrt{t-t_\theta}}\big(\log_\ell (t-t_\theta)\big)^{-\delta}.
\]
\end{proposition}
\begin{proposition}\label{endLD}Let $\theta\in (-1,0]$. Then uniformly in large enough $t$ and $t_*-t_*^{2/3}\leq V/2\leq U_A(t_*)$, 
\[
    \mathbb{P}\big\{H^\theta\cap G_A \cap G_\mathcal{L}\big\}\ll \frac{A\big(U_A({t_*})-{V/2}+\sqrt{t_*}\big)}{{t_*}}\frac{e^{-V^2/(t-t_\theta)}}{\sqrt{t-t_\theta}}.
\]
\end{proposition}
{
 \begin{proof}[Proof of Theorem \ref{conditionLD}] Using the decomposition
 \begin{align*}
         &\mathbb{P}\big\{H^\theta \cap G_A\big\} \\
         &= \mathbb{P}\big\{H^\theta\cap G_A \cap G_1^c\big\}+\sum_{\ell=1}^{\mathcal{L}-1}\mathbb{P}\big\{H^\theta\cap G_A\cap G_\ell\setminus G_{\ell+1}\big\}+\mathbb{P}\big\{H^\theta\cap G_A\cap G_{\mathcal{L}}\big\},
 \end{align*}
 the theorem follows from the propositions above, noting that the factors $(\log_\ell (t-t_\theta))^{-\delta}$ are summable (see \eqref{eqn: summable}).
 \end{proof}
 }
A few heuristics are in order prior to proving these results. As in Section \ref{upperboundsection}, we think of $\log|\zeta_\tau e^{-S_{t_\theta}}|$ as the partial sum $\mathscr{S}_{t-t_\theta}$, which can itself be thought of as a Gaussian random walk of length $t-t_\theta$. The probability that such a walk's endpoint is in $[V,V+1]$ should roughly equal $e^{-V^2/(t-t_\theta)}/\sqrt{t-t_\theta}$.

The additional factor to the left of this Gaussian term in Propositions \ref{baseLD}, \ref{indLD} and \ref{endLD} above arises from the event $G_A$, which restricts $\mathscr{S}_k$ to lie under a barrier (starting at $A>0$) up to time {$t_*$}. Were ${\mathscr{S}_{k}}$ Gaussian, the probability of this happening and $\mathscr{S}_{t_*}\approx u$ is
\[
   \ll \frac{A\big(U_A({t_*})-{u}+1\big)}{{t_*}^{3/2}}e^{-u^2/{t_*}},
\]
by {the ballot theorem in Proposition \ref{prop:ballot}}. Being supported on different primes, $\mathscr{S}_{{t_*}}$ and $\mathscr{S}_{t-t_\theta}-\mathscr{S}_{{t_*}}$ should behave independently, and if viewed as Gaussian increments, the probability of $H^\theta \cap G_A$ should therefore be 
\[
    \ll \int_{u< U_A(t_*)} {\frac{A\big(U_A({t_*})-u+1\big)}{{t_*}}\frac{e^{-u^2/{t_*}}}{\sqrt{{t_*}}} \frac{e^{-(V-u)^2/({t_*})}}{\sqrt{{t_*}}}\mathrm{d}u. }
\]
This integral can then be estimated directly. By completing the square and using the substitution $v=u-V/2$, this is
\begin{align}\label{eq:partialbarrier}
        &=\frac{e^{-V^2/(t-t_\theta)}}{\sqrt{t_*}}\frac{A}{t_*} \int^{U_A(t_*)-V/2}_{-\infty}  \frac{e^{-2v^2/t_*}}{\sqrt{t_*}}\big(U_A(t_*)-V/2-v+1\big)\mathrm{d}v\nonumber\\
&\leq \frac{e^{-V^2/(t-t_\theta)}}{\sqrt{t_*}}\frac{A}{t_*}  \Big(2\big(U_A(t_*)-V/2+1\big)+\int_{0}^{\infty}\frac{ve^{-2v^2/t_*}}{\sqrt{t_*}}\mathrm{d}v\Big) \\
&\leq \frac{e^{-V^2/(t-t_\theta)}}{\sqrt{t_*}}\frac{A}{t_*}  \Big(2\big(U_A(t_*)-V/2+1\big)+\mathbb{E}\big\{|X|\big\}\Big)\nonumber
\end{align}
where $X$ is a Gaussian random variable with variance $t_*/4$. Since $\mathbb{E}\big\{|X|\big\}\leq \sqrt{t_*}\asymp \sqrt{t-t_\theta}$, we recover the right-hand side in Theorem \ref{conditionLD}. The remainder of this section is devoted to making this argument rigorous. 



\subsection{Proof of Proposition~\ref{baseLD}}
\label{sec6.1}

A union bound gives
\begin{align*}
&\mathbb{P} \bigl\{H^\theta\cap G_{1}^{c}\cap
G_{A} \bigr\}
\\
&\quad \ll \mathbb{P} \bigl\{A_{1}^{c} \bigr\}+\mathbb{P}
\bigl\{A_{1} \cap D_{1}^{c} \bigr\}+
\mathbb{P} \bigl\{B_{1}^{c}\cap G_{A}
\bigr\}+\mathbb{P} \bigl\{H^\theta\cap C_{1}^{c}\cap
A_{1}\cap D_{1}\cap G_{A} \bigr\},
\end{align*}
We show that each of these terms satisfies the bound in \ref{baseLD}. To begin with, the Gaussian tail bound in Corollary~\ref{cor:tails} gives
$\mathbb{P}  \{A_{1}^{c}  \}\ll \sqrt{t_{1}-t_0}\cdot \exp (-
10^6(t_{1}-t_0))$, which is clearly of the desired order for some
$\delta >0$.

 To bound $\mathbb{P}  \{A_{1}\cap D_{1}^{c}  \}$, we first note that
$A_{1}\cap \{|\zeta _{\tau}| \leq e^{2t}\}\subset A_{1}\cap D_{1}$, since
Lemma~\ref{mollifying} (whose assumptions are satisfied on
$A_{1}$) implies that
\begin{equation*}
\begin{split}\bigl| \zeta _{\tau }e^{-(\mathscr{S}_{t_1}-\mathscr{S}_{t_0})} \bigr| &\leq
c_{1} |\zeta _{
\tau }\mathcal{M}_{1}
| +e^{2t-10^{5}(t_{1}-t_{0})}
\\
&\leq c_{1} | \zeta
_{
\tau }\mathcal{M}_{1} | +e^{-{10^4}(t_{1}-t_{0})}.
\end{split}\end{equation*}
It follows that
\begin{equation*}
\mathbb{P} \bigl\{A_{1}\cap D_{1}^{c}
\bigr\} \leq \mathbb{P} \bigl\{ | \zeta _{\tau } |
>e^{2t} \bigr\} \leq e^{-4t}\mathbb{E} \bigl\{ |
\zeta _{\tau } | ^{2} \bigr\},
\end{equation*}
which $\leq e^{-3t}$ by Lemma~\ref{secondmomentzeta}, and thus small enough.

The bound on $\mathbb{P}\{B_1^c\cap G_A\}$ follows from Markov's inequality and Equation \eqref{eq:gaussianmoment}, as they together give
\begin{equation}\label{eq:BCconstraint1}
    \mathbb{P}\{B_1^c\}\ll \frac{\mathbb{E}\big\{|\mathscr{S}_{t_1}|^{2q}\big\}}{(\kappa t_1+\mathfrak{s}_0\log(t-t_\theta))^{2q}} \ll \frac{e^{-V^2/(t-t_\theta)}}{\sqrt{t-t_\theta}} (t-t_\theta)^{1+\kappa^2 \mathfrak{s}-2\kappa \mathfrak{s}_0},
\end{equation}
and $1+\kappa^2\mathfrak{s}-2\kappa\mathfrak{s}_0=1-10^6+o(1)\leq -10^5$.



Lastly, to estimate
$\mathbb{P}  \{H^\theta\cap C_{1}^{c}\cap A_{1}\cap D_{1}\cap G_{A}
  \}$, we partition according to the value of
$\mathscr{S}_{t_{1}}$ (in $A_{1}\cap C_{1}^{c}$), to get the bound
\begin{equation*}
\leq \sum_{-10^3(t_{1}-t_{0})<u<\kappa t_{1}-\mathfrak{s}_0\log t_{1}} \mathbb{P} \Bigl\{ \bigl\{
\mathscr{S}_{t_{1}}\in (u,u+1], \bigl| \zeta _{
\tau }e^{-S_{t_{1}+t_{\theta }}}
\bigr| >e^{V-u-1} \bigr\} \cap D_{1}\Bigr\}.
\end{equation*}
A short proof by contradiction (see Section~\ref{lowerextsection}) shows that if
$|\zeta _{\tau }e^{-S_{t_{1}+t_{\theta}}}|>e^{V-u-1}$ and
$h\in C_{1}^{c}\cap A_{1}\cap D_{1}$, then 
$|\zeta _{\tau }\mathcal{M}_{1}|>\frac {1}{100}e^{V-u}$. Markov's inequality then implies that the previous sum is
%
\begin{align}
\label{sumCbound}
\begin{split} &\leq \sum
_{u} \mathbb{P} \Bigl\{ \big\{| \zeta _{\tau }
\mathcal{M}_{1} | > \frac {1}{100}e^{V-u}\big\}
\cap \bigl\{\mathscr{S}_{t_{1}}\in (u,u+1]\bigr\}\Bigr\}
\\
&\ll \sum_{u} e^{-4(V-u)} \mathbb{E}
\Bigl\{ | \zeta _{\tau } \mathcal{M}_{1} |
^{4}\mathbf{1} \bigl\{\mathscr{S}_{t_{1}}\in (u,u+1]\bigr\}\Bigr\}. \end{split} 
\end{align}
The sum over $u<0$ is $\ll te^{-4V}\ll e^{-3t_*}$ and can thus safely be ignored. For the remaining range, note that it follows
directly from the proof of Lemma~\ref{lem4} that
\begin{align*}
\mathbb{E} \bigl\{ | \zeta _{\tau }\mathcal{M}_{1}
| ^{4}\mathbf{1} ({\mathscr{S}_{t_{1}}\in
\bigl(u,u+1]} \bigr) \bigr\}
\quad \ll e^{4(t-t_{\theta}-t_{1})} \frac{e^{-u^{2}/{t_1}}}{\sqrt{{t_1}}}
\end{align*}
for every $0\leq u\leq 2t_1$. Applying this to the remaining terms in the summation
in (\ref{sumCbound}), we find that it is
\begin{equation*}
\ll \sum_{0
\leq u<\kappa t_{1}-\mathfrak{s}_0\log t_{1}}
e^{4((t-t_{\theta})-t_{1})}e^{-4(V-u)} \frac{e^{-u^{2}/t_{1}}}{\sqrt{t_{1}}}.
\end{equation*}
Using the change of variable $w=\kappa t_{1}-u$, the summation above is
bounded by
%
\begin{align}
\begin{split}\label{eq:Cconstraint1} &\frac{e^{-\kappa ^{2} t_{1}}}{\sqrt{t_{1}}}e^{4(1-\kappa )((t-t_{
\theta})-t_{1})} \sum
_{w>\mathfrak{s}_0 \log (t-t_{\theta})} e^{-(4-2
\kappa )w}
\\
&\quad \ll \frac{e^{-V^{2}/(t-t_{\theta})}}{\sqrt{t-t_{\theta}}} (t-t_{
\theta})^{\mathfrak{s}\kappa^{2}+\mathfrak{s}(4-4\kappa )-2(2-\kappa )\mathfrak{s}_0}.
\end{split}
\end{align}
This yields the claim since the exponent of $(t-t_\theta)$ is bounded by $-10^5$.

\subsection{Proof of Propositions \ref{indLD} and \ref{endLD}}
These propositions follow from the same argument used to prove Propositions 2.2 and 2.3 in \cite{arguinbailey1}, suitably modified to include the event $G_A$. This modification simply amounts to replacing their Lemmas 2.6 and 2.7 with Lemmas \ref{lem:6.5} and \ref{lem:6.6} below, {respectively}, which play the role of Lemmas \ref{lem3} and \ref{lem4} in Section \ref{upperboundsection}. 

\begin{lemma}\label{lem:6.5}Let $\ell \geq 1$ satisfy (\ref{finalincrement}). Let $\mathcal{Q}$ be a Dirichlet polynomial of length $N \leq \exp(\tfrac{1}{100}e^t)$ supported on integers all of whose prime factors are greater than $\exp(e^{t_\theta+t_\ell})$. Then for $|w-\kappa t_\ell|\leq \mathfrak{s}_0\log_\ell (t-t_\theta)$
and large enough $t$, we have
\begin{align*}
    \mathbb{E}\big\{|\mathcal{Q}(\tfrac{1}{2}+i\tau)|^2\mathbf{1}&\big(B_\ell\cap C_\ell\cap {G_A}\,\cap\{\mathscr{S}_{t_\ell}\in (w,w+1]\}\big)\big\} \\
    &\ll \E{|\mathcal{Q}(\tfrac{1}{2}+i\tau)|^2}\cdot \frac{A\big(U_A(t_*)-{w(t_*/t_\ell)}+\sqrt{t_*})}{{t_*}}\frac{{e^{-w^2/t_{\ell}}}}{\sqrt{t_\ell}}.
\end{align*}
\end{lemma}

\begin{proof}
    Partitioning according to the value of $\mathscr{S}_{t_*}$, the right-hand side in the lemma above is bounded by 
    \begin{align*}
        \sum_{u\in [{L_A(t_*),U_A(t_*)}]\cap \mathbb{Z}} \mathbb{E}\big\{|\mathcal{Q}(\tfrac{1}{2}+i\tau)|^2\mathbf{1}\big(B_\ell \,\cap\, &C_\ell\,\cap\{\mathscr{S}_{t_\ell}-\mathscr{S}_{{t_*}}\in (w-u,w-u+1]\big)\, \\
        &\times\mathbf{1}\big(\{\mathscr{S}_{{t_*}}\in [u,u+1]\}\, {\cap\, G_A}\big)\big\}.
    \end{align*}
    Each term in the summation above can then be bounded as in the proof of Lemma \ref{lem3} (see also Lemma 2.6 in \cite{arguinbailey1}). Namely, one expresses the indicators in the expectation as a sum over tuples $(u_j)_j$ of products of the form $\prod_{\mathscr{Y}_j} \mathbf{1}(\mathscr{Y}_j\in[u_j,u_j+\Delta_j^{-1}])$, where $\mathscr{Y}_j$ ranges over $$\mathscr{Y}_j\in\{\mathscr{S}_{r}\}\cup\{\mathscr{S}_{j}-\mathscr{S}_{j-1}\}_{r<j\leq t_*}\cup\{\mathscr{S}_{t_1}-\mathscr{S}_{t_*}\}\cup\{\mathscr{S}_{t_k}-\mathscr{S}_{t_{k-1}}\}_{1<k\leq \ell},$$ for $r=\lceil A/4\rceil$ and a suitable choice of mesh size $\Delta_j$.
    After approximating each of these indicators by a Dirichlet polynomial $\mathcal{D}_{\Delta_j,A}$ (cf.\ Equation \eqref{eqn: D}), one uses Lemma \ref{splittinglemma} to decouple the second moment of their product from $\mathbb{E}\{|\mathcal{Q}(1/2+i\tau)|^2\}$, and Lemma \ref{MVMV} to bound said second moment by $\prod_j \mathbb{E}\{|\mathcal{D}_{\Delta_j,A}(\mathcal{Y}_j-u_j)|^2\}$, where $\mathcal{Y}_j$ is a Steinhaus model for $\mathscr{Y}_j$ (cf.\ Equations \eqref{eq:Steinhaus} and \eqref{eq:Steinhaus2}). Finally, undoing the Dirichlet polynomial approximation yields a product of probabilities, which are then bounded by a Gaussian counterpart using the Berry-Esseen theorem\footnote{In the $\theta=0$ case, one uses the estimate in \cite[Lemma 18]{FHK1} for the probabilities involving the first increment $\mathscr{S}_r=S_r$ rather than the Berry-Esseen bound.}. Altogether, this procedure yields the bound
    \[
       \ll \sum_{u\in [{L_A(t_*),U_A(t_*)}]\cap \mathbb{Z}} \frac{A(U_A({t_*})-{u}+1)}{{t_*}}\frac{e^{-u^2/{t_*}}}{\sqrt{{t_*}}}\frac{e^{-(w-u)^2/(t_\ell-{t_*})}}{\sqrt{t_\ell-{t_*}}},
    \]
    which by the change of variables $u={w(t_*/t_\ell)+v}$ is
    \begin{align*}
        &\ll \sum_{\substack{v+w(t_*/t_\ell)\\\in \mathbb{Z}\cap {[L_A(t_*),U_A(t_*)]}}}
        \frac{A({U_A(t_*)-w(t_*/t_\ell)-v+1})}{{t_*}}\frac{e^{-({w(t_*/t_\ell)+v})^2/{t_*}}}{\sqrt{{t_*}}}\frac{e^{-\tfrac{({w-w(t_*/t_\ell)-v})^2}{{t_\ell-t_*}}}}{\sqrt{{t_\ell-t_*}}}\\
        &= A\frac{e^{-w^2/t_\ell}}{t_*^{3/2}}\sum_{\substack{v+w(t_*/t_\ell)\\\in \mathbb{Z}\cap {[L_A(t_*),U_A(t_*)]}}} ({U_A(t_*)-w(t_*/t_\ell)-v+1})\frac{e^{{-v^2(t_\ell/t_*(t_\ell-t_*))}}}{\sqrt{{t_\ell-t_*}}}.
    \end{align*}
    {This sum can then be estimated by comparison with the corresponding Gaussian integral. In this range of $w$, this is $\ll {U_A(t_*)-w(t_*/t_\ell)}+\sqrt{t_*}$ by completing the square as in \eqref{eq:partialbarrier}.}
\end{proof}
\noindent Combining the same argument with Lemma \ref{twisted} also yields the following.
\begin{lemma}\label{lem:6.6}
    Let $\ell\geq 1$ satisfy (\ref{finalincrement}). Then for $|w-\kappa t_\ell|\leq \mathfrak{s}_0\log_\ell (t-t_\theta)$
and large enough $t$, we have
\begin{align*}
    \mathbb{E}\big\{|\zeta_\tau \mathcal{M}_{1}\cdots \mathcal{M}_{\ell}|^4&\mathbf{1}\big({G_A}\cap B_\ell\cap C_\ell, \mathscr{S}_{t_\ell}\in[w,w+1]\big)\big\} \\&\ll e^{4((t-t_\theta)-t_\ell)}\cdot \frac{A\big({U_A(t_*)-w(t_*/t_\ell)+\sqrt{t_*}}\big)}{{t_*}}\frac{{e^{-w^2/t_{\ell}}}}{\sqrt{t_\ell}}.
\end{align*}
\end{lemma}

\begin{proof}[Proof of Proposition \ref{indLD}]{By the proof of Proposition 2 in \cite{arguinbailey1}, replacing their Lemmas 2.6 and 2.7 with our Lemmas \ref{lem:6.5} and \ref{lem:6.6}, we get the following estimates:
\begin{align}\label{eq:BCconstraint}
    &\P{G_A\cap A_{\ell+1}^c\cap G_\ell}\ll \frac{A\big({U_A(t_*)-V/2+\sqrt{t_*}}\big)}{{t_*}}\frac{e^{-V^2/(t-t_\theta)}}{\sqrt{t-t_\theta}}(\log_\ell(t-t_\theta))^{\kappa^2\mathfrak{s}-10^6\mathfrak{s}+2\kappa\mathfrak{s}_0},\\
     &\P{ G_A\cap B_{\ell+1}^c\cap G_\ell}\ll \frac{A\big({U_A(t_*)-V/2+\sqrt{t_*}}\big)}{{t_*}}\frac{e^{-V^2/(t-t_\theta)}}{\sqrt{t-t_\theta}}(\log_\ell(t-t_\theta))^{1/2+\kappa^2\mathfrak{s}-2\kappa\mathfrak{s}_0},
\end{align}
\begin{align}
    &\mathbb{P}\,\{H^\theta \cap  G_A\cap  C^{c}_{\ell+1}\cap A_{\ell+1}\cap D_{\ell+1}\cap G_\ell\}\\
    &\hspace{20px}\ll \frac{A\big({U_A(t_*)-V/2+\sqrt{t_*}}\big)}{{t_*}}\frac{e^{-V^2/(t-t_\theta)}}{\sqrt{t-t_\theta}}(\log_\ell(t-t_\theta))^{\mathfrak{s}(2-\kappa)^2-2(2-\kappa)\mathfrak{s}_0+1/2},\nonumber
\end{align}
and finally
\begin{align}
          &\P{ G_A\cap D_{\ell+1}^c\cap A_{\ell+1}\cap G_\ell} \label{eq:BCconstraintf}\\
          &\hspace{20px}\ll \frac{A\big({U_A(t_*)-V/2+\sqrt{t_*}}\big)}{{t_*}}\frac{e^{-V^2/(t-t_\theta)}}{\sqrt{t-t_\theta}}(\log_{\ell-1}(t-t_\theta))^{\kappa^2 \mathfrak{s}+2\kappa\mathfrak{s}_0-4\mathfrak{s}(10^3-1)} \nonumber
\end{align}
for any $\ell\geq 1$.
In each of these, the exponent of the rightmost term is negative upon substituting $\mathfrak{s}=2\cdot 10^6$ and $\mathfrak{s}_0=(3/4)10^6$. They therefore imply Proposition \ref{indLD} by the union bound
\begin{align*}
    &\P{H^\theta \cap G_A \cap(G_\ell\setminus G_{\ell+1})}\leq \P{G_A\cap A_{\ell+1}^c\cap G_\ell}+\P{ G_A\cap B_{\ell+1}^c\cap G_\ell} \\
    &\hspace{40px}+\mathbb{P}\,\{H^\theta \cap  G_A\cap  C^{c}_{\ell+1}\cap A_{\ell+1}\cap D_{\ell+1}\cap G_\ell\}+\P{ G_A\cap D_{\ell+1}^c\cap A_{\ell+1}\cap G_\ell}.
\end{align*}}
\end{proof}
\begin{proof}[Proof of Proposition \ref{endLD}] Since $\P{H^\theta\cap G_A\cap G_\mathcal{L}}\leq \P{G_A\cap G_\mathcal{L}}$, it suffices to show that the latter satisfies the desired bound. Partitioning on the value of $u\in \mathbb{Z}$ of $\mathscr{S}_{t_\mathcal{L}}$, we have that $\P{G_A \cap G_\mathcal{L}}$ is  
\begin{align*}
    \leq \sum_{-\mathfrak{s}_0\log_\mathcal{L}(t-t_\theta)<u-\kappa t_\mathcal{L}<\mathfrak{s}_0\log_\mathcal{L}(t-t_\theta)} \mathbb{E}\big\{\mathbf{1}\big(B_{\mathcal{L}}\cap C_{\mathcal{L}}\cap {G_A}\,\cap\{\mathscr{S}_{t_{\mathcal{L}}}\in (u,u+1]\}\big)\big\}.
\end{align*}
By Lemma \ref{lem:6.5} (with $\mathcal{Q}\equiv 1$), this is
\begin{align*}
    \ll \sum_{-\mathfrak{s}_0\log_\mathcal{L}(t-t_\theta)<u-\kappa t_\mathcal{L}<\mathfrak{s}_0\log_\mathcal{L}(t-t_\theta)}  \frac{A\big({U_A(t_*)-u(t_*/t_\mathcal{L})+\sqrt{t_*}}\big)}{{t_*}}\frac{e^{-u^2/t_\mathcal{L}}}{\sqrt{t_\mathcal{L}}}.
\end{align*}
Recalling that $\log_\mathcal{L}(t-t_\theta)=O(1)$ (cf. Equation \eqref{finalincrement}), a change of variables $w=u-\kappa t_\mathcal{L}$ in the previous sum yields
\begin{align*}
&\ll \frac{e^{-\kappa^2(t-t_\theta)}}{\sqrt{t-t_\theta}}\sum_{-\mathfrak{s}_0\log_\mathcal{L}(t-t_\theta)<w<\mathfrak{s}_0\log_\mathcal{L}(t-t_\theta)} \frac{A\big({U_A(t_*)-w(t_*/t_\mathcal{L})-V/2+\sqrt{t_*}}\big)}{{t_*}}e^{-2\kappa w},
\end{align*}
and the sum is ${\ll A(U_A(t_*)-V/2+\sqrt{t_*})/t_*}$. 
\end{proof}

\begin{appendix}
\bigskip 

\section{Discretization}
{The following lemma allows us to discretize the maximum of a Dirichlet polynomial over a continuous interval to a countable mesh of points, and is essentially a consequence of the Poisson summation formula. }
\begin{lemma}[Lemma 27 in \cite{FHK1}]\label{disclemma}
    Let $\mathcal{I}$ be a finite set of indices and $D_i$ with $l\in \mathcal{I}$ be a sequence of Dirichlet polynomials of length $\leq N$. Then for any $\ell\geq 1, A\geq 100$, 
        \begin{align*}
        &\max_{|h|\leq 2\log^\theta T} \bigg(\sum_{l\in \mathcal{I}}|D_i(\tfrac{1}{2}+i\tau+ih)|^2\bigg) \\&\ll_A
        \sum_{{|j|\leq 16\log^{\theta}T \log N} }\bigg(\sum_{l\in \mathcal{I}} \Big|D_l\Big(\frac{1}{2}+i\tau+\frac{2\pi i j}{8\log N}\Big)\Big|^2 \bigg) \\&+\sum_{{|j|> 16\log^{\theta}T \log N} } \frac{1}{1+|j|^A} \bigg(\sum_{l\in \mathcal{I}} \Big|D_l\Big(\frac{1}{2}+i\tau+\frac{2\pi i j}{8\log N}\Big)\Big|^2 \bigg).
    \end{align*}
\end{lemma}
\noindent We also have the following discretization result for the maximum of the zeta function. It can be seen as a consequence of the previous lemma, since the zeta function is well-approximated by a Dirichlet polynomial of length $T$ (see, e.g., Proposition 2 in \cite{bombierifriedlander}).
\begin{lemma}[Lemma 26 in \cite{FHK1}]\label{disczeta}
    Fix $\theta\in (-1,0)$ and let $t_\theta=t|\theta|$. Let $\mathcal{T}_t^\theta$ be a set of $e^{-t-100}$ well-spaced points in $[-2e^{-t_\theta}, 2e^{-t_\theta}]$. There exists an absolute constant $C>1$ such that for any $V>1$, $A>100$
    \[
        \P{\max_{|h|\leq e^{-t_\theta}}|\zeta_\tau(h)e^{-S_{t_\theta}(0)}|>V}\leq \P{\max_{h\in \mathcal{T}_t^\theta}|\zeta_\tau(h)e^{-S_{t_\theta}(0)}|>V/C}+O_A(e^{-At}).
    \]
\end{lemma}

\section{Moments of the zeta function and of Dirichlet polynomials}
{The following moment (standard) estimate for Dirichlet polynomials follows from expanding $|\tilde{S}_k(h)-\tilde{S}_j(h)|^{2q}$  and using $\binom{q}{\alpha_1,...,\alpha_r}^2\leq q!\binom{q}{\alpha_1,...,\alpha_r}$ to bound the main contribution.}
\begin{lemma}\label{complexgaussian}Let $h\in [-2,2]$ and $\theta\in (-1,0)$ be arbitrary. Then for any positive integers $t|\theta|\leq j\leq k$ and $2q\leq e^{t-k}$, we have
\[
    \E{|\widetilde{S}_k(h)-\widetilde{S}_j(h)|^{2q}}\ll q!(k-j+1)^q.
\]
\end{lemma}
\begin{proof}
    This is a corollary of Lemma 3 in \cite{SoundMoments}.
\end{proof}
\noindent We also require analogous estimates for the real part of $\tilde{S}_j$.
\begin{lemma}[Lemma 16 in \cite{FHK1}]\label{gaussianmoments} Let $h\in [-2,2]$ and $\theta\in (-1,0)$ be arbitrary. Then for any integers $t|\theta|\leq j\leq k$, and  $2q\leq e^{t-k}$ we have
    \begin{equation}\label{eq:gaussianmoment}
                \E{|S_k(h)-S_{j}(h)|^{2q}} \ll \frac{(2q)!}{2^qq!}\left(\frac{k-j}{2}\right)^{q},
    \end{equation}
    as well as 
    \begin{equation}\label{eq:gaussianmomentshifted}
                \E{|S_k(h)-S_{j}(h)+A|^{2q}} \ll \frac{(2q)!}{q!}\left(k-j\right)^{q}+(2A)^{2q}
    \end{equation}
    for any constant $A>0$.
    Moreover, there exists a constant $C>0$ such that for any $j<k$, $2q\leq e^{t-k}$ we have
    \[  
        \E{|S_k(h)-S_{j}(h)|^{2q}} \ll \sqrt{q}\frac{(2q)!}{2^qq!}\left(\frac{k-j+C}{2}\right)^{q}.
    \]
\end{lemma}
\noindent Using Markov's inequality and Stirling's approximation, these bounds yield the following Gaussian tail estimates. 
\begin{cor}[Gaussian tail for Dirichlet polynomials]\label{cor:tails} {Let $h\in [-2,2]$ and $\theta\in (-1,0)$ be arbitrary. Then for any $t|\theta|\leq j \leq k$ and $V^2\leq (k-j)e^{t-k}/2$,
\begin{align}\label{Gaussiantail}
        \mathbb{P}\Big\{|S_k(h)-S_j(h)|>V\Big\}&\ll \exp\bigg({-\frac{V^2}{k-j}}\bigg),\\
            \P{|\widetilde{S}_k(h)-\widetilde{S}_j(h)|>V}&\ll \frac{V}{(k-j+1)^{1/2}}\exp\left(-\frac{V^2}{k-j+1}\right).
\end{align}
}
\end{cor}
\begin{proof} {By Markov's inequality and the moment estimate in Lemma \ref{gaussianmoments}, we have
\[
    \P{|{S}_k(h)-{S}_j(h)|>V}\ll\frac{\E{|{S}_k(h)-{S}_j(h)|^{2q}}}{V^{2q}} \ll \frac{1}{V^{2q}}\frac{(2q)!}{2^qq!}\bigg(\frac{k-j}{2}\bigg)^q,
\]
which by Stirling's approximation is
\begin{equation}\label{eq:stirling}
    \ll \frac{1}{V^{2q}}\bigg(\frac{2q}{e}\bigg)^q\bigg(\frac{k-j}{2}\bigg)^q.
\end{equation}
Picking $q=\lceil V^2/(k-j)\rceil$, this simplifies to $\ll e^{-q}$ which is the bound in Equation \eqref{Gaussiantail}. The second bound is proved in the same way, using the moment estimate in Lemma \ref{complexgaussian} and picking $q=\lceil V^2/(k-j+1)\rceil$.}
\end{proof}
 We record the following mean-value theorem for Dirichlet polynomials (Lemma \ref{MVMV}), due to Montgomery and Vaughan \cite{MVMV}, and an immediate corollary of it which allows us to decouple averages of Dirichlet polynomials supported on disjoint ranges of primes (Lemma \ref{splittinglemma}).
\begin{lemma}[Mean--value theorem for Dirichlet polynomials]\label{MVMV} Let $T>0$ and $\tau$ be uniformly distributed in $[T,2T]$. Then
\begin{align*}
    \mathbb{E}\bigg\{\Big|\sum_{n\leq N}a(n)n^{-i\tau}\Big|^2\bigg\} 
    &=\bigg(1+O\Big(\frac{N}{T}\Big)\bigg)\sum_{n\leq N}|a(n)|^2\\
    &= \bigg(1+O\Big(\frac{N}{T}\Big)\bigg)\mathbb{E}\bigg\{\Big|\sum_{n\leq N}a(n)Z_n\Big|^2\bigg\}.
\end{align*}
\end{lemma}
\begin{proof}
    This follows from Corollary 3 in \cite{MVMV}.
\end{proof}
\begin{lemma}[Lemma 14 in \cite{FHK1}]\label{splittinglemma} 
    Let $h\in [-2,2]$ and 
    \[
        A(s)=\sum_{\substack{n\leq N \\ p|n\implies p\leq w}} \frac{a(n)}{n^s}\text{ and } B(s)=\sum_{\substack{n\leq N \\ p|n\implies p> w}}\frac{b(n)}{n^s}
    \]
    be Dirichlet polynomials with $N\leq T^{1/4}$. Then we have
    \begin{align*}
        &\E{|A(\tfrac{1}{2}+i\tau+ih)|^2|B(\tfrac{1}{2}+i\tau+ih)|^2}\\
        &=\left(1+O(T^{-1/2})\right)\E{|A(\tfrac{1}{2}+i\tau+ih)|^2}\E{|B(\tfrac{1}{2}+i\tau+ih)|^2}.
    \end{align*}
\end{lemma}
\noindent Lastly, we use the following classical result for the second moment of $\zeta$ due to Hardy and Littlewood \cite{hardylittlewood}, see also Equation (1.79) in \cite{Ivic} for a quick proof of the upper bound.
\begin{lemma}\label{secondmomentzeta}For all $h\in [-2,2]$, we have $\E{|\zeta_\tau(h)|^2}\sim \log T$.
\end{lemma}

\section{Ballot theorem} {As in Equation \eqref{eq:Wdef}, let
\[
    \mathcal{W}_k = \sum_{1000\leq \ell\leq k}\mathcal{N}_\ell
\]
where each $\mathcal{N}_\ell$ is a centered, independent read Gaussian variable with variance $1/2$. The following \textit{ballot theorem} estimates the probability that this random walk ends around the value $w$ after $k$ steps, while having remained under a curved barrier at times $j\leq k$.
\begin{proposition}[Proposition 4 in \cite{FHK1}]\label{prop:ballot} Uniformly in $t\geq 1, 1\leq y\leq 2t, t/2\leq k\leq t$ and $m(k)+L_y(k)-4\leq w\leq m(k)+U_y(k)$ (defined in Equations \eqref{logbarrier} and \eqref{lowerbarrierl}), we have
\begin{align*}
    \P{\{\mathcal{W}_k\in (w,w+1]\}\cap_{1000<j\leq k}\{\mathcal{W}_j\leq m(j)+U_y(j)\}}\\
    \ll (y+1)(U_y(k)+m(k)-w+1)k^{-3/2}e^{-w^2/k}. 
\end{align*}
\end{proposition}}
\end{appendix}

\begin{acks}[Acknowledgments]
The authors are grateful to the referee for the comments that substantially improved the first version of the paper, and to Jon Keating and Christine Chang for insightful discussions. J. H. thanks Maksym Radziwiłł for sharing helpful resources on the subject. L.-P. A. is supported by the EPSRC grant EP/Z535990/1 and NSF award DMS 2153803, and J. H. is supported by the EPSRC Centre for Doctoral Training in Mathematics of Random Systems: Analysis, Modelling and Simulation (EP/S023925/1).
\end{acks}




\bibliographystyle{imsart-number} 
\bibliography{biblio}       

\end{document}